\documentclass{amsart}

\usepackage{amsmath}
\usepackage{amssymb}
\usepackage{verbatim}
\usepackage{dsfont}
\usepackage{enumerate}
\usepackage[all]{xy}
\usepackage[mathscr]{eucal}

\usepackage{amsthm}
\usepackage{stmaryrd}
\usepackage{amscd}
\usepackage{amsfonts}
\usepackage{latexsym}

\usepackage{amscd}

\usepackage{graphicx}

\newtheorem{theorem}{Theorem}[subsection]
\newtheorem{lemma}[theorem]{Lemma}
\newtheorem{proposition}[theorem]{Proposition}
\newtheorem{corollary}[theorem]{Corollary} 
\theoremstyle{definition}  
\newtheorem{definition}[theorem]{Definition}
\newtheorem{example}[theorem]{Example}

\newtheorem{remark}[theorem]{Remark}

\newcommand{\End}{\operatorname{End}}
\newcommand{\Aut}{\operatorname{End}} 
\newcommand{\Hom}{\operatorname{Hom}} 
\def\uRep{\underline{\operatorname{Re}}\!\operatorname{p}} 
\newcommand{\lift}{\operatorname{lift}}
\newcommand{\Rep}{\operatorname{Rep}}
\newcommand{\tr}{\operatorname{tr}}
\newcommand{\im}{\operatorname{im}}
\newcommand{\ch}{\operatorname{ch}}
\newcommand{\edeg}{\operatorname{edeg}}
\newcommand{\id}{\text{id}}
\newcommand{\ev}{\text{ev}}
\newcommand{\coev}{\text{coev}}
\newcommand{\inj}{\iota}
\newcommand{\proj}{\pi}
\newcommand{\parity}[1]{\bar{#1}}
\newcommand{\psspace}[2]{#1_{\parity #2}}

\newcommand{\emp}{\varnothing}
\newcommand{\biemp}{{(\varnothing,\varnothing)}}
\newcommand{\bibox}{{(\Box,\Box)}}
\newcommand{\down}{\vee}
\newcommand{\up}{\wedge}
\newcommand{\cross}{{\Large\text{$\times$}}}
\newcommand{\bigo}{\bigcirc}
\newcommand{\bh}{{\bf{\bar h}}}
\newcommand{\h}{{\bf h}}
\newcommand{\s}{{\bf s}}

\newcommand{\de}{{\sf d}}
\newcommand{\gl}{\mathfrak{gl}}
\newcommand{\Z}{\mathbb{Z}}
\newcommand{\K}{\mathbb{K}}
\newcommand{\KS}{\mathbb{K}\Sigma}
\newcommand{\cat}{\mathcal}
\newcommand{\ep}{\varepsilon}

\newcommand{\black}{\bullet}
\newcommand{\white}{\circ}
\newcommand{\arup}[1]{\stackrel{#1}{\longrightarrow}}
\newcommand{\larup}[1]{\stackrel{#1}{\longleftarrow}}
\newcommand{\onto}{\twoheadrightarrow}
\newcommand{\into}{\hookrightarrow}

\newcommand{\flip}{\sigma}
\newcommand{\braid}{\mathrm{c}}
\newcommand{\nattrans}{\Rightarrow}
\newcommand{\unit}{\mathds{1}}
\newcommand{\funcat}{\mathcal{H}om}
\newcommand{\monfuncat}{\funcat^{\otimes}}
\newcommand{\preaddfuncat}{\funcat_{+}}
\newcommand{\tensorfuncat}{\preaddfuncat^{\otimes}}
\newcommand{\strtensorfuncat}{\preaddfuncat^{\otimes\text{-str}}}
\newcommand{\add}[1]{{#1}^{\mathrm{add}}}
\newcommand{\kar}[1]{{#1}^{\mathrm{kar}}}

\newcommand{\newterm}[1]{\textit{#1}}
\newcommand{\Vect}[1]{\text{Vect}_{#1}}
\newcommand{\SVect}[1]{\text{SVect}_{#1}}
\newcommand{\weights}{\Gamma}
\newcommand{\glmn}{\mathfrak{gl}(m|n)}
\newcommand{\cartan}{\mathfrak{h}}

\begin{document}

\title[Deligne's $\uRep(GL_\delta)$ and representations of $\mathfrak{gl}(m|n)$]{Deligne's category $\uRep(GL_\delta)$ and representations of general linear supergroups}

\author{Jonathan Comes}
\address{J.C.: Department of Mathematics,
University of Oregon, Eugene, OR 97403, USA}
\email{jcomes@uoregon.edu}

\author{Benjamin Wilson}
\address{B.W.: Dieffenbachstr., 27, Berlin 10967, Germany}
\email{benjamin@asmusas.net}

\begin{abstract}
We classify indecomposable summands of mixed tensor powers of  the natural representation for the general linear supergroup up to isomorphism.  We also give a formula for the characters of these summands in terms of composite supersymmetric Schur polynomials, and give a method for decomposing their tensor products.  Along the way, we describe indecomposable objects in $\uRep(GL_\delta)$ and explain how to decompose their tensor products. 
\end{abstract}

\date{\today}

\maketitle  

\setcounter{tocdepth}{1}
\tableofcontents

\section{Introduction}
\subsection{}Classical Schur-Weyl duality concerns the commuting actions of the symmetric group $\Sigma_r$ and the general linear group $GL_d$ on the tensor power $V^{\otimes r}$ of the natural representation $V$ of $GL_d$.
It enables the labelling of the isomorphism classes of indecomposable summands of the tensor power by partitions $\lambda \vdash r$ with height $l(\lambda) \leq d$ (this is \newterm{Weyl's Strip Theorem}), and the description of the characters of these summands in terms of Schur polynomials.
Schur-Weyl duality for the general linear supergroup $GL(m|n)$, established by Sergeev \cite{Serg} and by Berele and Regev \cite{BR}, provides similar insights into structure of the tensor power $V^{\otimes r}$ of the natural representation $V$ of $GL(m|n)$.
In this case, the isomorphism classes of the indecomposable summands of $V^{\otimes r}$ are parametrized by those partitions $\lambda \vdash r$ that are \newterm{$(m|n)$-hook}, that is, partitions whose Young diagram can be covered by an $m$-wide, $n$-high hook, and the characters of the summands are described by the so-called supersymmetric Schur polynomials.

An analogue of Schur-Weyl duality for the \newterm{mixed tensor powers} 
$$T(r,s) = T_{V,V^*} (r,s) = V^{\otimes r} \otimes (V^*)^{\otimes s}$$ 
of $GL_d$ was developed by \cite{MR1280591,Stembridge1987,Koike,MR1024455}.
Here, the walled Brauer algebras $B_{r,s} (\delta)$, with $\delta = d$, replace the group algebra of the symmetric group as the generic centralizer of the $GL_d$-action.
It is known, in particular, that the indecomposable summands are labelled up to isomorphism by certain bipartitions (i.e. pairs of partitions).
Schur-Weyl duality for mixed tensor powers also holds for the general linear supergroup $GL(m|n)$, where $\delta = m-n$ is the super dimension of the natural representation $V$.
However, in this case, many fundamental questions remain unresolved.
In this paper, we classify for the first time the indecomposable summands of the mixed tensor powers for the general linear supergroups up to isomorphism, and derive a character formula for the indecomposable summands in terms of composite supersymmetric Schur polynomials.
In addition, we describe a method for the decomposition of tensor products of these indecomposable summands.

\subsection{}We work over a field $\K$ of characteristic zero throughout, identify finite dimensional representations of $GL(m|n)$ with integral representations of the Lie superalgebra $\gl(m|n)$, and write $\Rep (GL_d)$ and $\Rep(\glmn)$ for the categories of finite-dimensional representations of $GL_d$ and $\gl(m|n)$, respectively.
Fundamental to our approach is the tensor category $\uRep(GL_\delta)$, defined by Deligne [Del1-2]\nocite{Del96, Del07}, that permits the simultaneous study of the mixed tensor powers for the general linear groups and the general linear supergroups.
This category, which we refer to as \newterm{Deligne's category}, is constructed as the additive and Karoubi envelope of a ``skeleton'' tensor category $\uRep_0(GL_\delta)$ (cf. \S\ref{uGLdefn}).
Up to isomorphism, the objects $w_{r,s}$ of $\uRep_0 (GL_\delta)$ are parametrized by pairs $(r,s)$ of non-negative integers, representing the potencies of a mixed tensor power, and the morphism spaces are spanned by walled Brauer diagrams of the appropriate sizes.
The structure  of Deligne's category depends upon the parameter $\delta \in \K$ so that, in particular,  the endomorphism algebras are the walled Brauer algebras $B_{r,s}(\delta)$.
The universal property of Deligne's category guarantees that for any rigid $\delta$-dimensional object $V$ in a tensor category $\cat T$ satisfying hypotheses familiar from classical Schur-Weyl duality (cf. \S\ref{mtp}), there exists a full tensor functor $F : \uRep(GL_\delta) \to \cat T$ such that for any $r,s \geq 0$, $F(w_{r,s}) = T(r,s)$ is the corresponding mixed tensor power.
In particular, for any $d > 0$ and for any $m,n \geq 0$ there exist full tensor functors
$$ F_d : \uRep(GL_d) \to \Rep(GL_d), \qquad F_{m|n} : \uRep(GL_{m-n}) \to \Rep(\glmn) $$
defined by the natural representations of $GL_d$ and of $GL(m|n)$, respectively.
Recent results in the representation theory of the walled Brauer algebra also play crucial roles.
For any bipartition $\lambda$, let $\lambda^\black, \lambda^\white$ be the partitions defined by $\lambda = (\lambda^\black, \lambda^\white)$, let $l(\lambda) = l(\lambda^\black) + l(\lambda^\white)$, and write $\lambda\vdash ( |\lambda^\black|, | \lambda^\white|)$.
It was shown in \cite{CDDM} that the walled Brauer algebras $B_{r,s}(\delta)$ are cellular, with standard modules  parametrized by the set of bipartitions 
$$\Lambda_{r,s} = \{ \ \lambda \ | \   \lambda\vdash (r-i, s-i), \ 0 \leq i \leq \min(r,s) \ \}.$$
A recursive formula for the decomposition numbers for the walled Brauer algebras was described in \cite{CD}, using cap diagrams introduced by Brundan and Stroppel [BS2-5]\nocite{BS1,BS2,BS3,BS4}.
The classification of the simple modules for the walled Brauer algebras up to isomorphism obtained in \cite{CDDM} enables the parametrization of the indecomposable objects of Deligne's category by bipartitions (cf. Theorem \ref{indeclass}).
Writing $L(\lambda)$ for the indecomposable object corresponding to the bipartition $\lambda$,
 we have the following description for the decomposition of tensor products of indecomposable objects of $\uRep(GL_\delta)$ for generic $\delta$ (see Theorem \ref{genten} for a precise statement).
\begin{theorem}\label{introthm1}
For generic values of $\delta \in \K$,
$$ L(\lambda) \otimes L(\mu) \cong \bigoplus_{\nu} L(\nu)^{\oplus \Gamma_{\lambda, \mu}^\nu}$$
where the sum is over bipartitions $\nu$ and the coefficients $\Gamma_{\lambda, \mu}^\nu$ are given in terms of Littlewood-Richardson coefficients (see \eqref{Gamma}).
\end{theorem}
This theorem is derived using Koike's Theorem \cite{Koike}, which gives the decomposition of a tensor product of irreducible rational representations for the general linear group.
In the case of general linear supergroups, it is compatible with the result of Sergeev \cite{Serg}.
The calculation of the decomposition numbers for the walled Brauer algebras \cite{CD} and the cap diagrams of Brundan and Stroppel permit the definition of a ``lifting isomorphism'' in the spirit of Comes and Ostrik \cite{CO1}.
The lifting isomorphism relates the additive Grothendieck rings of $\uRep(GL_\delta)$ in the singular and generic cases, thus enabling the decomposition of any tensor product of indecomposable objects of $\uRep(GL_\delta)$ for any value of $\delta$. 

The known decomposition of the mixed tensor powers for $GL_d$ \cite{MR1280591,Stembridge1987,Koike,MR1024455} finds the following expression in terms of the functor $F_d$:
\begin{theorem}\label{introthm2}(see Theorem \ref{Fdim})
For any $d > 0$ and any bipartition $\lambda$, $F_d(L(\lambda))$ is an indecomposable object of $\Rep(GL_d)$ and is non-zero if and only if $l(\lambda) \leq d$.
Moreover, any non-zero indecomposable summand of a mixed tensor power $T(r,s)$ in $\Rep(GL_d)$ is isomorphic to $F_d(L(\lambda))$ for precisely one bipartition $\lambda \in \Lambda_{r,s}$ with $l(\lambda) \leq d$.
\end{theorem}

The \newterm{Young diagram of a bipartition} $\lambda$ is obtained by superimposing the Young diagrams for the partitions $\lambda^\black$ and $\lambda^\white$ so that their top and left edges coincide, and then rotating the Young diagram for $\lambda^\white$ $180$-degrees about its upper-left corner (cf. \S\ref{bipart}).
A bipartition is \newterm{$(m|n)$-cross}\footnote{In \cite{MV06}, the term \newterm{$\glmn$-standard} is used.} if its Young diagram can be covered with an $m$-high, $n$-wide cross (cf. \S\ref{vanishW}).
The decomposition of the mixed tensor powers for $\glmn$ can be described in terms of the functor $F_{m|n}$ as follows.
\begin{theorem}\label{introthm3}
For any $m,n \geq 0$ and any bipartition $\lambda$, $F_{m|n}(L(\lambda))$ is an indecomposable object of $\Rep(\glmn)$ and is non-zero if and only if $\lambda$ is $(m|n)$-cross.
Moreover, any non-zero indecomposable summand of a mixed tensor power $T(r,s)$ in $\Rep(\glmn)$ is isomorphic to $F_{m|n}(L(\lambda))$ for precisely one $(m|n)$-cross bipartition $\lambda \in \Lambda_{r,s}$.
\end{theorem}
The same result was recently obtained via a different approach by Brundan and Stroppel \cite{BS}. 
 Their approach yields  additional information about the modules $F_{m|n}(L(\lambda))$; for instance, the irreducible socles and heads are computed explicitly.  

In the case when $s=0$, a bipartition $\lambda \in \Lambda_{r,0}$ is $(m|n)$-cross if and only if $\lambda^\black$ is $(m|n)$-hook.
Thus, in this case, Theorem \ref{introthm3} gives the decomposition of the covariant tensor power $V^{\otimes r}$ familiar from the work of Sergeev \cite{Serg} and of Berele and Regev \cite{BR} (similarly, in the contravariant case, i.e. when $r=0$).
On the other hand, if $n=0$ then a bipartition $\lambda$ is $(m|n)$-cross if and only if $l(\lambda) \leq m$.
Thus Theorem \ref{introthm2} can be seen as the special case of Theorem \ref{introthm3} where $n=0$.

For any bipartition $\mu$, let $\s_{\mu}$ denote the corresponding composite supersymmetric Schur polynomial (see e.g. \cite{MV06}).
Then we have the following formula for the characters of the indecomposable summands of the mixed tensor powers for $GL(m|n)$.
\begin{theorem}\label{introthm4}(see Theorem \ref{chW})
Let $m,n \geq 0$ and $\delta = m-n$.  Then for any bipartition $\lambda$, $$\ch F_{m|n} (L(\lambda)) = \sum_{\mu 
} D_{\lambda, \mu} (\delta) \ \s_\mu,$$ where the $D_{\lambda, \mu}$ are the decomposition numbers for the walled Brauer algebras.
\end{theorem}
When $|\lambda^\white|=0$, the decomposition number $D_{\lambda,\mu}$ is $1$ if $\lambda=\mu$ and $0$ otherwise, and  $\s_\lambda$ is the (non-composite) supersymmetric Schur polynomial associated to $\lambda^\black$.  
Thus, if $\lambda = (\lambda^\black, \emp)$ and $\lambda^\black$ is $(m|n)$-hook, then the character formula of Theorem \ref{introthm4} reduces to that of Sergeev \cite{Serg} and Berele and Regev \cite{BR}.

\subsection{}The paper is organized as follows.  We begin in \S\ref{preliminaries} with a review of the category-theoretic notions necessary for the definition of Deligne's category and derivation of its universal property in \S\ref{delignescat}.
The indecomposable objects of Deligne's category are then classified in \S\ref{indecomposables} using the cellular structure of the walled Brauer algebras as described in \cite{CDDM}.
In \S\ref{Fd}, the representation theory of the general linear group is briefly recalled, and Koike's Theorem on the decomposition of tensor products of indecomposable representations is reviewed.
The lifting isomorphism is defined in \S\ref{lift}, relating the additive Grothendieck rings of Deligne's category in the singular and generic cases, and it is shown that the defining coefficients are the decomposition numbers of the walled Brauer algebras.
The analogue of Koike's Theorem in Deligne's category for generic values of $\delta$ is presented in \S\ref{tensordecomp}, and it is shown that the lifting isomorphism enables the decomposition of tensor products of indecomposables in all cases.
In \S\ref{super}, composite supersymmetric Schur polynomials are introduced and the results of \S\ref{lift} and \S\ref{tensordecomp} are employed to derive the character formula for the indecomposable summands of the mixed tensor powers of the general linear supergroup.
Next, the character formula is used to prove the classification of these indecomposable summands in terms of $(m|n)$-cross bipartitions, as described by Theorem \ref{introthm3}.  Finally, we illustrate how to decompose tensor products of these indecomposable summands with an explicit example.  

\subsection{Acknowledgements}  The majority of the work for this paper was done while the authors shared an office at the Technische Universit\"at M\"unchen.  We would like to thank the university for providing us with the time and freedom to explore this project.  We are also grateful to Maud De Visscher for explaining the results of \cite{CD} in the case $\delta=0$.  The second author would like to thank Victor Ostrik and Jon Brundan for many valuable conversations concerning this paper.  

\section{Category-theoretic preliminaries}\label{preliminaries}

Let $\K$ denote a field of characteristic zero.  A category is said to be \newterm{$\K$-linear} if the Hom sets are equipped with the structure of vector spaces over the field $\K$ in such a way that composition of morphisms is bilinear.

\subsection{Monoidal categories}
For any category $\cat C$, let $\flip_\cat C : \cat C \times \cat C \to \cat C \times \cat C$ denote the functor $(X,Y) \mapsto (Y,X)$.
A \newterm{monoidal category} is a tuple $(\cat C, \otimes, \unit, \braid)$ where $\cat C$ is a category,  
$$ T := - \otimes - : \cat C \times \cat C \to \cat C$$
is a bifunctor and $\unit$ is a distinguished object (the \newterm{unit for $T$}), satisfying
$$ T \circ ( T \times \id_\cat C) = T \circ (\id_\cat C \times T), \qquad \unit \otimes - = \id_{\cat C} = - \otimes \unit,$$ 
and $ \braid : T \nattrans T \circ \flip_\cat C$ is a natural isomorphism (the \newterm{symmetric braiding}) satisfying
\begin{eqnarray*}
(\id_Y \otimes \braid_{X,Z}) \circ (\braid_{X,Y} \otimes \id_Z) &=& \braid_{X, Y \otimes Z}, \qquad \braid_{Y,X} \circ \braid_{X,Y} = \id_{X \otimes Y},\\
(\braid_{X,Z} \otimes \id_Y) \circ (\id_X \otimes \braid_{Y,Z}) &=& \braid_{X \otimes Y, Z},
\end{eqnarray*}
for all objects $X,Y,Z$ of $\cat C$.
When no confusion arises, we may write $\cat C$ for both the underlying category and monoidal category
$(\cat C, \otimes_\cat C, \unit_\cat C, \braid_\cat C)$ and omit the subscript $\cat C$ when it is implicit.
A monoidal category, in our sense, is elsewhere called a strict, symmetric monoidal category.

A \newterm{monoidal functor} is a tuple $(F, \eta, \alpha)$ where $F: \cat C \to \cat{D}$ is a functor, $\cat C, \cat{D}$ are monoidal categories, $\eta : F \circ T_\cat C \nattrans T_\cat{D} \circ (F \times F)$ is a natural transformation and $\alpha : F(\unit_\cat C) \to \unit_\cat{D}$ is an isomorphism such that
\begin{eqnarray*}
&&\eta_{Y,X} \circ F(\braid_{X,Y})  = \braid_{FX, FY} \circ \eta_{X,Y}, \quad
\id_{FX} \otimes (\alpha \circ \eta_{X, \unit}) = \id_{FX}, \\
&&(\eta_{X,Y} \otimes \id_{FZ}) \circ \eta_{X \otimes Y, Z} = (\id_{FX} \otimes \eta_{Y,Z}) \circ \eta_{X, Y \otimes Z}.
\end{eqnarray*}
The monoidal functor is \newterm{strict} if $F(X \otimes Y) = FX \otimes FY$  and $\eta_{X,Y} = \id_{FX \otimes FY}$ for all objects $X,Y$ of $\cat C$, $F\unit = \unit$ and $\alpha = \id_\unit$.
A monoidal functor, in our sense, is elsewhere called a non-strict symmetric monoidal functor.

Let $\cat C, \cat{D}$ be monoidal categories and let 
$$ F = (F, \eta, \alpha), F' = (F', \eta', \alpha') : \cat C \to \cat{D}$$
be monoidal functors. 
A \newterm{monoidal natural transformation} $\epsilon: F \nattrans F'$ is a natural transformation of the underlying functors $F \nattrans F'$ such that $\epsilon_\unit = (\alpha')^{-1} \circ \alpha$ and
$$ \eta'_{X,Y} \circ \epsilon_{X \otimes Y} = \epsilon_X \otimes \epsilon_Y \circ \eta_{X,Y},$$
for all objects $X,Y$.

As in the following examples and throughout, we use monoidal categories in place of their familiar, non-strict, counterparts when convenient.
This is without loss of generality, by Maclane's Coherence Theorem (see e.g. \cite{MacLane}).

\begin{example}\label{vectismonoidal}
For finite-dimensional vector spaces $U$, $V$ over $\K$, write $U \otimes V$ for the usual tensor product.
Define
$ \braid_{U,V} : U \otimes V \to V \otimes U$, by 
$u \otimes v \mapsto v \otimes u $
for all $u \in U$ and $v \in V$, and let $\unit$ denote a one-dimensional vector space.
Let $\Vect \K$ denote the category of finite-dimensional vector spaces and linear maps over $\K$, modulo the identification
\begin{equation}\label{strictidentification}
U \otimes (V \otimes W) = (U \otimes V) \otimes W
\end{equation}
for all objects $U$, $V$ and $W$.
Then $(\Vect \K,  \otimes, \unit, \braid)$ is a monoidal category.
\end{example}
\begin{example}\label{svectismonoidal}
Recall that a superspace over $\K$ is a $\Z / 2 \Z$-graded vector space $U = \psspace U 0 \oplus \psspace U 1$.
Elements of $\psspace U 0$ and $\psspace U 1$ are said to be \newterm{even} and \newterm{odd}, respectively.
An element of $U$ is said to be \newterm{pure} if it is either even or odd.
For $u \in \psspace U i$, write $\parity u = i$ for the \newterm{parity} of $u$.
A morphism of superspaces is simply a morphism of vector spaces.  
If $\varphi : U \to V$ is a superspace morphism, then declare $\varphi$ to be pure and of parity $\parity \varphi$ if 
$$ \varphi : \psspace U i \to \psspace V {i + \parity \varphi}, \qquad i=0,1.$$
Thus the vector space of all superspace morphisms $U \to V$ becomes itself superspace.
Given superspaces $U,V$, let $U \otimes V$ denote the their tensor product as vector spaces, considered as a superspace with the grading
$$ \psspace {(U \otimes V)} i = \bigoplus_{\parity j + \parity k = \parity i} \psspace U j \otimes \psspace V k,$$
and define $\braid_{U,V} : U \otimes V \to V \otimes U$ via
$$ \braid_{U,V} : u \otimes v \mapsto (-1)^{\parity u \parity v} v \otimes u$$
(this is the so-called \newterm{rule of signs}).
Write $\unit$ for a one-dimensional purely even superspace.
Finally, write $\SVect \K$ for the category of all finite-dimensional superspaces and their morphisms, modulo the identification \eqref{strictidentification} for all superspaces $U$,$V$ and $W$.
Then $(\SVect \K,  \otimes, \unit, \braid)$ is a monoidal category.
\end{example}

\subsection{Tensor categories}\label{tensorcat}
Let $\cat C$ be a monoidal category and $X$ an object of $\cat C$.
A \newterm{dual} of $X$ is a tuple $(X^*, \ev_X, \coev_X)$ where $X^*$ is an object of $\cat C$ and $\ev_X$, $\coev_X$ are morphisms
$$ \ev_X : X^* \otimes X \to \unit, \qquad \coev_X : \unit \to X \otimes X^*,$$
of $\cat C$ such that
\begin{equation}\label{triangleequalities}
 ({\id_X \otimes \ev_X}) \circ ({\coev_X \otimes \id_X}) = \id_X, \qquad (\ev_X \otimes \id_{X^*}) \circ (\id_{X^*} \otimes \coev_X) = \id_{X^*}.
\end{equation}
The category $\cat C$ is \newterm{rigid} if every object has a dual.

A \newterm{tensor category} is a rigid $\K$-linear monoidal category such that $\End \unit = \K$ and $- \otimes -$ is a bilinear bifunctor.
A \newterm{(strict) tensor functor} is a (strict) monoidal functor that is preadditive.

Let $\varphi : X \to Y$ be a $\cat C$-morphism.  
Duals $(X^*, \ev_X, \coev_X)$ and $(Y^*, \ev_Y, \coev_Y)$ for $X$ and $Y$ define a \newterm{dual morphism} $\varphi^* : Y^* \to X^*$ by
$$ \varphi^* = \ev_Y \otimes \id_{X^*} \circ \id_{Y^*} \otimes \varphi \otimes \id_{X^*} \circ \id_{Y^*} \otimes \coev_X,$$
and one has that
\begin{equation}\label{dualisearoundbends} 
\ev_X \circ (\varphi^* \otimes \id_X) = \ev_Y \circ (\id_{Y^*} \otimes \varphi), \qquad (\id_X \otimes \varphi^*) \circ \coev_X = (\varphi \otimes \id_{Y^*}) \circ \coev_Y.
\end{equation}
Now let $\cat{D}$ be another monoidal category, and $F= (F, \eta, \alpha) : \cat C \to \cat{D}$ a monoidal functor.
Then the functor $F$ and the dual for $X$ in $\cat C$ define a dual 
\begin{equation}\label{imageofadual}
( F(X^*), \ev_{FX}, \coev_{FX})
\end{equation}
for $FX$ in $\cat{D}$, where
$$ \ev_{FX} = \alpha \circ F (\ev_X) \circ \eta^{-1}_{X^*, X} \qquad \coev_{FX} = \eta_{X, X^*} \circ F(\coev_X) \circ \alpha^{-1}. $$
\begin{example}\label{vectistensor}
For any vector space $U$ in $\Vect \K$, let $U^*$ denote the usual linear dual, and define
$$ \ev_U : U^* \otimes U \to \unit, \qquad \ev_U: \lambda \otimes u \mapsto \lambda(u),$$
for all $\lambda \in U^*, u \in U$.
Choose any basis $\{u_i\}$ of $U$, let $\{\lambda_i\}$ denote the basis of $U^*$ orthonormal to it, and define
$$ \coev_U : \unit \to U \otimes U^*, \qquad 1 \mapsto \sum_i u_i \otimes \lambda_i,$$
(this map is independent of the choice of basis).
Then $(U^*, \ev_U, \coev_U)$ is a dual for $U$, and thus $\Vect \K$ is a tensor category.
\end{example}
\begin{example}\label{svectistensor}
For any superspace $V$ in $\SVect \K$, write $U^*$ for the superspace defined by
$$ \psspace {(U^*)} i = (\psspace U i)^*, \qquad i=0,1$$
and define $\ev_U$, $\coev_U$ exactly as in example (1) above, only choosing the basis $\{ u_i \}$ to consist of pure elements.
Then $(U^*, \ev_U, \coev_U)$ is a dual for $U$, and so $\SVect \K$ is a tensor category. 
\end{example}

The following proposition will be useful later.
\begin{proposition}\label{monoidalnattranspropn}
Let $\cat C$ be a rigid monoidal category, $\cat{D}$ a monoidal category and
$$ \epsilon : (F, \eta, \alpha) \nattrans (F', \eta', \alpha')$$
a monoidal natural transformation of monoidal functors $\cat C \to \cat{D}$.
Then $\epsilon$ is a natural isomorphism, and for all objects $X$ in $\cat C$,
$$ \epsilon_{X^*} = ((\epsilon_X)^*)^{-1} = (\epsilon_X^{-1})^*$$
with respect to the duals for $FX$ and $F'X$ defined by \eqref{imageofadual}.
\end{proposition}
\begin{proof}
Let $X$ be an object of $\cat C$ and $(X^*, \ev_X, \coev_X)$ a dual for $X$.
With respect to the duals for $FX$ and $F'X$ defined by \eqref{imageofadual}, one has
$$ (\epsilon_X)^* = (\ev_{F'X} \otimes \id_{FX^*}) \circ (\id_{F'X^*} \otimes \epsilon_X \otimes \id_{FX^*}) \circ (\id_{F'X^*} \otimes \coev_{FX}).$$
The second equality of the claim follows from the first by functoriality.  We demonstrate the first.
Using that $\epsilon$ is both monoidal and natural, one sees that
\begin{equation}\label{inversioneqns}
\ev_{F'X} \circ (\epsilon_{X^*} \otimes \epsilon_X) = \ev_{FX}, \qquad (\epsilon_X \otimes \epsilon_{X^*}) \circ \coev_{FX} = \coev_{F'X}.
\end{equation}
Thus 
\begin{eqnarray*}
(\epsilon_X)^* \circ \epsilon_{X^*} &=& (\epsilon_X)^* \circ (\epsilon_{X^*} \otimes \id_\unit) \\
&=& (\ev_{F'X} \otimes \id_{FX^*}) \circ (\epsilon_{X^*} \otimes \epsilon_X \otimes \id_{FX^*}) \circ (\id_{FX^*} \otimes \coev_{FX}) \\
&=& \id_{FX}.
\end{eqnarray*}
One shows that $\epsilon_{X^*} \circ (\epsilon_X)^* = (\id_\unit \otimes \epsilon_{X^*}) \circ (\epsilon_X)^* = \id_{F'X}$ in a similar fashion.
\end{proof}

\subsection{Categorical dimension}\label{dimension}
Suppose that $\cat C$ is a rigid monoidal category.
For any object $X$ of $\cat C$ and $\varphi \in \End_\cat C X$, define the \newterm{categorical trace} $\tr \varphi$ by
$$ \tr \varphi = \ev_X \circ \braid_{X, X^*} \circ (\varphi \otimes \id_{X^*}) \circ \coev_X \quad \in \End_\cat C \unit, $$
and define the \newterm{categorical dimension} of $X$ by $\dim X = \tr ( \id_X)$.
The categorical trace and dimension do not depend upon the choice of dual for $X$, and are preserved by any monoidal functor.
One has, moreover, that 
$$ \dim (X \otimes Y) = \dim X \cdot \dim Y,$$ 
for all objects $X, Y$.  If $\cat C$ is a tensor category,
then one has additionally that
$$\tr_X : \End X \to \End \unit = \K$$ 
is a homomorphism of abelian groups for any object $X$, and 
$$\dim (X \oplus Y) = \dim X + \dim Y,$$ 
whenever the biproduct of objects $X$ and $Y$ exists.

\begin{example}
The categorical trace and dimension in $\Vect \K$ coincide with their elementary counterparts.
\end{example}
\begin{example}
The categorical trace and dimension in $\SVect \K$ coincide with supertrace and superdimension, respectively.
That is, if $\varphi \in \End U$ is an endomorphism in $\SVect \K$, then $\tr \varphi = \tr \psspace \varphi 0 - \tr \psspace \varphi 1$, where the summands are traces of vector space endomorphisms, and so $\dim U =  \dim_\K \psspace U 0 -  \dim_\K \psspace U 1$, where the summands are vector space dimensions.
\end{example}

When $\cat C = \Vect \K$ is the category of finite-dimensional vector spaces, categorical trace and dimension coincide with their elementary counterparts.
As described in \S\ref{superrepcategory}, when $\cat C = \SVect \K$ is the category of super vector spaces, the categorical dimension of an object coincides with the superdimension.

\subsection{Functor categories}
For categories $\cat C, \cat{D}$, define the following categories whose objects are functors $\cat C \to \cat{D}$ of the specified type and whose morphisms are natural transformations of the specified type:
\begin{center}
\begin{tabular}{l | c | c | c}
{\bf category} & $\cat C, \cat{D}$ &{\bf  functors} & {\bf nat.\ trans.\ } \\[2pt] \hline
$\funcat (\cat C, \cat{D})$ & any & any & any \\[2pt] 
$\monfuncat (\cat C, \cat{D})$ & monoidal & monoidal & monoidal\\[2pt] 
$\preaddfuncat (\cat C, \cat{D})$ & preadditive & preadditive & any \\[2pt] 
$\tensorfuncat (\cat C, \cat{D})$ & tensor & tensor & monoidal \\[2pt] 
$\strtensorfuncat (\cat C, \cat{D})$ & tensor & strict tensor & monoidal
\end{tabular}
\end{center}

\subsection{The additive envelope}\label{additiveenvelope}
Let $\cat C$ be a preadditive category.
An \newterm{additive envelope} of $\cat C$ is a pair $(\add {\cat C}, \iota)$ where $\add {\cat C}$ is an additive category and $\iota : \cat C \to \add {\cat C}$ is a fully-faithful preadditive functor such that for any additive category $\cat{D}$, the ``restriction functor''
\begin{equation}\label{additiveenvelopeup1}
 \preaddfuncat ( \add {\cat C}, \cat{D}) \xrightarrow{\sim} \preaddfuncat(\cat C, \cat{D}), \qquad F \mapsto F \circ \iota, \qquad (\eta: F \nattrans F') \mapsto \eta_\iota,
\end{equation}
is an equivalence of categories\footnote{preadditive functors necessarily preserve biproducts, when they exist.}.
Thus an additive envelope is unique up to equivalence of categories, when it exists, and the category $\cat C$ may be identified with a full subcategory of $\add {\cat C}$ via the functor $\iota$.

The additive envelope may be constructed as follows.
Let $\add {\cat C}$ denote the category with objects $X = (X_j)$ given by finite-length tuples of objects  from $\cat C$ and morphisms
$$ \varphi : (X_j) \to (Y_i), \qquad \varphi = (\varphi_{i,j}), \qquad \varphi_{i,j} \in \Hom_\cat C (X_j, Y_i),$$
given by ``matrices'' of morphisms from $\cat C$, composed via matrix multiplication, i.e.\
$$ (\varphi \circ \psi)_{i,j} = \sum_k \varphi_{i,k} \circ \psi_{k,j}.$$
Addition of morphisms in $\add {\cat C}$ is defined component-wise by addition in $\cat C$.
Concatenation of tuples defines a biproduct on $\add {\cat C}$ where the injection and projection maps are matrices built from identity and zero morphisms of $\cat C$ in a straightforward manner. 
The empty tuple is a zero object for this biproduct, and so $\add {\cat C}$ is an additive category.
The functor $\iota : \cat C \to \add {\cat C}$ defined by sending any $\cat C$-object $X$ to the length-1 tuple $(X)$ and any $\cat C$-morphism to the $1 \times 1$-matrix $(\varphi)$ is fully-faithful, and is such that the universal property \eqref{additiveenvelopeup1} holds.
It is straightforward to show that if $\cat C$ is a tensor category, then $\add {\cat C}$ can also be endowed with the structure of a tensor category such that $\iota$ is a tensor functor and for any additive tensor category $\cat{D}$, 
\begin{equation}\label{additiveenvelopeup2}
 \tensorfuncat ( \add {\cat C}, \cat{D}) \xrightarrow{\sim} \tensorfuncat(\cat C, \cat{D}), \qquad F \mapsto F \circ \iota, \qquad (\eta: F \nattrans F') \mapsto \eta_\iota,
\end{equation}
is an equivalence of categories.

\subsection{Splittings of Idempotents \& the Karoubi envelope}\label{karoubienvelope}
Let $\cat C$ be a category, $X$ an object of $\cat C$ and $e^2 = e \in \End_\cat C (X)$; one says that $e$ is an \newterm{idempotent} of $\cat C$.
A \newterm{splitting} of $e$ is a tuple $(\im e, \inj_e, \proj_e)$ where $\im e$ is an object of $\cat C$ and
$$ \inj_e : \im e \to X, \qquad \proj_e : X \to \im e,$$
are morphisms of $\cat C$ such that 
\begin{enumerate}\label{splittingrelns}
\item\label{idempotentpart1} $e = \inj_e \circ \proj_e$, \quad \item\label{idempotentpart2} $\id_{\im e} = \proj_e \circ \inj_e.$
\end{enumerate}
One says that the idempotent $e$ \newterm {splits}, and calls $\im e$ the \newterm{image} of $e$.
Given $e^2 = e$ and any tuple satisfying part \eqref{idempotentpart1}, part \eqref{idempotentpart2} is equivalent to $\inj_e$and $\proj_e$ being mono- and epi-morphisms, respectively.

A category is said to be \newterm{Karoubi} if every idempotent of the category splits.
A \newterm{Karoubi envelope} of a category $\cat C$ is a tuple $(\kar {\cat C}, \iota)$ where $\kar {\cat C}$ is Karoubi and $\iota : \cat C \to \kar {\cat C}$ is a fully-faithful functor such that for any Karoubi category $\cat{D}$, the ``restriction functor''
\begin{equation}\label{karoubienvelopeup1}
 \funcat ( \kar {\cat C}, \cat{D}) \xrightarrow{\sim} \funcat(\cat C, \cat{D}), \qquad F \mapsto F \circ \iota, \qquad (\eta: F \nattrans F') \mapsto \eta_\iota,
\end{equation}
is an equivalence of categories.  Thus a Karoubi envelope is unique upto equivalence of categories, when it exists, and the category $\cat C$ may be identified with a full subcategory of $\kar {\cat C}$ via the functor $\iota$.

The Karoubi envelope of any category $\cat C$ can be constructed as follows.
Let $\kar {\cat C}$ denote the category whose objects are tuples $(X, e)$ where $X$ is an object of $\cat C$ and $e \in \End_\cat C X$ is an idempotent, and morphisms
$$ \Hom_{\kar {\cat C}} ( (X,e), (Y,f) ) = \{ \varphi \in \Hom_\cat C (X,Y) | f \circ \varphi = \varphi = \varphi \circ e \}.$$
Write $\varphi_0$ for $\varphi : (X,e) \to (Y,f)$ considered as a morphism of $\cat C$.
Then $\kar {\cat C}$-morphisms $\varphi, \psi$ are equal if and only if they have the same source and target in $\kar {\cat C}$ and $\varphi_0 = \psi_0$.
The composition of morphisms in $\kar {\cat C}$ is inherited from $\cat C$, that is, $\varphi \circ \psi$ is defined by the source of $\psi$, the target of $\varphi$, and $(\varphi \circ \psi)_0 = \varphi_0 \circ \psi_0$; one has that $(\id_{(X,e)})_0 = e$.
Any idempotent $\varphi \in \End_{\kar {\cat C}} (X,e)$ has a splitting $(\im \varphi, \inj_\varphi, \proj_\varphi)$ defined by $\im \varphi = (X,\varphi)$ and $(\inj_\varphi)_0 = (\proj_\varphi)_0 = \varphi_0$.
Thus $\kar {\cat C}$ is a Karoubi category.
The functor $\iota : \cat C \to \kar {\cat C}$ defined by $\iota(X) = (X, \id_X)$ and $(\iota (\varphi))_0 = \varphi$ is fully-faithful fully-faithful, and is such that the universal property \eqref{karoubienvelopeup1} holds.
It can be shown that if $\cat C$ is a tensor category, then $\kar {\cat C}$ can also be endowed with the structure of a tensor category such that $\iota$ is a tensor functor and for any Karoubi tensor category $\cat{D}$
\begin{equation}\label{karoubienvelopeup2}
 \tensorfuncat ( \kar {\cat C}, \cat{D}) \xrightarrow{\sim} \tensorfuncat(\cat C, \cat{D}), \qquad F \mapsto F \circ \iota, \qquad (\eta: F \nattrans F') \mapsto \eta_\iota,
\end{equation}
is an equivalence of categories.
Moreover, if $\cat C$ is an additive tensor category, then so is $\kar {\cat C}$.

Idempotents $e, e'$ of a ring are said to be \newterm{orthogonal} if $e e' = e' e = 0$.
An idempotent is said to be \newterm{primitive} if it is non-zero and can not be written as the sum of two non-zero orthogonal idempotents.
The proof of the following lemma is elementary.
\begin{lemma}\label{karoubilemma}
Let $\cat C$ be a $K$-linear Karoubi category and let $X$ an object of $\cat C$.  Then $X$ is indecomposable if and only if $\id_X$ is a primitive idempotent.
\end{lemma}

\begin{proposition}\label{karoubipropn}
Let $\cat C$ denote a category considered as a full subcategory of its Karoubi envelope $\kar{\cat C}$, and choose splittings for the idempotents of $\cat C$.  
Then any object of $\kar{\cat C}$ is isomorphic to the image of an idempotent of $\cat C$.
\end{proposition}
\begin{proof}
The claim is true of the Karoubi envelope constructed explicitly above, with its constructed splittings.
Moreover, this Karoubi envelope is equivalent to any other Karoubi envelope of $\cat C$ via a functor compatible with the identifications of $\cat C$ as a full subcategory and the choices of splittings.
\end{proof}

\subsection{Krull-Schmidt categories}
A \newterm{$\K$-linear Krull-Schmidt} category is a category that is $\K$-linear, additive and Karoubi, with finite-dimensional Hom-spaces.
Thus, if $\cat C$ is any $\K$-linear category with finite-dimensional Hom-spaces, then the Karoubi envelope of the additive envelope $\kar{(\add{\cat C})}$ is an example of a $\K$-linear Krull-Schmidt category.
Recall that an object $X$ of a preadditive category $\cat C$ is \newterm{indecomposable} if for any biproduct decomposition $X = X_1 \oplus X_2$ with associated maps
$ \inj_i : X_i \to X, \proj_i : X \to X_i$, there exists $i \in \{ 1,2 \}$ with $\inj_i \circ \proj_i = 0 \in \End X$.
As indicated by the following proposition, objects in Krull-Schmidt categories possess essentially unique biproduct decompositions into indecomposable summands, as in the familiar case of finitely-generated modules over a finite-dimensional algebra.

\begin{proposition}\label{indidem}
Let $\cat C$ be a $\K$-linear category considered as a full subcategory of its Karoubi envelope $\kar{\cat C}$, let $A$ be an object of $\cat C$ and $e, e', e'' \in \End_{\cat C} A$ be idempotents with splittings chosen in $\kar{\cat C}$.
Then:

(1) $\im e$ is indecomposable if and only if $e$ is a primitive idempotent, and up to isomorphism, all indecomposables are so obtained.

(2) if $e = e' + e''$ and $e', e''$ are orthogonal, then $\im e \cong \im {e'} \oplus \im {e''}.$

(3) Suppose further that $\End_{\cat C} A$ is finite dimensional.  Then $\im e \cong \im {e'}$ if and only if $e$ and $e'$ are conjugate in $\End_{\cat C} A$.
\end{proposition}
\begin{proof}
To prove part (1), suppose that $e = e' + e''$ is a sum of orthogonal idempotents in $\cat C$.
Then, since $\kar{\cat C}$ is $\K$-linear,
\begin{equation}\label{indidemeq1}
\pi_e \circ e \circ \iota_e  = \pi_e \circ e' \circ \iota_e + \pi_e \circ e'' \circ \iota_e.
\end{equation}
If, instead, $\id_{\im e} = f' + f'' \in \End {\im e}$ is a sum of orthogonal idempotents, then similarly
\begin{equation}\label{indidemeq2}
\iota_e \circ \id_{\im e} \circ \pi_e = \iota_e \circ f' \circ \pi_e + \iota_e \circ f'' \circ \pi_e,
\end{equation}
since $\kar{\cat C}$ is $\K$-linear.
By the splittings relations \ref{splittingrelns}, equations \eqref{indidemeq1} and \eqref{indidemeq2} give decompositions of $\id_{\im e}$ and $e$, respectively, as sums of orthogonal idempotents.
Moreover, the substitution of the summands of \eqref{indidemeq1} for $f', f''$ in \eqref{indidemeq2} yields the original decomposition $e = e' + e''$, and inversely.
Thus there is a bijective correspondence between orthogonal idempotent decompositions of $e$ in $\cat C$ and $\id_{\im e}$ in $\kar{\cat C}$ that is linear in both summands.
Thus the claim follows from lemma \ref{karoubilemma} and proposition \ref{karoubipropn}.

Part (2) follows immediately from $\kar{\cat C}$ being both $\K$-linear and Karoubi and the definition of a biproduct.

To prove part (3), suppose that $\varphi \in \Aut_{\cat C} (A)$ and $e' = \varphi e \varphi^{-1}$.  Then
$$ \pi_{e'} \circ \varphi \circ \iota_e : \im e \to \im {e'}$$
is an isomorphism with inverse $\pi_{e} \circ \varphi^{-1} \circ \iota_{e'}$.
Conversely, suppose that $\Phi : \im e \to \im{e'}$ is an isomorphism.
Then the map
$$ (\End A) e' \to (\End A) e : \alpha \mapsto \alpha \circ \iota_{e'} \circ \Phi \circ \pi_e \circ e$$
is an isomorphism of left $\End A$-modules with inverse
$$ \beta \mapsto \beta \circ \iota_e \circ \Phi^{-1} \circ \pi_{e'} \circ e'$$
by the splitting relations \ref{splittingrelns}.
Since $\End A$ is a finite-dimensional algebra, it follows that $e$ and $e'$ are conjugate\footnote{Recall that if $\Lambda$ is a finite-dimensional algebra, then idempotents $e, e' \in \Lambda$ are conjugate if and only if $\Lambda e \cong \Lambda e'$ as left $\Lambda$-modules.}.
\end{proof}

Recall that a ring $R$ is \newterm{semiperfect} if $R/J$ is semisimple and idempotents of $R/J$ lift to $R$, where $J=J(R)$ denotes the radical.  In particular, any finite-dimensional algebra is semiperfect (see e.g. \cite{AndersonFuller}).

\begin{proposition}
Let $\cat C$ be a preadditive category and $X$ an object of $\cat C$.  Then
\begin{enumerate}
\item If $\End X$ is local, then $X$ is indecomposable.
\item Suppose further that $\cat C$ is Karoubi and $\End X$ is semiperfect.  Then if $X$ is indecomposable, then $\End X$ is local.
\end{enumerate}
\end{proposition}
\begin{proof}
Suppose that $X = X_1 \oplus X_2$ and write $\inj_i : X_i \to X$, $\proj_i: X \to X_i$ for the morphisms defining the biproduct decomposition.
Then $e_i = \proj_i \circ \inj_i$ is an idempotent in $\End X$ for $i=1,2$.
As $\End X$ is local, it has no non-trivial idempotents.  Hence $e_1 = 0$ or $e_2 = 0$.

Suppose that $\cat C$ is Karoubi, that $X$ is indecomposable and that $R = \End X$ is semiperfect.
Then the ring $R/J$ is semisimple, that is, is a semisimple module over itself.
Since idempotents in $\cat C$ split, $R = \End X$ has no non-trivial idempotents, and since $R$ is semiperfect, the same is true of $R/J$.
Hence $R/J$ is a simple module over itself, as any non-trivial summand defines a non-trivial idempotent of $R/J$.
Thus $J$ is a maximal left ideal of $R$.
But $J$ is the intersection of all maximal left ideals of $R$, so $J$ is the unique maximal left ideal.
Therefore $R$ is local.
\end{proof}

\begin{corollary}\label{endringind}
Let $\cat C$ be a $\K$-linear Krull-Schmidt category and let $X$ be an object of $\cat C$.
Then $X$ is indecomposable if and only if $\End X$ is local.
\end{corollary}

\begin{proposition}\label{distinctind}
Let $\cat C$ be a $\K$-linear Krull-Schmidt category, $\cat D$ a preadditive category and $F : \cat C \to \cat D$ a full preadditive functor.
Then $FX$ is indecomposable object of $\cat D$ if $X$ is an indecomposable object of $\cat C$.
Moreover, if $X,Y$ are indecomposable objects of $\cat C$ such that $FX, FY$ are non-zero isomorphic objects of $\cat D$, then $X \cong Y$.
\end{proposition}
\begin{proof}
The first part follows from corollary \ref{endringind} since homomorphic images of local rings are local.
To prove the second part, suppose that $X,Y$ are indecomposable objects of $\cat C$ and that $FX \cong FY$ are non-zero in $\cat D$.
As $F$ is full, there exist morphisms $\varphi : X\to Y$, $\psi: Y \to X$ such that $F (\psi \circ \varphi) = \id_{FX}$.
Thus $\alpha = \psi \circ \varphi$ is not nilpotent.
Hence $\alpha \not \in J$, where $J$ is the radical of the finite-dimensional algebra $\End X$.
As $\End X$ is local, $J$ is the unique maximal left ideal, so it follows that $(\End X) \alpha = \End X$.
Let $\beta \in \End X$ be such that $\beta \circ \alpha = \beta \circ \psi \circ \varphi = \id_X$.
Then $\varphi \circ \beta \circ \psi$ is a non-zero idempotent in the local algebra $\End Y$, hence is equal to $\id_Y$.
Thus $X \cong Y$.
\end{proof}

\subsection{The additive Grothendieck ring $R_{\cat C}$}\label{addGroth}
Let $\cat C$ denote a $\K$-linear Krull-Schmidt category, let $\Z [\cat C]$ denote the free $\Z$-module generated by the isomorphism classes of the objects of $\cat C$, and let $(\cdot, \cdot)_{\cat C}$ denote the bilinear form on $\Z [ \cat C]$ defined by bilinear extension of the rule
$$ ( [U], [V] )_{\cat C} =  \dim_\K \Hom_{\cat C} (U, V).$$

Write $R_{\cat C}$ for the quotient of $\Z [ \cat C ]$ by the relations
$$ [ A \oplus B ] - [A] - [B] $$
for all objects $A$, $B$ in $\cat C$. 
Thus $R_{\cat C}$ is the free $\Z$-module generated by the isoclasses of the indecomposable objects of $\cat C$.
Since
$$ \Hom (A \oplus A', B) \cong \Hom (A, B) \oplus \Hom (A', B),$$
and similarly in the second argument, the defining relations of $R_{\cat C}$ are contained in the left and right radicals of the bilinear form.
By abuse of notation, we use the same notation for the bilinear form induced in this way on $R_{\cat C}$.

We call $R_{\cat C}$ the \newterm{additive Grothendieck group} of $\cat C$.
Note that in the case where $\cat C$ is semisimple, $R_{\cat C}$ is the ordinary Grothendieck group and the bilinear form is non-degenerate with an orthonormal basis given by the isomorphism classes of the simple objects.

Now suppose that $\cat C$ is a $\K$-linear Krull Schmidt tensor category.
For any object $A$ of $\cat C$, the functors $- \otimes A$ and $A \otimes -$ are preadditive, and so setting $[A][B] = [A\otimes B]$ for all objects $A, B$ of $\cat C$ defines a bilinear multiplication on $R_\cat C$.
This multiplication is commutative, since $\cat C$ carries a symmetric braiding, and has unit $[\unit]$.
Thus $R_{\cat C}$ is a commutative ring, called the \newterm{additive Grothendieck ring} of $\cat C$.
The duality on $\cat C$ defines an involutive ring automorphism $*$ of $R_{\cat C}$ via $[A]^* = [A^*]$ for all objects $A$ of $\cat C$.
As duality defines a contravariant endofunctor, one has
$$ ([A], [B])_{\cat C} = ([B]^*, [A]^*)_{\cat C} $$
for all objects $A, B$.
Moreover, the Hom-set adjunction
$$ \Hom (A \otimes B, C) \cong \Hom (A, C \otimes B^{*})$$
(an immediate consequence of equations \eqref{triangleequalities}) gives the invariance relation
$$ ([A][B], [C]) = ([A], [C][B]^*)$$
for all objects $A,B,C$.

\section{The category \underline{Re}p$(GL_\delta)$}\label{delignescat}

In this section we define Deligne's tensor category $\uRep(GL_\delta)$ and prove that it satisfies a certain universal property  (see Proposition \ref{uni}).  To define $\uRep(GL_\delta)$ we first diagrammatically define a smaller ``skeleton category.'' We then take the additive envelope (\S\ref{additiveenvelope}) followed by the Karoubi envelope (\S\ref{karoubienvelope}) of this skeleton category to get $\uRep(GL_\delta)$.  To start, we introduce the diagrams we will use to define the skeleton category.

\subsection{Words and diagrams}\label{worddiagrams}  Suppose $w$ and $w'$ are finite (possibly empty) words in two letters denoted  $\bullet$ (black letter) and $\circ$ (white letter).  A \emph{$(w, w')$-diagram} is a graph which satisfies the following conditions:
\begin{enumerate}

\item[(i)] The vertices are positioned in two (possibly empty) horizontal rows.

\item[(ii)] Each vertex is drawn as either $\black$ or $\white$  so that the bottom (resp. top) row of vertices is the word $w$ (resp. $w'$).

\item[(iii)] Each vertex is adjacent to exactly one edge.

\item[(iv)] An edge is adjacent to both a black and a white vertex if and only if the vertices adjacent to that edge are either both in the top row or both in the bottom row.
\end{enumerate} 

An edge in a $(w, w')$-diagram is called a \emph{propagating edge} if it is adjacent to a vertex in the top row and a vertex in the bottom row.

\begin{example} (1) Let ${\bf 1}$ denote the empty word.  The empty graph is the unique $({\bf 1}, {\bf 1})$-diagram.  On the other hand, there are two $(\black\!\black\!\white\white, {\bf 1})$-diagrams: 

$$\includegraphics{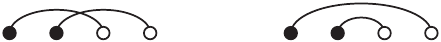}$$

(2) There are six $(\black\!\white\!\white\black, \black\white)$-diagrams:
$$\includegraphics{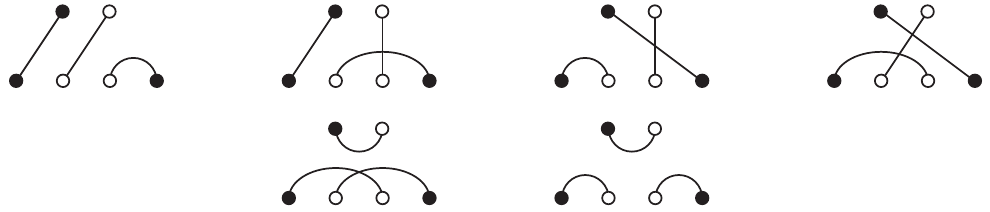}$$ Each of the top four $(\black\!\white\!\white\black, \black\white)$-diagrams have two propagating edges, whereas the bottom two have no propagating edges. 
\end{example}

\begin{remark}\label{existdiag} Suppose $w$ (resp. $w'$) is a word with $r$ (resp. $r'$) black letters and $s$ (resp. $s'$) whites letters.  It is easy to show that a $(w, w')$-diagram exists if and only if $r+s'=r'+s$, in which case the number of $(w, w')$-diagrams is $(r+s')!$. 
\end{remark}

Suppose $w, w'$, and $w''$ are finite words in the letters $\black$ and $\white$.  Given a $(w, w')$-diagram $X$ and a $(w', w'')$-diagram $Y$, we let $Y\star X$ denote the graph obtained by stacking $Y$ atop $X$ so that the top row of vertices of $X$ are identified with the bottom row of vertices of $Y$.  Next, we let $Y\cdot X$ denote the $(w, w'')$-diagram obtained by restricting $Y\star X$ to its top and bottom rows of vertices.  Finally, let $\ell(X, Y)$ denote the number of cycles in $Y\star X$ (i.e. the number of connected components of $Y\star X$ minus the number of connected components of $Y\cdot X$).   For example, if 
$$\includegraphics{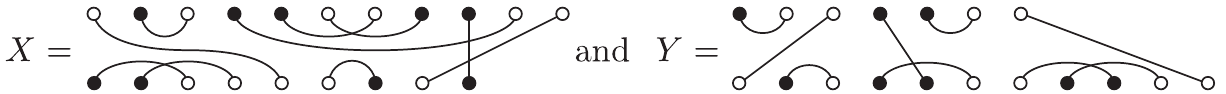}$$

$$\includegraphics{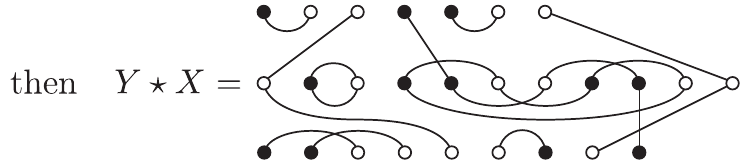}.$$

$$\includegraphics{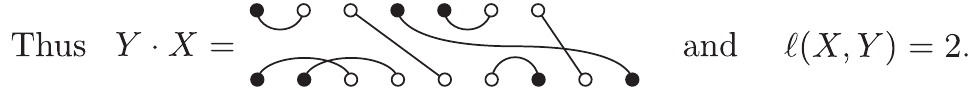}$$

\subsection{The skeleton category}\label{skeleton}  Fix $\delta\in\K$.  Using the setup from \ref{worddiagrams} we can now define the skeleton category $\uRep_0(GL_\delta)$.

\begin{definition} The category $\uRep_0(GL_\delta)$ has

Objects:  finite words in the letters $\black$ and $\white$.

Morphisms:  $\Hom(w, w')$ is the $\K$-vector space on basis $\{(w,w')\text{-diagrams}\}$.

Composition:  $\Hom(w', w'')\times\Hom(w, w')\to\Hom(w, w'')$ sending $(f,g)\mapsto fg$ is the $\K$-bilinear map satisfying $YX=\delta^{\ell(X, Y)}Y\cdot X$ whenever $X$ is a $(w, w')$-diagram and $Y$ is a $(w', w'')$-diagram.

\end{definition}

To show $\uRep_0(GL_\delta)$ is indeed a category, it is easy to check that composition in $\uRep_0(GL_\delta)$ is associative.  Also, if $w$ is a finite word in $\black$ and $\white$, then the $(w, w)$-diagram with each vertex in the bottom row adjacent to the vertex directly above it is the identity morphism in $\End(w)$.  For example, 
$$\includegraphics{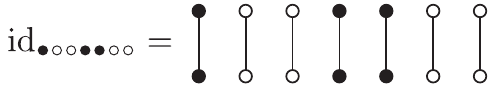}~.$$

\subsection{Tensor category structure of $\uRep_0(GL_\delta)$}\label{tens0}  We will now equip $\uRep_0(GL_\delta)$ with the structure of a tensor category in the sense of \S\ref{tensorcat}.  

\begin{definition}\label{tens} The bifunctor $- \otimes - :\uRep_0(GL_\delta)\times \uRep_0(GL_\delta)\to\uRep_0(GL_\delta)$ is defined as follows:

On objects:  Set $w_1\otimes w_2=w_1w_2$ (concatenation of words) for any objects $w_1$ and $w_2$ in $\uRep_0(GL_\delta)$.

On morphisms:  Assume $w_i$ and $w_i'$ are finite words in $\black$ and $\white$, and $X_i$ is a $(w_i, w_i')$-diagram for $i=1,2$.  Let $X_1\otimes X_2$ denote the $(w_1w_2, w_1'w_2')$-diagram obtained by placing $X_1$ directly to the left of $X_2$.  Now extend $\K$-linearly in both arguments to define tensor products of arbitrary morphisms in $\uRep_0(GL_\delta)$.
\end{definition}

Let $c_{w_1, w_2}:w_1\otimes w_2\to w_2\otimes w_1$ be the $(w_1w_2, w_2w_1)$-diagram such that the vertex in the bottom row corresponding to the $i$th letter in $w_1$ (resp. $w_2$) is adjacent to the vertex in the top row corresponding to the $i$th letter in $w_1$ (resp. $w_2$).  For example, 
$$\includegraphics{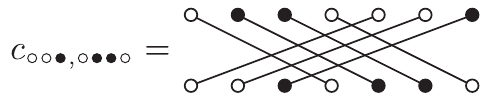}~.$$
It is easy to see that Definition \ref{tens} gives $\uRep_0(GL_\delta)$ the structure of a monoidal category with unit object ${\bf 1}$ (the empty word) and symmetric braiding $c$.  

Next, we will show that $\uRep_0(GL_\delta)$ is rigid.  To do so, given a finite word $w$ in $\black$ and $\white$ let $w^\ast$ denote the word obtained from $w$ by replacing all black letters with white letters and vice versa.  Now define the morphism $\ev_w:w^\ast\otimes w\to{\bf 1}$ (resp. $\coev_w:{\bf 1}\to w\otimes w^\ast$) to be the $(w^\ast w, {\bf 1})$-diagram (resp. $({\bf 1}, ww^\ast)$-diagram) such that the $i$th letter in $w^\ast$ is adjacent to the $i$th letter in $w$ for all $i$.  For example,
$$\includegraphics{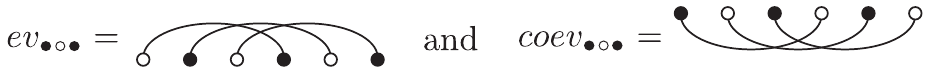}.$$  It is easy to check that $\ev_w$ and $\coev_w$ make $w^\ast$ a dual to $w$.  
Since $\End({\bf 1})=\K$ and $- \otimes -$ is bilinear, it follows that $\uRep_0 (GL_\delta)$ is a tensor category.

\subsection{Universal property of $\uRep_0(GL_\delta)$}\label{uniprop0}
Following \S\ref{dimension} we can compute the categorical dimension of any object in $\uRep_0(GL_\delta)$.  In particular, 
$$\dim(\black)=\ev_\black c_{\black,\white} \coev_\black=\delta$$ 
where the last equality follows from the fact that $$\includegraphics{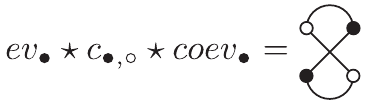}.$$  
$\uRep_0(GL_\delta)$ is characterized as the universal tensor category generated by an object of dimension $\delta$ and its dual \cite{Del07}.  More precisely, it possesses the following universal property:

\begin{proposition} \label{uni0}  Given a tensor category $\cat{T}$, let $\cat{T}_\delta$ denote the category of $\delta$-dimensional objects in $\cat{T}$ and their isomorphisms.  Then the following functor induces an equivalence of categories:$$\begin{array}{rcl}
\Theta : \strtensorfuncat(\uRep_0(GL_\delta), \cat{T}) & \to & \cat{T}_\delta\\
F & \mapsto & F(\black)\\
(\eta:F \nattrans F') & \mapsto & \eta_\black
\end{array}$$
\end{proposition}

\begin{proof}
Since tensor functors preserve categorical dimension, $\Theta (F) = F(\black)$ is an object of the category $\cat{T}_\delta$, and by Proposition \ref{monoidalnattranspropn}, $\Theta (\eta) = \eta_\black$ is an isomorphism of $F(\black)$, hence a morphism of $\cat{T}_\delta$.  Let $X$ be an object of $\cat{T}_\delta$ and let $(X^*, \ev_X, \coev_X)$ be a dual for $X$ in $\cat T$.  Let $w, w'$ be finite words in $\black$ and $\white$.  By the coherence theorem for tensor categories (see for instance \cite{Selinger} and references therein), the image of any $(w,w')$-diagram under a strict tensor functor is completely determined by the image of \begin{equation}\label{pertinentmorphisms}\id_{\black}, \id_{\white}, \ev_\black, \coev_\black.\end{equation}  Since 
$$ \ev_X \circ \braid_{X, X^*} \circ \coev_X = \dim X = \delta,$$ there exists a unique, well-defined, strict tensor functor $F_X : \uRep_0(GL_\delta) \to \cat T$ with  
\begin{eqnarray*}
F_X : &\black \mapsto X, \quad & \white \mapsto X^* \\
&\ev_\black \mapsto \ev_X, \quad & \coev_\black \mapsto \coev_X.
\end{eqnarray*}
Thus $\Theta$ is essentially surjective.

We now show that $\Theta$ is full.
Let $X,Y$ be objects of $\cat{T}_\delta$ and write $F_X, F_Y$ for the functors defined $X,Y$ as above.
Suppose now that $\varphi : X \to Y$ is a morphism in $\cat{T}_\delta$; so $\varphi$ is invertible.
Define a family of isomorphisms 
$$\epsilon = (\epsilon_w : F_X (w) \nattrans F_Y (w))_w$$ 
indexed by finite words $w$, by
$$ \epsilon_\black = \varphi, \qquad \epsilon_\white = (\varphi^{-1})^* $$
and $\epsilon_{w \otimes w'} = \epsilon_w \otimes \epsilon_{w'}$ for all finite words $w, w'$.
It remains to show that $\epsilon$ is a natural transformation, that is, for all morphisms $\sigma : w \to w'$ in $\uRep_0(GL_\delta)$, that $F_Y (\sigma) \circ \epsilon_w = \epsilon_{w'} \circ F_X (\sigma).$
For finite words $w, w'$, we have
\begin{eqnarray*}
F_Y (\braid_{w, w'}) \circ \epsilon_{w w'} &=& \braid_{F_Y w, F_Y w'} \circ \epsilon_{w w'} \\
&=& \epsilon_{w' w} \circ \braid_{F_X w , F_X w'} \\
&=& \epsilon_{w' w} \circ F_X (\braid_{w, w'}).
\end{eqnarray*}
Thus, in verifying naturality, the words may be reordered.
As $\epsilon$ is monoidal by construction, if suffices to check naturality in the cases where $\sigma$ is one of the four morphisms \eqref{pertinentmorphisms}.
The cases of the identity morphisms are trivial.
In the cases of the latter two morphisms, the naturality relations are precisely
$$
\ev_X = \ev_Y \circ (\varphi^{-1})^* \otimes \varphi, \qquad \varphi \otimes (\varphi^{-1})^* \circ \coev_X = \coev_Y,
$$
which follow immediately from \eqref{dualisearoundbends}.  
Thus $\Theta$ is full.

Finally, suppose that $\epsilon$ is a morphism of $\strtensorfuncat(\uRep_0(GL_\delta), \cat{T})$. 
Since $\epsilon$ is monoidal, it is determined by $\epsilon_\black$ and $\epsilon_\white$.
By Proposition \ref{monoidalnattranspropn}, $\epsilon_\white = (\epsilon_\black^{-1})^*$, and so $\epsilon$ is determined by $\epsilon_\black$ alone.
Thus $\Theta$ is faithful.
\end{proof}

\subsection{Definition of $\uRep(GL_\delta)$}\label{uGLdefn}  
Let  $\uRep(GL_\delta) = \kar {(\add{\uRep_0(GL_\delta)})}$ denote the Karoubi envelope of the additive envelope of  $\uRep_0(GL_\delta)$, 
as per \S\ref{additiveenvelope} and \S\ref{karoubienvelope}. 
Thus $\uRep_0(GL_\delta)$ may be identified with a full subcategory of $\uRep(GL_\delta)$, and the tensor category structure of $\uRep(GL_\delta)$ extends that of $\uRep_0(GL_\delta)$.
For every idempotent $e$ of $\uRep_0(GL_\delta)$, fix a splitting $(\im e, \inj_e, \proj_e)$ of $e$ in $\uRep(GL_d)$.
As a notational convenience, whenever $e \in \End X$ and $f \in \End Y$ are idempotents of objects $X,Y$ of $\uRep_0(GL_d)$, identify 
$$\Hom_{\uRep(GL_d)} (\im e, \im f) = f  \Hom_{\uRep_0(GL_d)} (X,Y)  e \subset \Hom_{\uRep_0(GL_d)} (X,Y)$$
via the morphisms $\inj_e, \inj_f, \proj_e, \proj_f$.
For any tensor category $\cat{T}$, let $\cat{H}om'(\uRep (GL_\delta), \cat{T})$ denote the full subcategory of $\tensorfuncat(\uRep (GL_\delta), \cat{T})$ whose objects are those functors whose restriction $\uRep_0 (GL_\delta) \to \cat{T}$ yields a strict tensor functor.
The category $\uRep(GL_\delta)$ has the following universal property  (see  \cite[Proposition 10.3]{Del07}).

\begin{proposition} \label{uni}  
Suppose that $\cat{T}$ is a tensor category and let $\cat{T}_\delta$ be as in Proposition \ref{uni0}.
Then the following functor induces an equivalence of categories:$$\begin{array}{rcl}
\cat{H}om'(\uRep(GL_\delta), \cat{T}) & \to & \cat{T}_\delta\\
F & \mapsto & F(\black)\\
(\eta:F \nattrans F') & \mapsto & \eta_\black
\end{array}$$
\end{proposition}
\begin{proof}
The universal properties of the additive envelope \eqref{additiveenvelopeup2} and Karoubi envelope \eqref{karoubienvelopeup2}, yield that
$$ \tensorfuncat (\uRep(GL_\delta), \cat{T}) \cong \tensorfuncat (\uRep_0(GL_\delta), \cat{T}) $$
via the restriction functors there described.
As $\strtensorfuncat (\uRep_0(GL_\delta), \cat{T})$ is a full subcategory of the latter, the result follows from Proposition \ref{uni0}.
\end{proof}

\section{Indecomposable objects in \underline{Re}p$(GL_\delta)$}\label{indecomposables}

The main goal of this section is to classify isomorphism classes of indecomposable objects in $\uRep(GL_\delta)$.  By Proposition \ref{indidem}(1), these indecomposable objects correspond to primitive idempotents in endomorphism algebras of $\uRep_0(GL_\delta)$.  In light of this, we first describe the classification of conjugacy classes of such primitive idempotents.

To start, let us fix some notation.  For nonnegative integers $r$ and $s$, let $w_{r,s}$ denote the word with $r$ black letters followed by $s$ white letters:
$$w_{r,s}=\underbrace{\black\cdots\black}_{r}\underbrace{\white\cdots\white}_{s}.$$  Using the symmetric braiding it is easy to see that every object in $\uRep_0(GL_\delta)$ is naturally isomorphic to $w_{r,s}$ for some $r, s\geq 0$.  Hence, we will only consider endomorphisms of the $w_{r,s}$'s.  We will write $\K B_{r,s}(\delta)$ (or just $B_{r,s}$) for the endomorphism algebra $\End(w_{r,s})$.  The algebras $B_{r,s}$ are the so-called \emph{walled Brauer algebras} (compare with \cite{MR1280591}, \cite{Koike}, \cite{MR1024455}).  It is well known that conjugacy classes of primitive idempotents in an algebra $A$ are in bijective correspondence with isomorphism classes of simple $A$-modules\footnote{It will be convenient for us to work with right modules.  However, the categories of right and left $B_{r,s}$-modules are equivalent via the anti-automorphism on $B_{r,s}$ given by reading diagrams up rather than down the page.}, which in turn are in bijective correspondence with isomorphisms classes of projective indecomposable $A$-modules, hereafter referred to as PIMs (see for example \cite{Benson}).  In this correspondence a primitive idempotent $e\in A$ corresponds to the PIM $eA$.  
The isomorphism classes of simples in the walled Brauer algebras are classified in \cite{CDDM}.  To explain their classification, we first need to recall some properties of (bi)partitions and their relation to symmetric groups.

\subsection{(Bi)partitions}\label{bipart}  A \emph{partition}  is a tuple of nonnegative integers $\alpha=(\alpha_1, \alpha_2,\ldots)$ whose \emph{parts} (i.e. $\alpha_i$'s) are such that $\alpha_i\geq\alpha_{i+1}$ for all $i>0$, and $\alpha_i=0$ for all but finitely many $i$.  We write $|\alpha|=\sum_{i>0}\alpha_i$ for the \emph{size} of $\alpha$ and we write $\alpha\vdash |\alpha|$.  We define the \emph{length} of $\alpha$, written $l(\alpha)$, to be the smallest positive integer with $\alpha_{l(\alpha)+1}=0$.  We will sometimes write $(\cdots 2^{a_2}1^{a_1})$ for the partition with $a_i$ parts equal to $i$.  It will be convenient for us to identify a partition $\alpha$ with its \emph{Young diagram} which consists of $l(\alpha)$ left-aligned rows of boxes, with $\alpha_i$ boxes in the $i$th row (reading from top to bottom).  For example, 
$$\includegraphics{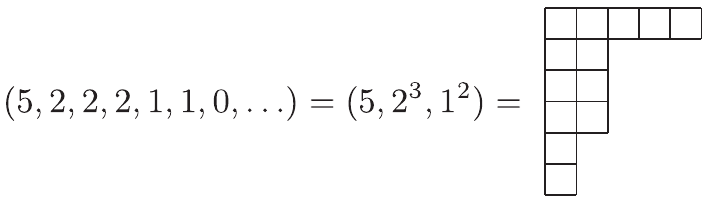}.$$Next, we let $\alpha^{\bf t}$ denote the \emph{transpose} of $\alpha$, i.e. $\alpha^{\bf t}_i$ is the number of boxes in the the $i$th column of $\alpha$.   For example, $(5,2^3,1^2)^{\bf t}=(6, 4, 1^3)$.  Finally, we let $\cat{P}$ denote the set of all partitions.  

Elements of $\cat{P}\times\cat{P}$ are called \emph{bipartitions}.  Given a bipartition $\lambda$, we let $\lambda^\black$ and $\lambda^\white$ denote the partitions such that $\lambda=(\lambda^\black, \lambda^\white)$.    We write $|\lambda|=(|\lambda^\black|, |\lambda^\white|)$ for the  \emph{size} of $\lambda$ and we write $\lambda\vdash|\lambda|$.  Moreover, we write $l(\lambda):=l(\lambda^\black)+l(\lambda^\white)$ for the \emph{length} of $\lambda$.  We define a partial order on sizes of bipartitions by declaring that $(a, b)\leq(c, d)$ whenever $a\leq c$ and $b\leq d$.  In particular, we write $|\mu|<|\lambda|$ to mean $|\mu|\leq|\lambda|$ and $\mu\not=\lambda$.  Additionally, we set $\lambda^\ast=(\lambda^\black, \lambda^\white)^\ast=(\lambda^\white, \lambda^\black)$.  We also have a bipartition-version of a Young diagram which we get as follows:  first place the Young diagram of $\lambda^\circ$ atop the Young diagram of $\lambda^\black$ so that the upper left corners are overlapping, then rotate the Young diagram of $\lambda^\circ$ 180 degrees about its upper left corner.  For example, the diagram associated to the bipartition $((4, 3, 1), (2^2, 1^2))$ is 
$$\includegraphics{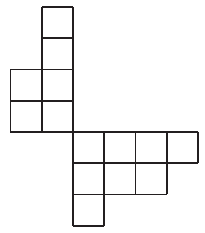}.$$

\subsection{Symmetric groups}
For a nonnegative integer $r$, let $\Sigma_r$ denote the symmetric group on $r$-elements, and let $\KS_r$ denote the corresponding group algebra\footnote{By convention,  $\Sigma_0$ denotes the trivial group of size $0!=1$, hence we identify $\KS_0=B_{0,0}=\K$.}.  It is well known that the primitive idempotents in $\KS_r$ (up to conjugation) are in bijective correspondence with partitions of size $r$ (see for example \cite{FH}).  Given $\alpha\vdash r$, let $z_\alpha\in\KS_r$ denote the corresponding primitive idempotent.  
For example, $z_{(r)}=\frac{1}{n!}\sum_{\sigma\in \Sigma_r}\sigma$ so that the partition $(r)$ corresponds to the trivial $\KS_r$-module.

We now connect the theory of symmetric groups with that of the walled Brauer algebras.  Regardless of $\delta$, the walled Brauer algebras $B_{r, 0}$ and $B_{0, r}$ are isomorphic to the group algebra $\KS_r$.  These isomorphisms are given by $$\begin{array}{rcccl}
B_{r,0} & \larup{\sim} & \KS_r & \arup{\sim} & B_{0,r}\\
\sigma^\black & \mapsfrom & \sigma & \mapsto & \sigma^\white\\
\end{array}$$
where, given $\sigma\in\Sigma_r$,  $\sigma^\black$ (resp. $\sigma^\circ$) is the $(w_{r,0},w_{r,0})$-diagram (resp. $(w_{0,r},w_{0,r})$-diagram) whose $i$th bottom vertex is adjacent to its $\sigma(i)$th top vertex (reading left to right) for $1\leq i\leq r$.  For example, if $\sigma\in\Sigma_5$ is the $3$-cycle $2\mapsto 3\mapsto 5\mapsto 2$, then 
$$\includegraphics{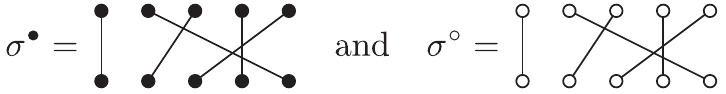}~.$$
More generally, given nonnegative integers  $r$ and $s$, we have an inclusion of algebras $\K[\Sigma_r\times\Sigma_s]\into B_{r,s}$ given by $(\sigma, \tau)\mapsto \sigma^\black\otimes\tau^\circ$ for all $\sigma\in\Sigma_r, \tau\in\Sigma_s$ (here $\K[\Sigma_r\times\Sigma_s]$ denotes the group algebra of the direct product $\Sigma_r\times\Sigma_s$).  Using this embedding we can consider $\K[\Sigma_r\times\Sigma_s]$ as a subalgebra of $B_{r,s}$, and we will do so for the rest of the paper.  

Now, let $J\subset B_{r,s}$ denote the $\K$-span of all $(w_{r,s}, w_{r,s})$-diagrams with less than $r+s$ propagating edges.  One can show that $J$ is a two-sided ideal in $B_{r,s}$ with $B_{r,s}/J\cong\K[\Sigma_r\times\Sigma_s]$ (see \cite[Proposition 2.3 and (2)]{CDDM}).  Hence, we have a surjection of algebras $\pi:B_{r,s}\onto\K[\Sigma_r\times\Sigma_s]$.  It is straightforward to show
\begin{equation}\label{piprop}\pi(a)=a\quad\text{for all }a\in\K[\Sigma_r\times\Sigma_s]\subset B_{r,s}.\end{equation} 

\subsection{Definition of the idempotent $e_\lambda$} Given a bipartition $\lambda\vdash(r,s)$, we set $z_\lambda:= z_{\lambda^\black}^\black\otimes z_{\lambda^\white}^\white\in\K[\Sigma_r\times\Sigma_s]\subset B_{r,s}$.    Note that $z_\lambda$ is an idempotent defined up to conjugation.  The assignment $\lambda\mapsto z_\lambda$ induces a bijection between bipartitions of size $(r, s)$ and primitive idempotents in $\K[\Sigma_r\times\Sigma_s]$ (up to conjugation).  It is important to notice that while $z_\lambda$ is a primitive idempotent in $\K[\Sigma_r\times\Sigma_s]$, it will generally not be primitive in the (usually) larger algebra $B_{r,s}$.  
Let $z_\lambda=e_1+\cdots+e_k$ be a decomposition of $z_\lambda$ into mutually orthogonal primitive idempotents in $B_{r,s}$.  Then $\pi(e_1),\ldots, \pi(e_k)$ are mutually orthogonal idempotents in $\K[\Sigma_r\times\Sigma_s]$ whose sum, by (\ref{piprop}), is $z_\lambda$.  As $z_\lambda$ is primitive in $\K[\Sigma_r\times\Sigma_s]$, there is a unique $i\in\{1,\ldots, k\}$ such that $\pi(e_i)\not=0$.  Set $e_\lambda=e_i$.  Again, note that $e_\lambda\in B_{r,s}$ is a primitive idempotent defined up to conjugation.  

\begin{example}\label{elamexs} (1) Let $\varnothing$ denote the empty partition $(0,0, \ldots)$.  Then $z_\varnothing$ is the identity element of $\KS_0=B_{0,0}$.  Hence, if $\lambda\vdash(r,0)$, then $z_\lambda=z_{\lambda^\black}^\black\otimes\id_{\bf 1}=z_{\lambda^\black}^\black$.  Moreover, since $B_{r,0}=\KS_r$, we also have $e_\lambda=z_{\lambda^\black}^\black$. Similarly, if $\lambda\vdash(0,s)$ then $e_\lambda=z_{\lambda^\white}^\white$.  As a special case, $e_{(\varnothing, \varnothing)}=\id_{\bf 1}$ (the empty graph).

(2) Consider the bipartition $\bibox$.   $z_{\Box}^\black=\id_\black$ and $z_{\Box}^\white=\id_\white$ which implies $z_\bibox=\id_{\black\white}$.  If $\delta\not=0$, then $\id_{\black\white}$ decomposes as $e_1+e_2$ in $B_{1,1}$ where $e_1$ and $e_2$ are the following orthogonal primitive idempotents:  
$$\includegraphics{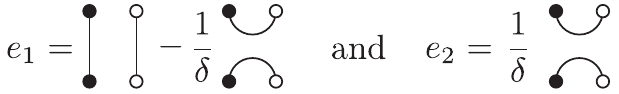}~.$$
In this case $\pi(e_1)=\id_{\black\white}$ and $\pi(e_2)=0$, hence $e_\bibox=e_1$.  If $\delta=0$, then $\id_{\black\white}$ is primitive in $B_{1,1}$, and hence is equal to $e_\bibox$.
\end{example}

We close this subsection with the following useful proposition.  For a proof, we refer the reader to the proof of a completely analogous statement for $\uRep(S_t)$ found \cite[Proposition 3.8]{CO1}.

\begin{proposition}\label{abs} The idempotents $e_\lambda$ are absolutely primitive.  In other words, if $\K\subset\K'$ is a field extension then $e_\lambda\in\K B_{r,s}(\delta)$ is primitive when viewed as an idempotent in $\K'B_{r,s}(\delta)$.
\end{proposition}

\subsection{Definition of the idempotent $e_\lambda^{(i)}$}  Next, we explain how to construct new idempotents from the the $e_\lambda$'s.  Consider the following morphisms: $$\begin{array}{ll}
\psi_{r,s} =(\id_\black)^{\otimes r}\otimes\coev_\black\otimes(\id_\white)^{\otimes s}, & \\[2pt]
\hat{\psi}_{r,s} =(\id_\black)^{\otimes r}\otimes\ev_\white\otimes(\id_\white)^{\otimes s}, & \\[2pt]
\phi_{r,s} =(\id_\black)^{\otimes r}\otimes((\ev_\white\otimes\id_\white)(\id_\black\otimes c_{\white,\white}))\otimes(\id_\white)^{\otimes s-1} & (s>0), \\[2pt]
\hat{\phi}_{r,s} =(\id_\black)^{\otimes r-1}\otimes((\id_\black\otimes\ev_\white)( c_{\black,\black}\otimes\id_\white))\otimes(\id_\white)^{\otimes s} & (r>0).\end{array}$$  For example, 
$$\includegraphics{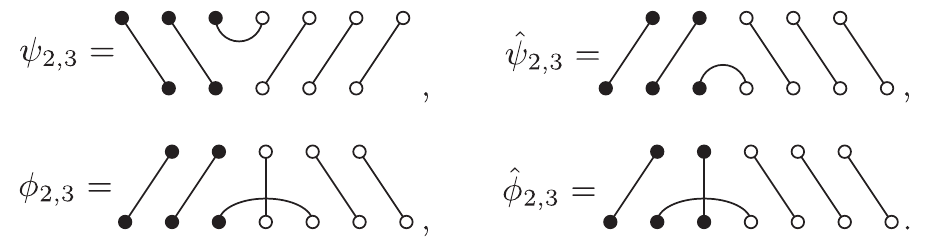}$$  
The following identities easily follow from the definitions above:  \begin{equation}\label{phipsi}
\phi_{r,s}\psi_{r,s}=\id_{w_{r,s}},\quad \hat{\phi}_{r,s}\psi_{r,s}=\id_{w_{r,s}},\quad \hat{\psi}_{r,s}\psi_{r,s}=\delta\id_{w_{r,s}}.
\end{equation}  Now, given a bipartition $\lambda\vdash(r,s)$ we set $e_\lambda^{(0)}=e_\lambda$ and define $e_\lambda^{(i)}\in B_{r+i, s+i}$ for $i>0$ recursively by
$$e_\lambda^{(i)}=\left\{\begin{array}{ll} 
\psi_{r+i-1,s+i-1}e_\lambda^{(i-1)}\phi_{r+i-1,s+i-1} & \text{if }s>0, \\[3pt]
\psi_{r+i-1,s+i-1}e_\lambda^{(i-1)}\hat{\phi}_{r+i-1,s+i-1} &  \text{if }s=0\text{ and } r>0, \\[3pt]
\frac{1}{\delta}\psi_{i-1,i-1}e_\lambda^{(i-1)}\hat{\psi}_{i-1,i-1} & \text{if }\lambda=(\varnothing,\varnothing)\text{ and }\delta\not=0.
\end{array}\right.$$  
Notice that $e_\lambda^{(i)}$ is undefined when $i>0$, $\lambda=\biemp$ and $\delta=0$.  However, $e^{(i)}_\lambda$ is defined (up to conjugation) in all other cases.

\begin{example}\label{elamiexs} (1) In Example \ref{elamexs}(1) we found that $e_\biemp$ is the empty graph.  Hence for $\delta\not=0$ we have $$\includegraphics{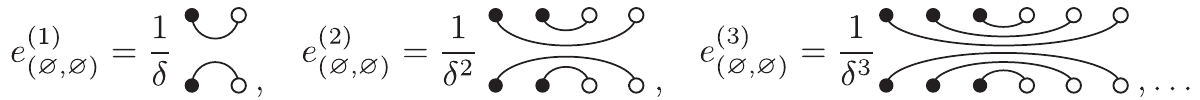}$$

(2) By Example \ref{elamexs}(1), for any $\delta\in\K$ we have $$\includegraphics{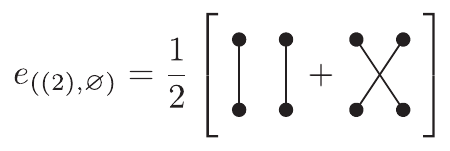}.$$   $$\includegraphics{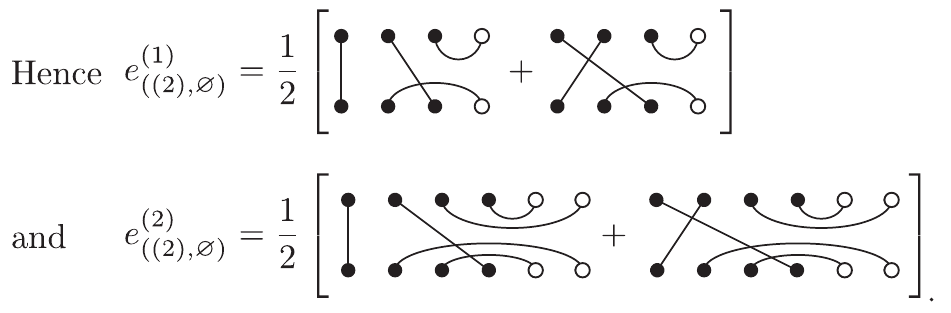}$$

(3) If $\delta=0$, then it follows from Example \ref{elamexs}(2) that $$\includegraphics{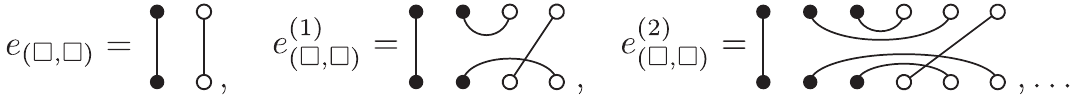}$$
On the other hand, if $\delta\not=0$ then by  Example \ref{elamexs}(2) $$\includegraphics{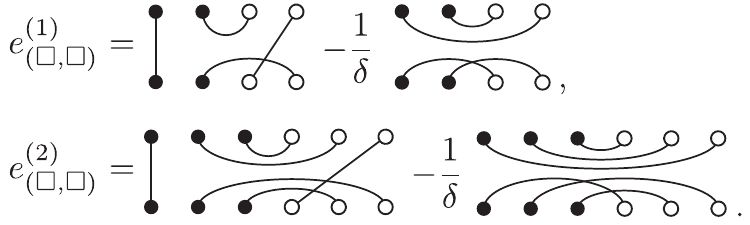}$$

\end{example}

It follows from (\ref{phipsi}) that $e_\lambda^{(i)}$ is an idempotent whenever it is defined.  Hence, the image of $e_\lambda^{(i)}$ is an object in $\uRep(GL_\delta)$. 

\begin{proposition}\label{isomlambdas} Given a bipartition $\lambda$, the objects $\im e_\lambda$ and $\im e_\lambda^{(i)}$ are isomorphic in $\uRep(GL_\delta)$ whenever $e_\lambda^{(i)}$ is defined.
\end{proposition}

\begin{proof} It suffices to show $\im e_\lambda^{(i-1)}\cong \im e_\lambda^{(i)}$ whenever $i>0$.  Using (\ref{phipsi}) it is easy to check that the following table lists mutually inverse morphisms between $\im e_\lambda^{(i-1)}$ and $\im e_\lambda^{(i)}$ in all desired cases:$$\begin{array}{c|c|c}
\im e_\lambda^{(i-1)}\to\im e_\lambda^{(i)} & \im e_\lambda^{(i)}\to \im e_\lambda^{(i-1)} & \text{{\bf case:}}\\[2pt]\hline
\psi_{r+i-1, s+i-1}e_\lambda^{(i-1)}& e_\lambda^{(i-1)}\phi_{r+i-1,s+i-1}& s>0\\[2pt]
\psi_{r+i-1, s+i-1}e_\lambda^{(i-1)}& e_\lambda^{(i-1)}\hat{\phi}_{r+i-1,s+i-1}& s=0, r>0\\[2pt]
\psi_{i-1, i-1}e_\lambda^{(i-1)}& \frac{1}{\delta}e_\lambda^{(i-1)}\hat{\psi}_{i-1,i-1}& \lambda=\biemp, \delta\not=0
\end{array}$$
\end{proof}

\begin{corollary}\label{absi} $e_\lambda^{(i)}$ is absolutely primitive whenever it is defined.
\end{corollary}

\begin{proof}  By Proposition \ref{abs}, $e_\lambda$ is absolutely  primitive.  Thus, by Proposition \ref{indidem}(1), $\im e_\lambda\cong\im e_\lambda^{(i)}$ is absolutely indecomposable, which implies $e_\lambda^{(i)}$ is absolutely primitive.
\end{proof}

\subsection{Classification of primitive idempotents in walled Brauer algebras}  

We are now in position to state the classification of conjugacy classes of primitive idempotents in walled Brauer algebras.  The next theorem is merely a translation of the classification of simple modules for walled Brauer algebras \cite[Theorem 2.7]{CDDM} to the language of primitive idempotents (see for instance \cite{Benson}).

\begin{theorem}\label{idempclass} (1) If $r\not=s$ or $\delta\not=0$ then $\{e_\lambda^{(i)}~|~\lambda\vdash(r-i, s-i), 0\leq i\leq\min(r,s)\}$ is a complete set of pairwise non-conjugate primitive idempotents in $B_{r,s}$.

(2) If $\delta=0$ and $r>0$, then $\{e_\lambda^{(i)}~|~\lambda\vdash(r-i, r-i), 0\leq i<r\}$ is a complete set of pairwise non-conjugate primitive idempotents in $B_{r,r}$.

\end{theorem}

\begin{corollary}\label{absolute}  Primitive idempotents in walled Brauer algebras are absolutely primitive.
\end{corollary}

\begin{proof} This follows from Theorem \ref{idempclass} along with Corollary \ref{absi}.
\end{proof}

\subsection{Classification of indecomposable objects in $\uRep(GL_\delta)$}  Given a bipartition $\lambda$, let  $L(\lambda)$ denote the image of $e_\lambda$ in $\uRep(GL_\delta)$.  Since the primitive idempotent $e_\lambda$ is only defined up to conjugation, $L(\lambda)$ is an indecomposable object in $\uRep(GL_\delta)$ which is defined up to isomorphism.  The following proposition concerning $L(\lambda)$ will be used to prove our upcoming classification of indecomposable objects.  Since the result will be used later in the paper, we record it separately here:

\begin{proposition}\label{HomLs} If $\lambda$ and $\mu$ are bipartitions with $\Hom(L(\lambda),L(\mu))\not=0$, then $|\lambda^\black|+|\mu^\white|=|\lambda^\white|+|\mu^\black|$.
\end{proposition}

\begin{proof}  $\Hom(L(\lambda),L(\mu))=\Hom(\im e_\lambda,\im e_\mu)\subset\Hom(w_{|\lambda^\black|,|\lambda^\white|}, w_{|\mu^\black|,|\mu^\white|})$, hence the proposition follows from  Remark \ref{existdiag}.
\end{proof}

Now we are ready to classify indecomposable objects in $\uRep(GL_\delta)$.

\begin{theorem}\label{indeclass}  The assignment $\lambda\mapsto L(\lambda)$ induces a bijection $$\left\{\begin{tabular}{c}bipartitions of\\
arbitrary size\end{tabular}\right\}\arup{\text{bij.}}\left\{\begin{tabular}{c}nonzero indecomposable objects in\\
 $\uRep(GL_\delta)$, up to isomorphism\end{tabular}\right\}$$
\end{theorem}

\begin{proof}  By Proposition \ref{indidem}(1) every indecomposable object in $\uRep(GL_\delta)$ is isomorphic  to the image of a primitive idempotent endomorphism in $\uRep_0(GL_\delta)$.  Since every object in $\uRep_0(GL_\delta)$ is isomorphic to $w_{r,s}$ for some $r,s\geq 0$, it follows that each indecomposable object in $\uRep(GL_\delta)$ is isomorphic to the image of a primitive idempotent in $B_{r,s}$ for some $r, s\geq 0$.  Hence, by Theorem \ref{idempclass} and Proposition \ref{isomlambdas} the assignment $\lambda\mapsto L(\lambda)$ is surjective.  

Now suppose $\lambda\vdash(r,s)$ and $\mu\vdash(r',s')$ are two bipartitions with $L(\lambda)\cong L(\mu)$.  For convenience, assume $r\geq r'$ so that $(r, s)=(r'+i, s'+i)$ for some integer $i\geq 0$ (Proposition \ref{HomLs}).  If $\mu=\biemp$, then the existence of an isomorphism $L(\biemp)\rightarrow L(\lambda)$ implies that the composition map \begin{equation}\label{comp}\Hom(w_{r,s},{\bf 1})\times\Hom({\bf 1}, w_{r,s})\to B_{0,0}\end{equation} is nonzero.  If $\delta=0$, then the map (\ref{comp}) is necessarily zero unless $r=s=0$.  Hence, if $\mu=\biemp$ and $\delta=0$, then $\lambda=\biemp$ too.  Now assume $\mu\not=\biemp$ or $\delta\not=0$ so that $e_\mu^{(i)}$ is defined.  By Proposition  \ref{isomlambdas},  $\im  e_\mu^{(i)}\cong \im e_\mu\cong\im e_\lambda$, which implies $e^{(i)}_\mu$ and $e_\lambda$ are conjugate idempotents in $B_{r,s}$ (see Proposition \ref{indidem}(3)).  By Theorem \ref{idempclass}, $\lambda=\mu$.  Thus the assignment $\lambda\mapsto L(\lambda)$ is injective.
\end{proof}

\begin{remark} Instead of relying on \cite{CDDM}, one can prove Theorem \ref{indeclass} with straightforward modifications of the proof of \cite[Theorem 3.7]{CO1}.
\end{remark}

We end this subsection with a couple propositions concerning $L(\lambda)$ which will be useful later.

\begin{proposition}\label{zdecomp}  Given a bipartition $\lambda\vdash(r,s)$, \begin{equation*}\label{imz}L((\lambda^\black,\emp))\otimes L((\emp,\lambda^\white))=\im z_\lambda=L(\lambda)\oplus L(\mu^{(1)})\oplus\cdots\oplus L(\mu^{(k)})\end{equation*} for some bipartitions $\mu^{(1)},\ldots,\mu^{(k)}$ which have the property    $\mu^{(j)}\vdash(r-i_j,s-i_j)$ with $0<i_j\leq\min(r,s)$ for all $j=1,\ldots, k$.
\end{proposition}

\begin{proof}  First, using Example \ref{elamexs}(1) we have $e_{(\lambda^\black,\emp)}\otimes e_{(\emp,\lambda^\white)}=z^\black_{\lambda^\black}\otimes z^\white_{\lambda^white}=z_\lambda$ which implies the left equality.  For the right equality, notice that by the definition of $e_\lambda$ we can write $z_\lambda=e_\lambda+e_1+\cdots+e_k$ where $e_\lambda, e_1,\ldots, e_k$ are mutually orthogonal primitive idempotents in $B_{r,s}$.  Moreover,  $\pi(e_j)=0$ for all $j=1,\ldots,k$.  By Theorem \ref{idempclass}, there exists a bipartition $\mu^{(j)}\vdash(r-i_j, s-i_j)$ for some $0\leq i_j\leq\min(r,s)$ such that $e_j$ is conjugate to $e_{\mu^{(j)}}^{(i_j)}$ for all $j=1,\ldots, k$.  It follows that $\pi(e_{\mu^{(j)}}^{(i_j)})=0$, which implies $i_j\not=0$ for each $j=1,\ldots,k$.  Finally, by  Propositions \ref{indidem}(2) and \ref{isomlambdas} we are done.
\end{proof}

The following example illustrates Proposition \ref{zdecomp}.

\begin{example}\label{zboxdec} (1) Assume $\delta\not=0$. Then by Examples \ref{elamexs}(2) and \ref{elamiexs}(1) we have the following orthogonal decomposition of $z_\bibox=\id_{\black\white}$ into primitive idempotents: $z_\bibox=e_\bibox+e_\biemp^{(1)}$.   Hence $\black\white=L(\bibox)\oplus L(\biemp)$ in $\uRep(GL_\delta)$.

(2)  When $\delta = 0$, $z_\bibox=e_\bibox$, by Example \ref{elamexs}(2). 
Hence $\black\white=L(\bibox)$ in $\uRep(GL_0)$.  
\end{example}

\begin{proposition}\label{Ldual} $L(\lambda)^*=L(\lambda^*)$.
\end{proposition}

\begin{proof} Given $\sigma\in\Sigma_r$, it is easy to check that the dual morphisms $(\sigma^\black)^*=(\sigma^\white)^{-1}$ and $(\sigma^\white)^*=(\sigma^\black)^{-1}$.  Given a partition $\alpha\vdash r$, the idempotent $z_\alpha\in\KS_r$ is invariant under the operation $\KS_r\to\KS_r$, $\sigma\mapsto\sigma^{-1}$.  Hence, $z_\lambda^*=(z^\black_{\lambda^\black}\otimes z^\white_{\lambda^\white})^*=z^\white_{\lambda^\black}\otimes z^\black_{\lambda^\white}.$  Hence, up to isomorphism we have $$(\im z_\lambda)^*=\im(z_\lambda^*)=\im(z^\white_{\lambda^\black})\otimes \im(z^\black_{\lambda^\white})=\im(z^\black_{\lambda^\white})\otimes\im(z^\white_{\lambda^\black})=\im(z_{\lambda^*}).$$
  The result now follows from Proposition \ref{zdecomp} after inducting  on  $|\lambda|$.
\end{proof}

\subsection{Indecomposable summands of mixed tensor powers}\label{mtp}
Let $V$ denote a $\delta$-dimensional object of a tensor category $\cat T$, let $V^*$ denote a dual for $V$ in $\cat T$, and write
$$T(r,s) = T_{V, V^*} (r,s) =  V^{\otimes r} \otimes {V^*}^{\otimes s}$$
for the \newterm{mixed tensor power} of $V$.
The following theorem gives a useful criterion for the fullness of the functor 
$$ F : \uRep(GL_\delta) \to \cat T, \qquad F(\black) \mapsto V $$
defined by Proposition \ref{uni}, and moreover shows that any indecomposable summand of a $T(r,s)$ is isomorphic to the image under $F$ of an indecomposable object from $ \uRep(GL_\delta)$.
This will be later applied in \S\ref{Fd} and \S\ref{super}, where $V$ will denote the natural representations of the general linear group and the general linear supergroup, respectively.
 
\begin{theorem}\label{fullness}
Suppose that the $\K$-algebra maps
$ \K \Sigma_p \to \End_{\cat T} (V^{\otimes p}) $
defined by the symmetric braiding of $\cat T$ are surjective, and $\Hom_{\cat T} (T(r,s), T(r',s')) = 0$ whenever $r+s' \not = r'+s$.  Then:
\begin{enumerate}
\item F is full.\\
\item $ \Lambda = \{ \ F ( L(\lambda) ) \  | \ \lambda \text{ is a bipartition, } F ( L(\lambda) ) \ne 0 \ \} $
is a complete set of indecomposable summands of the mixed tensor powers $T(r,s)$.  Moreover, the members of $\Lambda$ are pairwise non-isomorphic.
\end{enumerate}
\end{theorem}

\begin{proof}
To prove the first part, it suffices to show the restriction of $F$ to 
$\uRep_0(GL_{\delta})$ is full.  Since every object in $\uRep_0(GL_{\delta})$ is isomorphic (by braidings) to a word of the form $w_{r,s}$ for some $r,s$, and the restriction of $F$ to $\uRep_0(GL_{\delta})$ is strict (hence it preserves braidings) it suffices to show 
\begin{equation}\label{SurF}
F:\Hom_{\uRep_0(GL_{\delta})}(w_{r,s},w_{r',s'})\to\Hom_{\cat T}(T(r,s),T(r',s'))
\end{equation} 
is surjective for every $r,r',s,s'$.  
By hypothesis, we can assume $r+s'=r'+s$.  
Consider the diagram 
\begin{equation}\label{commdiag}
\xymatrix{\Hom_{\uRep_0(GL_\delta)}(w_{r,s},w_{r',s'})\ar[r]\ar[d]_{F} & \End_{\uRep_0(GL_\delta)}(\black^{\otimes r+s'})\ar[d]_{F}\cr
\Hom_{\cat T}(T(r,s),T(r',s')) \ar[r] & \End_{\cat T}(V^{\otimes r+s'})}
\end{equation} where the horizontal maps are given by 
\begin{equation}\label{hor1}
f\mapsto(\id_\black^{\otimes r'}\otimes\ev_{w_{s',0}}\otimes\id_\black^{\otimes s})(f\otimes c_{w_{s,0},w_{s',0}})(\id_\black^{\otimes r}\otimes \coev_{w_{0,s}}\otimes\id_\black^{\otimes s'})
\end{equation} and 
\begin{equation}\label{hor2}
f\mapsto(\id_V^{\otimes r'}\otimes\ev_{T(s',0)}\otimes\id_V^{\otimes s})(f\otimes c_{T(s,0),T(s',0)})(\id_V^{\otimes r}\otimes \coev_{T(0,s)}\otimes\id_V^{\otimes s'}).
\end{equation}  
Since the restriction of $F$ to $\uRep_0(GL_\delta)$ is a strict tensor functor, the diagram (\ref{commdiag}) commutes.  Moreover, the maps (\ref{hor1}) and (\ref{hor2}) are $\K$-vector space isomorphisms with inverses 
\begin{equation*}
f\mapsto(\id_\black^{\otimes r'}\otimes\ev_{w_{0,s}}\otimes\id_\white^{\otimes s'})(f\otimes c_{w_{0,s'},w_{0,s}})(\id_\black^{\otimes r}\otimes \coev_{w_{s',0}}\otimes\id_\white^{\otimes s})
\end{equation*} and 
\begin{equation*}
f\mapsto(\id_V^{\otimes r'}\otimes\ev_{T(0,s)}\otimes\id_{V^*}^{\otimes s})(f\otimes c_{T(0,s'),T(0,s)})(\id_V^{\otimes r}\otimes \coev_{T(s',0)}\otimes\id_{V^*}^{\otimes s})
\end{equation*}
respectively (this is easily verified using diagram calculus for tensor categories  \cite{Selinger}).  
Since the rightmost vertical map is surjective, by hypothesis, we are done.

We now prove the second part.
A summand $W$ of $T(r,s)$ is the image of an idempotent in $\End_{\cat T}(T(r,s))$.  
By part (1), such an idempotent has a pre-image $e\in B_{r,s}$ under $F$.  Write 
$\im e=L(\lambda^{(1)})\oplus\cdots\oplus L(\lambda^{(k)})$
in $\uRep(GL_\delta)$ for some bipartitions $\lambda^{(i)}$ (see Theorem \ref{indeclass}).  
Then $W=F(L(\lambda^{(1)}))\oplus\cdots\oplus F(L(\lambda^{(k)}))$ in $\cat T$, and by Proposition \ref{distinctind} the $F(L(\lambda^{(i)}))$ are indecomposable.  
Thus if $W$ is indecomposable, then $W=F(L(\lambda^{(i)}))$ for some $i$.
\end{proof}

\subsection{Generic semisimplicity of $\uRep(GL_\delta)$}  We close this section with the following theorem which tells us exactly when the category $\uRep(GL_\delta)$ is semisimple\footnote{Recall that a $\K$-linear Krull Schmidt category $\cat{C}$ is semisimple if and only if $\End_\cat{C}(L)$ is a finite $\K$-dimensional division algebra for all indecomposable objects $L$ and $\Hom_\cat{C}(L,L')=0$ for all non-isomorphic indecomposable objects $L, L'$.}

\begin{theorem}\label{ss}
$\uRep(GL_\delta)$ is semisimple if and only if $\delta$ is not an integer.
\end{theorem}

\begin{proof}   The semisimplicity of the walled Brauer algebras is completely determined in  \cite[Theorem 6.3]{CDDM}.  In particular, if $\delta\in\Z$ then $B_{r,s}(\delta)$ is not semisimple for some $r$ and $s$.  In this case, using Theorem \ref{idempclass}, there exist distinct bipartitions $\lambda\vdash(r-i,s-i)$ and $\mu\vdash(r-j,s-j)$ for some $i, j\geq 0$ with $e_\mu^{(j)}B_{r,s}e_\lambda^{(i)}\not=0$.  Since $e_\mu^{(j)}B_{r,s}e_\lambda^{(i)}=\Hom(\im e_\lambda^{(i)}, \im e_\mu^{(j)})=\Hom(L(\lambda),L(\mu))$, we are done in case $\delta\in\Z$.  

Now assume $\delta\not\in\Z$, and let $\lambda$ and $\mu$ be bipartitions with $\Hom(L(\lambda),L(\mu))\not=0$.  If $\lambda\vdash(r,s)$, then $\mu\vdash(r-i,s-i)$ for some $i$ (Proposition \ref{HomLs}).  We proceed under the assumption $i\geq 0$, 
the case $i \leq 0$ being dual, by Proposition \ref{Ldual}.
Then 
\begin{equation}\label{someHom}
0\not=\Hom(L(\lambda),L(\mu))=e_\mu^{(i)}B_{r,s}e_\lambda.
\end{equation}  Since $\delta\not\in\Z$,  $B_{r,s}$ is a semisimple algebra, hence (\ref{someHom}) can only be true if $e_\lambda$ and $e_\mu^{(i)}$ are conjugate, which implies $\lambda=\mu$ (Theorem \ref{idempclass}).  Moreover, the semisimplicity of $B_{r,s}$ implies  $\End(L(\lambda))=e_\lambda B_{r,s}e_\lambda$ is indeed a division algebra.
\end{proof}

\section{Connection to representations of the general linear group}\label{Fd}

Fix a nonnegative integer $d$ and consider the category $\Rep(GL_d)$ of finite dimensional representations of the general linear group $GL_d$ over $\K$.   In this short section we describe how the categories $\uRep(GL_\delta)$ and $\Rep(GL_d)$ are related.  One can view this section as a preview of \S\ref{super} where we show how $\uRep(GL_\delta)$ is related to representations of general linear supergroups.  However, we will see in \S\ref{tensordecomp} that formulas for decomposing tensor products in $\uRep(GL_\delta)$ can be obtained by interpolating decomposition formulas in $\Rep(GL_d)$; hence this preview of \S\ref{super} is also important to the structure of the paper.

\subsection{The category $\Rep(GL_d)$}\label{GenLin}
Let $GL_d = GL(d, \K)$ denote the general linear group, that is, the group of invertible $d \times d$-matrices with entries from $\K$, and write $V$ for the $d$-dimensional natural module.
Let $\Rep (GL_d)$ denote the category of finite-dimensional $GL_d$-modules.
The tensor product of $GL_d$-modules defined in the usual way, and $\unit$ denoting the trivial one-dimensional module, $\Rep(GL_d)$ becomes a monoidal category as per example \ref{vectismonoidal}.
Note that, in particular, we consider the monoidal category $\Rep(GL_d)$ to be strict.

For any object $U$ of $\Rep(GL_d)$, let $U^*$, $\ev_U$, $\coev_U$ be defined as per example \ref{vectistensor}.
Recall that $U^*$ is a $GL_d$-module via
\begin{equation}\label{dualaction}(x \cdot \phi) (u) = \phi (x^{-1} \cdot u),\qquad(x\in GL_d, u \in U, \phi \in U^*)\end{equation}
Then $\ev_U$, $\coev_U$ are $GL_d$-module maps, and so $(U^*, \ev_U, \coev_U)$ is a dual for $U$ in $\Rep(GL_d)$.
Thus $\Rep(GL_d)$ is a tensor category.
It is well known that $\Rep(GL_d)$ is semisimple (see, e.g. \cite{FH}).
In particular, an object is simple if and only if it is indecomposable.

Let $\ep_i$ denote the function that takes any matrix to the its $(i,i)$ entry for $i = 1, \ldots d$, and let $\weights$ denote the set of weights.  
For $k \in \Z$, letting
$$ \weights_k = \{ \gamma = \sum_{i=1}^d \gamma_i \ep_i | \gamma_i \in \Z, \gamma_1 \geqslant \cdots \geqslant \gamma_d, \sum_i \gamma_i = k \},$$
we have $\weights = \sqcap_{k \in \Z} \weights_k$.
Write $\gamma_i$ for the coefficients of $\gamma \in \weights$ with respect to the $\ep_i$.
For $\gamma \in \weights$, let $V(\gamma)$ denote the finite-dimensional highest-weight $GL_d$-module of highest-weight $\gamma$.
The highest-weight modules classify the indecomposable objects of $\Rep(GL_d)$ up to isomorphism.
That is, any indecomposable object of $\Rep(GL_d)$ is isomorphic to $V_\gamma$ for some $\gamma \in \weights$, and if $V_{\gamma} \cong V_{\gamma'}$ for $\gamma, \gamma' \in \weights$, then $\gamma = \gamma'$.
The weights $\weights$ are in bijection with bipartitions $\lambda$ with $l(\lambda) \leq d$, via
\begin{equation}\label{wt}wt(\lambda)=\sum_{i>0}\lambda^\black_i\ep_i-\sum_{j>0}\lambda^\white_j\ep_{d-j+1}.\end{equation}
Writing $V_\lambda$ for $V_{wt(\lambda)}$, the isomorphism classes of indecomposable objects in $\Rep(GL_d)$ are thus parameterized by such bipartitions.

For $k \in \Z$, define a partial order $\leqslant$ on $\weights_k$ by declaring
$$ \gamma \leqslant \gamma' \quad \Leftrightarrow \quad \gamma_1 + \cdots \gamma_k \leqslant \gamma'_1 + \cdots \gamma'_k, \quad k = 1, \ldots, d.$$
If $\gamma \in \weights$, then there exists $k$ such that all weights of the highest-weight module $V(\gamma)$ belong to $\weights_k$ ($k$ is the rational degree of the module).
Thus the set of the weights of $V(\gamma)$ are partially ordered.
If $\gamma'$ is a weight of $V(\gamma)$, then $\gamma' \leq \gamma$.

\begin{proposition} \label{dualwts} 
Let $\lambda$ be a bipartition with $l(\lambda) \leq d$.  Then $(V_\lambda)^*=V_{\lambda^*}$.
\end{proposition}
\begin{proof}
Duality defines an endofunctor of $\Rep(GL_d)$.
Thus $(V_\lambda)^*$ is indecomposable, so $(V_\lambda)^* \cong V_\mu$ for some bipartition $\mu$ with $l(\mu) \leq d$.
It follows immediately from \eqref{dualaction} that the dual of a weight space of $V_\lambda$ is a weight space of $(V_\lambda)^*$ with the weight negated.
Negation inverts the partial order on weights, so $wt(\mu)$ is the lowest weight of $V_\lambda$.

Recall that the symmetric group $\Sigma_d$ embeds in $GL_d$ as the permutation matrices.
Thus $\Sigma_d$ acts on $V_\lambda$ and hence on its set of weights.
If $\sigma$ denotes the longest element of $\Sigma_d$, and $\gamma$ is a weight, then $(\gamma^\sigma)_i = \gamma_{d-i+1}$ for $i=1,\ldots d$, and 
furthermore, $wt(\nu^*) = wt(\nu)^\sigma$ for any bipartition $\nu$.
Finally, $\sigma$ is an anti-involution of the poset of weights, and so $wt(\lambda)^\sigma = wt(\mu)$ is the lowest weight of $V_\lambda$, and so $\mu = \lambda^*$.
\end{proof}

\begin{theorem}\label{Koike}  (Compare with \cite[Theorem 2.4]{Koike})  Fix bipartitions $\lambda\vdash(r,s)$ and $\mu\vdash(r',s')$ such that $l(\lambda), l(\mu)\leq d$.  For each bipartition $\nu$ with $l(\nu)\leq d$ let $\Gamma^\nu_{\lambda,\mu}$ be such that 
\begin{equation*}
V_\lambda\otimes V_\mu=\bigoplus_{\nu}V_\nu^{\oplus \Gamma^\nu_{\lambda,\mu}}.
\end{equation*}  Then $\Gamma_{\lambda,\mu}^\nu=0$ unless the $|\nu|\leq(r+r', s+s')$.  Moreover, if $l(\lambda)+l(\mu)\leq d$ then
\begin{equation}\label{Gamma}
\Gamma^\nu_{\lambda,\mu}=\sum_{\alpha,\beta,\eta,\theta\in\cat{P}}\left(\sum_{\kappa\in\cat{P}}LR^{\lambda^\black}_{\kappa,\alpha}LR^{\mu^\white}_{\kappa,\beta}\right)\left(\sum_{\gamma\in\cat{P}}LR^{\lambda^\white}_{\gamma,\eta}LR^{\mu^\black}_{\gamma,\theta}\right)LR^{\nu^\black}_{\alpha,\theta}LR^{\nu^\white}_{\beta,\eta}
\end{equation} where $LR^\alpha_{\beta,\gamma}$'s are the Littlewood Richardson coefficients.  
\end{theorem}

\subsection{The functor $F_d:\uRep(GL_d)\to\Rep(GL_d)$}\label{functorFd}  
Write $F_d:\uRep(GL_d)\to\Rep(GL_d)$ for the tensor functor which sends $\black\mapsto V$ defined by Proposition \ref{uni} . 
By classical Schur-Weyl duality (see \cite{WeylClassical}), the $\K$-algebra map $\K \Sigma_p \to \End_{\Rep(GL_d)}(V^{\otimes p})$ defined by the symmetric braiding is surjective, for any $p \geq 0$.
Write $T(r,s)$ for the mixed tensor powers of $V$ as per section \S\ref{mtp}.
For any $\zeta \in \K$, the central element $\zeta \cdot \id \in GL_d$ acts on $T(r,s)$ by the scalar $\zeta^{r-s}$, so
$\Hom_{\Rep(GL_d)} (T(r,s), T(r',s') )= 0$ unless $r+s'=r'+s$.
Thus, by Theorem \ref{fullness}, the functor $F_d$ is full (this is the so-called First Fundamental Theorem of invariant theory) and any indecomposable summand of a mixed tensor powers $T(r,s)$ is isomorphic to $F_d (L(\lambda))$ for some bipartition $\lambda$.

Given a bipartition $\lambda$, we write $W(\lambda)=F_d(L(\lambda))$.  In this section we will completely describe $W(\lambda)$.  We start by assuming one of $\lambda^\black$, $\lambda^\white$ is $\emp$.

\begin{proposition}\label{leftW} Assume $\lambda\vdash(r,s)$ with $rs=0$.  If $l(\lambda)\leq d$, then $W(\lambda)=V_\lambda$.  If $l(\lambda)>d$, then $W(\lambda)=0$.
\end{proposition}



\begin{proof}  First, $\Sigma_r$ acts on $V^{\otimes r}$ by permuting tensors.  Since strict tensor functors preserve symmetric braidings, this action coincides with  $F_d:\KS_r\to\End(V^{\otimes r})$.   Hence if $r=0$ then, by Example \ref{elamexs}(1), $W(\lambda)$ is the image of the idempotent $z_{\lambda^\black}\in\KS_r$ acting on $V^{\otimes r}$.  This is precisely \emph{Weyl's construction} of $V_\lambda$ (see for example \cite{FH}).  The case $s=0$ follows from the case $r=0$ using  Propositions \ref{Ldual} and \ref{dualwts}.
\end{proof}

In fact, Proposition \ref{leftW} holds without the assumption that $rs=0$.

\begin{theorem}\label{Fdim} Suppose $\lambda$ is an arbitrary bipartition.  Then 
\begin{equation*}
W(\lambda)=\left\{\begin{array}{ll}
V_\lambda & \text{if }l(\lambda)\leq d,\\
0 & \text{if }l(\lambda)>d.
\end{array}\right.
\end{equation*}

\end{theorem}

\begin{proof} We induct on the size of $\lambda$.  The base case $\lambda=\biemp$ is clear, so assume $\lambda\vdash(r,s)$ with $rs\not=0$.  By Example \ref{elamexs}(1) we have $z_\lambda=e_{(\lambda^\black,\emp)}\otimes e_{(\emp,\lambda^\white)}$ from which it follows \begin{equation}\label{Fz1}F_d(\im z_\lambda)=W((\lambda^\black,\emp))\otimes W((\emp,\lambda^\white)).\end{equation}
Also, by Proposition \ref{zdecomp} \begin{equation}\label{Fz2}F_d(\im z_\lambda)=W(\lambda)\oplus W(\mu^{(1)})\oplus\cdots\oplus W(\mu^{(k)})\end{equation} where $\mu^{(1)},\ldots,\mu^{(k)}$ are bipartitions whose sizes are strictly smaller than $(r,s)$.

Assume $l(\lambda)\leq d$.  Then by (\ref{Fz1}) and Proposition \ref{leftW}, $F_d(\im z_\lambda)$ has a highest weight vector with weight $wt(\lambda)$.  On the other hand, since $l(\lambda)\leq d$ there is no cancelation in (\ref{wt}).  Hence, by induction,  $W(\mu^{(j)})$ does not have a highest weight vector of weight $wt(\lambda)$ for any $j=1,\ldots, k$.  Thus, by (\ref{Fz2}), $W(\lambda)$ has a highest weight vector with weight $wt(\lambda)$.  By Proposition \ref{distinctind}, $W(\lambda)$ is simple and we are done.

Now assume $l(\lambda)>d$.  If either $l(\lambda^\black)$ or $l(\lambda^\white)$ is greater than $d$ we are done by Proposition \ref{leftW}, hence we can assume $l(\lambda^\black),l(\lambda^\white)\leq d$.  By Theorem \ref{Koike}, the highest weights in (\ref{Fz1}) are all of the form $wt(\nu)$ with $|\nu|\lneqq(r,s)$.  Thus, by induction, $F_d(z_\lambda)$ decomposes as a direct sum of $W(\nu)$'s with $|\nu|\lneqq(r,s)$.  Hence, by (\ref{Fz2}) and Proposition \ref{distinctind}, $W(\lambda)=0$.
\end{proof}


\section{The lifting map}\label{lift}  As before, we fix $\delta\in\K$.  Let $t$ be an indeterminate.   In this section we construct a ring isomorphism (called the \emph{lifting map}) between the Grothendieck rings of $\uRep(GL_\delta)$ and $\uRep(GL_t)$.  The definition of the lifting map is not an explicit one, however we show in \S\ref{computelift} that values of the lifting map can be computed using combinatorics of certain diagrams introduced by Brundan and Stroppel.  We will see later (\S\ref{tensordecomp}) that the lifting map will play a crucial role in our ability to decompose tensor products in $\uRep(GL_\delta)$.  We begin by fixing some notation for the Grothendieck rings mentioned above.

\subsection{The rings $R_\delta$ and $R_t$}\label{R}  Write $\K(t), \K[[t-\delta]]$, and $\K((t-\delta))$ for the field of fractions in $t$, ring of power series in $t-\delta$, and the field of Laurent series\footnote{In other words, $\K((t-\delta))$ is the field of fractions of $\K[[t-\delta]]$.} in $t-\delta$ respectively.   Then we have the following categories with their corresponding additive Grothendieck rings (see \S\ref{addGroth}):
$$\begin{array}{c|c|c}
\text{{\bf category}} & \begin{array}{c}{\bf additive }\\\text{{\bf Grothendieck ring}}\end{array} &\begin{array}{c}{\bf bilinear }\\\text{{\bf form}}\end{array} \\ \hline
\begin{array}{cl}
\uRep(GL_\delta) & \text{ over }\K \\
\uRep(GL_t) & \text{ over }\K(t)\\
\uRep(GL_t) & \text{ over }\K((t-\delta)) 
\end{array} &
\begin{array}{c}
R_\delta\\
 R_t\\
 R_{t, \delta}
 \end{array} & \begin{array}{c}
(-,-)_\delta\\
 (-,-)_t\\
 ~\hspace{.1in}~(-,-)_{t, \delta}
 \end{array}
\end{array}$$
By Theorem \ref{indeclass} we can identify the elements of $R_\delta, R_t,$ and $R_{t,\delta}$ with formal $\Z$-linear combinations of bipartitions, and we will do so for the rest of the paper.  In particular, the rings $R_\delta, R_t,$ and $R_{t,\delta}$ are clearly isomorphic as abelian groups.  However, the multiplication in these rings depends on the parameter:  

\begin{example}\label{boxempempbox} It is always true that $e_{(\Box,\emp)}=\id_\black$ and $e_{(\emp,\Box)}=\id_\white$, which implies $L((\Box,\emp))=\black$ and $L((\emp,\Box))=\white$.  Hence $L((\Box,\emp))\otimes L((\emp,\Box))=\black\white$ always.  Thus, it follows from example \ref{elamexs}(2) that $$(\Box,\emp)(\emp,\Box)=(\Box,\Box)\in R_0,\quad\text{whereas}\quad(\Box,\emp)(\emp,\Box)=(\Box,\Box)+\biemp\in R_t.$$
\end{example}

The next proposition shows that although the multiplication of bipartitions in $R_\delta$ and $R_t$ may differ, the rings $R_t$ and $R_{t,\delta}$ can be identified regardless of $\delta$.

\begin{proposition}\label{Rts} (1) The $\Z$-linear map $R_t\to R_{t,\delta}$ with $\lambda\mapsto\lambda$ for each bipartition $\lambda$ is a ring isomorphism. 

(2) $(\lambda,\mu)_t=(\lambda,\mu)_{t,\delta}=\left\{\begin{array}{ll} 1 & \text{if }\lambda=\mu,\\ 0 & \text{if } \lambda\not=\mu.\end{array}\right.$ 
\end{proposition}

\begin{proof}  (1) Suppose \begin{equation}\label{decint}e_\lambda\otimes e_\mu=e_1+\cdots+e_k\end{equation} is a decomposition of $e_\lambda\otimes e_\mu$ into mutually orthogonal primitive idempotents over $\K(t)$.  Then  $\lambda\mu=\sum_\nu a_\nu\nu\in R_t$ where $a_\nu$ is the number of summands in (\ref{decint}) correspond to the bipartition $\nu$.  By Corollary \ref{absolute}, viewing (\ref{decint})  over the larger field $\K((t-\delta))\supset\K(t)$ still gives an orthogonal decomposition of $e_\lambda\otimes e_\mu$ into primitive idempotents, hence $\lambda\mu=\sum_\nu a_\nu\nu\in R_{t,\delta}$ too.  

(2) By Corollary \ref{absolute}, we can work over the algebraic closure of $\K(t)$ (resp. $\K((t-\delta))$ to compute $(\lambda,\mu)_t$ (resp. $(\lambda,\mu)_{t,\delta}$).  The result now follows from the fact that $\uRep(GL_t)$ is semisimple over any field containing the indeterminate $t$ (see Theorem \ref{ss}).
\end{proof}

With Proposition \ref{Rts} in mind, for the rest of the paper we will identify $R_{t,\delta}$ with $R_t$ for every $\delta$ and  write $R_t$ for both.

\subsection{The ring map $\lift_\delta:R_\delta\to R_t$}  Fix a bipartition $\lambda\vdash(r,s)$ and consider the idempotent $e_\lambda\in\K B_{r,s}(\delta)$.  We can lift $e_\lambda$ to an idempotent $\tilde{e}\in \K((t-\delta))B_{r,s}(t)$, i.e. $\tilde{e}$ is of the form $\tilde{e}=\sum_{X}a_XX$ with $a_X\in\K[[t-\delta]]$ for all $(w_{r,s},w_{r,s})$-diagrams $X$, and $\tilde{e}|_{t=\delta}=e$ (see \cite[Theorem A.2]{CO1}).  Now, given another bipartition $\mu$, let $D_{\lambda,\mu}=D_{\lambda,\mu}(\delta)$ denote the number of times $L(\mu)$ occurs in a decomposition of $\im(\tilde{e})$ into a direct sum of  indecomposables in $\uRep(GL_t)$ over $\K((t-\delta))$.  One can show that $D_{\lambda,\mu}$ does not depend on the choice of representative for $e_\lambda$ or on the choice of $\tilde{e}$ (compare with  \cite[Theorem 3.9]{CO1}).  Now, let $\lift_\delta:R_\delta\to R_t$ be the $\Z$-linear map defined on bipartitions by $$\lift_\delta(\lambda)=\sum_\mu D_{\lambda,\mu}\mu.$$

\begin{example}\label{liftemp}    If $\lambda\vdash(r,0)$ then, by Example \ref{elamexs}(1), $e_\lambda=z_{\lambda^\black}^\black\in\K B_{r,0}(\delta)$.  Since $z_{\lambda^\black}^\black$ does not depend on $\delta$, it can be lifted to $z_{\lambda^\black}^\black=e_\lambda\in\K((t))B_{r,0}(t)$.  Hence, $\lift_\delta(\lambda)=\lambda$ for all $\lambda\vdash(r,0)$, $\delta\in\K$.  Similarly, $\lift_\delta(\lambda)=\lambda$ whenever $\lambda\vdash(0,s)$.  
\end{example}

\begin{example}\label{liftex}  

(1) Assume $\delta=0$.  By Example \ref{elamexs}(2), $e_\bibox=\id_{\black\white}\in\K B_{1,1}(0)$ which lifts to $\id_{\black\white}\in\K((t))B_{1,1}(t)$.  By Example \ref{zboxdec}(1), $\black\white=L(\bibox)\oplus L(\biemp)$ in $\uRep(GL_t)$. Thus $\lift_0(\bibox)=\bibox+\biemp$.

(2) Assume $\delta\not=0$.  An explicit expression for $e_\bibox\in\K B_{1,1}(\delta)$ is given in Example \ref{elamexs}(2).  Since $\frac{1}{t}=\sum_{n=0}^\infty\frac{(-1)^n}{\delta^{n+1}}(t-\delta)^n\in\K[[t-\delta]]$, a lift of $e_\bibox$ is obtained by replacing $\delta$ with $t$ in that expression.  Hence, $\lift_\delta(\bibox)=\bibox$.
\end{example}

The following theorem lists properties of $\lift_\delta$ which are very useful for this paper.

\begin{theorem}\label{liftprops} (1) $\lift_\delta:R_\delta\to R_t$ is a ring isomorphism for every $\delta\in\K$.

(2) $D_{\lambda,\lambda}=1$ for all $\lambda$.  Moreover,  $D_{\lambda,\mu}=0$ unless $\mu=\lambda$ or  $\mu\vdash(|\lambda^\black|-i,|\lambda^\white|-i)$ for some $i>0$.

(3) Fix a bipartition $\lambda$.  $\lift_\delta(\lambda)=\lambda$ for all but finitely many $\delta\in\K$.

(4)  $(\lift_\delta(x), \lift_\delta(y))_t=(x,y)_\delta$ for all $x,y\in R_\delta$.  

\end{theorem}
 
For a proof of Theorem \ref{liftprops} we refer the reader to \cite[Proposition 3.12]{CO1} where the analogous statements are proved for $\uRep(S_t)$.

An important consequence of Theorem \ref{liftprops} is the following:

\begin{corollary}\label{deltaform} $(\lambda, \mu)_\delta=\sum_\nu D_{\lambda,\nu}D_{\mu,\nu}$ for all bipartitions $\lambda$ and $\mu$.
\end{corollary}

\begin{proof} $$\begin{array}{rll}
(\lambda,\mu)_\delta & =  (\sum_\nu D_{\lambda,\nu}\nu~,~ \sum_{\nu'} D_{\mu,\nu'}\nu')_t & (\text{Theorem \ref{liftprops}(4)})\\[2pt]
& = \sum_{\nu,\nu'}D_{\lambda, \nu}D_{\mu,\nu'}(\nu,\nu')_t\\[2pt]
& = \sum_{\nu}D_{\lambda, \nu}D_{\mu,\nu} & (\text{Proposition \ref{Rts}(2)})
\end{array}$$
\end{proof}

\subsection{The diagrams of Brundan and Stroppel}\label{BS}  We will soon show that $\lift_\delta(\lambda)$ can be computed explicitly using certain diagrams introduced by Brundan and Stroppel [BS2-5]\nocite{BS1,BS2,BS3,BS4}.  In this subsection we introduce these diagrams and give a few examples.

As usual, we fix $\delta\in\K$.  Given a bipartition $\lambda$, set $$\begin{array}{rl}
I_\up(\lambda) & =\{\lambda^\black_1, \lambda^\black_2-1, \lambda^\black_3-2, \ldots\},\\
I_\down(\lambda,\delta) & =\{1-\delta-\lambda^\white_1, 2-\delta-\lambda^\white_2, 3-\delta-\lambda^\white_3, \ldots\}.
\end{array}$$  Now, let $x_\lambda=x_\lambda(\delta)$ be the diagram obtained by labeling the integer vertices on the number line according to the following rule:  label the the $i$th vertex by $$\left\{\begin{array}{cl}
\bigo & \text{ if }i\not\in I_\up(\lambda)\cup I_\down(\lambda,\delta),\\
\up & \text{ if }i\in I_\up(\lambda)\setminus I_\down(\lambda,\delta),\\
\down & \text{ if }i\in I_\down(\lambda,\delta)\setminus I_\up(\lambda),\\
\cross & \text{ if }i\in I_\up(\lambda)\cap I_\down(\lambda,\delta).
\end{array}\right.$$  
For example, $$\begin{array}{r}
\includegraphics{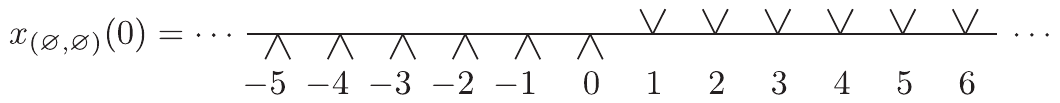},\\
\includegraphics{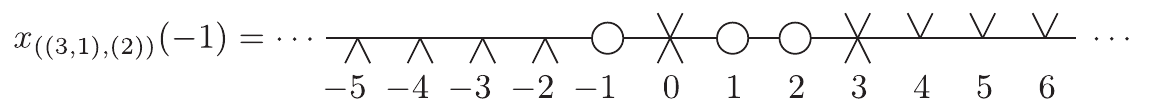},\\
\includegraphics{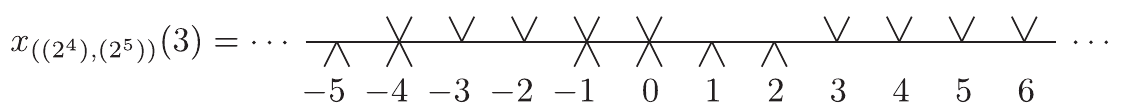},\\
\includegraphics{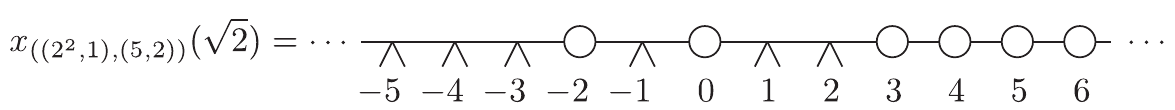}.
\end{array}$$

\begin{remark}\label{endsofx}  Notice that the integer $i$ in $x_\lambda$ is labelled $\up$ for $i\ll0$. Moreover, if $\delta \in \Z$ (resp. $\delta\not \in \Z$) then $i$ is 
labelled by $\down$ (resp. $\bigo$) for $i\gg 0$. 
In fact, it is not difficult to show that when $\delta\in\Z$ there is a bijection between the set of all bipartitions and the set of all diagrams with (1) $i$ labelled $\up$ for $i\ll0$; (2) $i$ labelled $\down$ for $i\gg0$; and (3) the number of $\cross$'s minus the number of $\bigo$'s equal to $\delta$.  
\end{remark}

Next, we construct the \emph{cap diagram} $c_\lambda=c_\lambda(\delta)$ in the following recursive manner:

\begin{enumerate}{\setlength\itemindent{0.35in} \item[\emph{Step 0:}]  Start with $x_\lambda$.

\item[\emph{Step $n$:}] Draw a cap connecting vertices $i$ and $j$ on the number line whenever (i) $i<j$; (ii) $i$ is labelled by $\down$ and $j$ is labelled by $\up$ in $x_\lambda$; and (iii) each integer between $i$ and $j$ in $x_\lambda$ is either labelled by $\bigo$, labelled by $\cross$, or already part of a cap from an earlier step.}
\end{enumerate}
It follows from Remark \ref{endsofx} that no new caps will be added after a finite number of steps, leaving us with the cap diagram $c_\lambda$.

\begin{example}\label{capex}  If $\delta=1$ and $\lambda=((5^2,4^2,3^2), (5^3,4,3,2))$ then $$\includegraphics{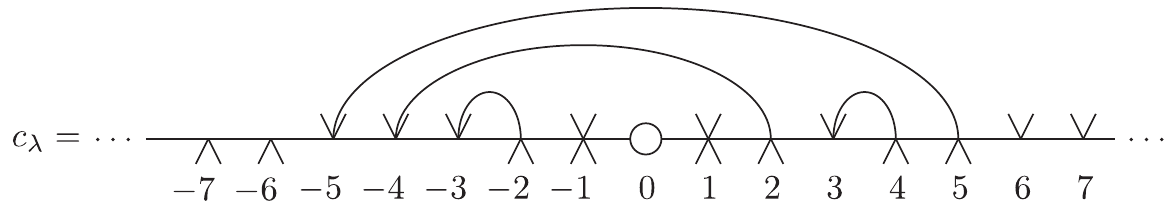}$$
\end{example}

Given integers  $i< j$, we say $(i, j)$ is a \emph{$\down\up$-pair in $x_\lambda$} if there is a cap from $j$ to $i$ in $c_\lambda$.  For instance, in Example \ref{capex} there are four $\down\up$-pairs: $(-5, 5), (-4,2), (-3,-2)$, and $(3,4)$.  Next, given bipartitions $\mu$ and $\lambda$, we say that \emph{$\mu$ is linked to $\lambda$} if there exists an integer $k\geq 0$ and bipartitions $\nu^{(n)}$ for $0\leq n\leq k$ such that  (i) $\nu^{(0)}=\lambda$, (ii) $\nu^{(k)}=\mu$, and (iii) $x_{\nu^{(n)}}$ is obtained from $x_{\nu^{(n-1)}}$ by swapping the labels of some $\down\up$-pair in $x_{\lambda}$ whenever $0< n\leq k$.  Finally, set $$D_{\lambda,\mu}'=D_{\lambda,\mu}'(\delta)=\left\{\begin{array}{ll}
1 & \text{if $\mu$ is linked to $\lambda$},\\
0 & \text{otherwise}.
\end{array}
\right.$$

\begin{remark} It is shown in \cite{CD} that $D'_{\lambda,\mu}$ give decomposition numbers for walled Brauer algebras.  This is easy to see when $\delta\not\in\Z$.  Indeed, when $\delta\not\in\Z$ there are no $\down$ labels on $x_\lambda$, so there are no $\down\up$-pairs; hence $D'_{\lambda,\mu}\not=0$ if and only if $\lambda=\mu$. \end{remark}


\begin{example}\label{exDprime}   Fix $\delta=-1$.  In this example we will compute the numbers $D'_{\lambda,\mu}$ where $\lambda=((3,2),(3,1))$ and $\mu$ is arbitrary.  $$\includegraphics{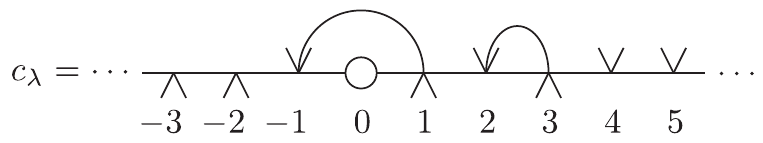}$$
Now we swap labels on  $\down\up$-pairs in $x_{\lambda}$ to determine which bipartitions $\mu$ are linked to $\lambda$.  The following table lists our results:
$$\begin{array}{|c||c||c|}\hline
\begin{array}{c}\down\up\text{-pairs}\\ \text{swapped}\end{array} & x_\mu & \mu\\ \hline\hline
\text{none} & \includegraphics{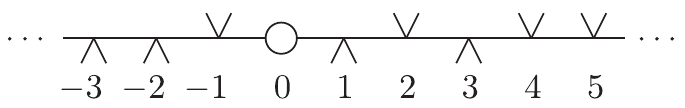} & ((3,2),(3,1))\\\hline\hline
(-1,1) &  \includegraphics{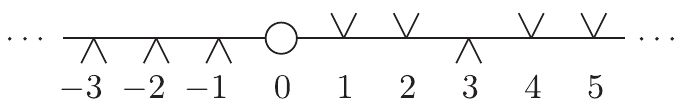} & ((3), (1^2))\\\hline\hline
(2,3) & \includegraphics{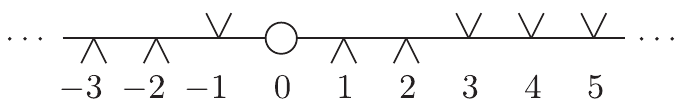} & ((2^2), (3))\\\hline\hline
\text{both} & \includegraphics{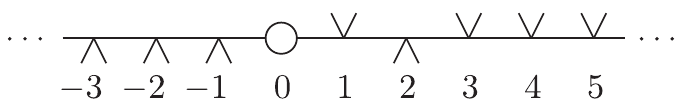} & ((2), (1))\\\hline\hline
\end{array}$$
Hence $D'_{\lambda,\mu}=1$ when $\mu$ is one of the four bipartitions listed in the table above, and $D'_{\lambda,\mu}=0$ for all other $\mu$.  
\end{example}

\begin{example}\label{biboxDprime}  In this example we compute $D_{\bibox,\mu}'(\delta)$ for all bipartitions $\mu$ and all $\delta\in\K$.  Since $I_\up(\bibox)=\{1,-1,-2,\ldots\}$, the diagram $x_\bibox(\delta)$ has a $\down\up$-pair if and only if $0\in I_\down(\bibox,\delta)$ and $1\not\in I_\down(\bibox,\delta)$, which occurs if and only if $\delta=0$ since $I_\down(\bibox,\delta)=\{-\delta,2-\delta,3-\delta,\ldots\}$.  Hence, when $\delta\not=0$ we have $$D'_{\bibox,\mu}(\delta)=\left\{\begin{array}{ll}
1 & \text{if }\mu=\bibox,\\
0 & \text{otherwise}.
\end{array}\right.$$  On the other hand, $$\includegraphics{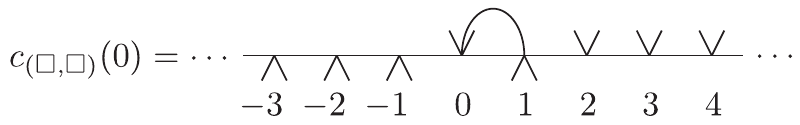}$$  Swapping the $\down\up$-pair $(0,1)$ in $x_\bibox(0)$ gives $x_\biemp(0)$.  Hence $$D'_{\bibox,\mu}(0)=\left\{\begin{array}{ll}
1 & \text{if }\mu=\bibox\text{ or }\mu=\biemp,\\
0 & \text{otherwise}.
\end{array}\right.$$  In particular, $D'_{\bibox,\mu}=D_{\bibox,\mu}$ for all $\mu$ regardless of $\delta$ (see Example \ref{liftex}).  In \S\ref{computelift} we will show $D'_{\lambda,\mu}=D_{\lambda,\mu}$ always (see Corollary \ref{D}).
\end{example}

The following proposition describes how swapping the labels on $\down\up$-pairs affects the size of the corresponding bipartitions.  

\begin{proposition}\label{swapsize}  Suppose $\lambda$ and $\mu$ are bipartitions and $(i,j)$ is a $\down\up$-pair in $x_\lambda$.  If $x_\mu$ is obtained from $x_\lambda$  by swapping the labels of $i$ and $j$, then $$|\mu|=(|\lambda^\black|+i-j,|\lambda^\white|+i-j).$$  
\end{proposition}

\begin{proof}  Set $a_k=\lambda^\black_k-k+1$ and $b_k=k-\delta-\lambda^\white_k$ for each $k>0$ so that 
\begin{align*}
I_\up(\lambda) & = \{a_1,a_2,a_3,\ldots\},\\
I_\down(\lambda,\delta) & =\{b_1,b_2,b_3,\ldots\}.
\end{align*}  Then $|\lambda^\black|=\sum_{k>0}(a_k+k-1)$ and $|\lambda^\white|=\sum_{k>0}(k-\delta-b_k)$.  Now let $L,M\in\Z$ be such that $a_M=j$ and $b_L=i$.  Then swapping the labels of $i$ and $j$ in $x_\lambda$ results in $x_\mu$ with 
\begin{align*}
I_\up(\mu)&=\{a_1,\ldots,a_{M-1},a_{M+1},\ldots,a_N,b_L,a_{N+1},a_{N+2},\ldots\},\\
I_\down(\mu,\delta) & =\{b_1,\ldots,b_{L-1},b_{L+1},\ldots,b_{N'},a_M,b_{N'+1},b_{N'+2},\ldots\}
\end{align*} for some $N, N'$.  Hence 
\begin{align*}
|\mu^\black| & = \sum_{0<k<M\atop\text{or } k>N}(a_k+k-1)+\sum_{M<k\leq N}(a_k+(k-1)-1)+(b_L+N-1)\\
&= \sum_{ k>0}(a_k+k-1)-(a_M+M-1)-(N-M)+(b_L+N-1)\\
&=|\lambda^\black|+i-j.
\end{align*}
Similarly,
\begin{align*}
|\mu^\white| & = \sum_{0<k<L\atop\text{or } k>N'}(k-\delta-b_k)+\sum_{L<k\leq N'}(k-1-\delta-b_k)+(N'-\delta-a_M)\\
&= \sum_{ k>0}(k-\delta-b_k)-(L-\delta-b_L)-(N'-L)+(N'-\delta-a_M)\\
&=|\lambda^\white|+i-j.
\end{align*}
\end{proof}

The next corollary, an immediate consequence of Proposition \ref{swapsize}, will be useful later.

\begin{corollary}\label{Dprimeunitri}  $D'_{\lambda,\lambda}=1$ for all $\lambda$.  Moreover,  
$D_{\lambda,\mu}'=0$ unless $\mu=\lambda$ or $|\mu|=(|\lambda^\black|-i,|\lambda^\white|-i)$ for some $i> 0$.
\end{corollary}

\subsection{Computing $\lift_\delta(\lambda)$}\label{computelift}  The following theorem will allow us  to explicitly compute values of the lifting map using the combinatorics developed in \S\ref{BS} (see Corollary \ref{D}).  Our proof of this theorem relies heavily on the results in \cite{CD}.

\begin{theorem}\label{DD} $\sum_\nu D_{\lambda,\nu}'D_{\mu,\nu}'=(\lambda,\mu)_\delta$ for all bipartitions $\lambda, \mu$ and all $\delta\in\K$.
\end{theorem}

\begin{proof}  It follows from Corollary \ref{deltaform} that the statement of theorem is symmetric in $\lambda$ and $\mu$, hence we may assume $|\mu|\not>|\lambda|$.  Suppose $\lambda\vdash(r,s)$ and $\mu\vdash(r',s')$.  $(\lambda,\mu)_\delta$ is zero unless $r+s'=r'+s$ (Proposition \ref{HomLs}).  Hence, by Corollary \ref{Dprimeunitri}, it suffices to consider the case $\mu\vdash(r-i,s-i)$ for some $i\geq 0$.  

Assume $\delta\not=0$.   Then $B_{r,s}$ is a quasi-hereditary (hence cellular) algebra (see \cite[Corollary 2.8]{CDDM}) with decomposition numbers given by $D'_{\lambda,\nu}$ (see \cite[Theorem 4.10]{CD}\footnote{The results in \cite{CD} are only proved for $\K=\mathbb{C}$,  however, using Corollary \ref{absi} it can be shown that their results hold over arbitrary fields of characteristic zero.}).  In particular, this implies that the $B_{r,s}$-module homomorphisms between the projective modules $e_\lambda B_{r,s}$ and $e_\mu^{(i)} B_{r,s}$ satisfy the following:
 \begin{equation}\label{DDdim}
 \sum_\nu D_{\lambda,\nu}'D_{\mu,\nu}'=\dim_\K\Hom_{B_{r,s}}(e_\lambda B_{r,s}, e_\mu^{(i)} B_{r,s})
\end{equation} 
(see, for instance \cite[Theorem 3.7(iii)]{GL}).  Since 
\begin{equation}\label{LB}
\Hom_{B_{r,s}}(e_\lambda B_{r,s},e_\mu^{(i)} B_{r,s})=e_\mu^{(i)}B_{r,s}e_\lambda=\Hom_{\uRep(GL_\delta)}(L(\lambda),L(\mu)),
\end{equation} 
it follows that $(\lambda,\mu)_\delta$ agrees with (\ref{DDdim}).

If $\delta=0$, the algebra $B_{r,s}$ is no longer quasi-hereditary, but it is still cellular (see \cite[Theorem 2.7]{CDDM}). The decomposition numbers are still given by $D_{\lambda,\nu}'$, however there is no PIM labelled by $\biemp$ in this case, hence we must require $\lambda\not=\biemp$.  Hence, if neither $\lambda$ nor $\mu$ is $\biemp$, (\ref{DDdim}) and (\ref{LB}) still hold we are done as before.  Since we are assuming $|\mu|\not>|\lambda|$, to complete the proof of the theorem we only need to prove the case $\mu=\biemp$, $\lambda\vdash(r,r)$ and $\delta=0$.  Since $D'_{\biemp,\nu}=0$ whenever $\nu\not=\biemp$ and $D'_{\biemp,\biemp}=1$, we have $\sum_\nu D'_{\lambda,\nu}D'_{\biemp,\nu}=D'_{\lambda,\biemp}$.  The decomposition number $D'_{\lambda,\biemp}$ is the composition multiplicity of the simple $B_{r,r}$-module labelled by $\lambda$ in $\Hom(w_{r,r},{\bf 1})$, the standard $B_{r,r}$-module labelled by $\biemp$.  Hence, \begin{equation}\label{Demp}
D_{\lambda,\biemp}'=\dim_\K\Hom_{B_{r,r}}(e_\lambda B_{r,r},\Hom(w_{r,r},{\bf 1})).
\end{equation}  Since $$\Hom_{B_{r,r}}(e_\lambda B_{r,r},\Hom(w_{r,r},{\bf 1}))=\Hom(w_{r,r},{\bf 1})e_\lambda=\Hom_{\uRep(GL_0)}(L(\lambda), L(\biemp)),$$ $(\lambda,\biemp)_0$ agrees with (\ref{Demp}).
\end{proof}

\begin{corollary}\label{D} $D_{\lambda,\mu}(\delta)=D'_{\lambda,\mu}(\delta)$ for all bipartitions $\lambda, \mu$ and all $\delta\in\K$.  
\end{corollary}

\begin{proof} First, put the following partial order on pairs of bipartitions:  $(\lambda,\mu)>(\lambda',\mu')$ means either $|\lambda|>|\lambda'|$, or $\lambda=\lambda'$ and $|\mu|>|\mu'|$.
We prove the corollary by inducting on this partial order.  First notice that $D_{\biemp,\biemp}=1=D'_{\biemp,\biemp}$.  Now assume $(\lambda,\mu)\not=(\biemp,\biemp)$.  By Theorem \ref{liftprops}(2) and Corollary \ref{Dprimeunitri} we may assume $|\lambda|>|\mu|$.  Thus 
$$\begin{array}{rll} 
D_{\lambda,\mu} & = (\lambda, \mu)_\delta-\sum\limits_{\nu\atop|\nu|<|\mu|}D_{\lambda,\nu}D_{\mu,\nu} & (\text{Corollary \ref{deltaform} and Theorem \ref{liftprops}(2)})\\
\\
& = (\lambda, \mu)_\delta-\sum\limits_{\nu\atop|\nu|<|\mu|}D'_{\lambda,\nu}D'_{\mu,\nu} & (\text{Induction})\\
& = D'_{\lambda,\mu} & (\text{Theorem \ref{DD} and Corollary \ref{Dprimeunitri}}).
\end{array}$$
\end{proof}

\begin{example}\label{exLift} Using Corollary \ref{D} and Example \ref{exDprime}  we have $$\lift_{-1}(((3,2),(3,1)))=((3,2),(3,1))+((3),(1^2))+((2^2),(3))+((2),(1)).$$
\end{example}

\section{Decomposing tensor products in \underline{Re}p$(GL_\delta)$}\label{tensordecomp}  In this section we give a generic decomposition formula for decomposing tensor products of indecomposable objects in $\uRep(GL_t)$.  We then show how this generic decomposition formula along with the lifting map from the previous section can be used to decompose arbitrary tensor products in $\uRep(GL_\delta)$.  Throughout this section we will work in the Grothendieck rings $R_\delta$ and $R_t$ (see \S\ref{R}).

\subsection{The generic case}\label{gencase}  The following theorem explains how to decompose the tensor product of two indecomposable objects in $\uRep(GL_t)$.

\begin{theorem}\label{genten}   Given bipartitions $\lambda,\mu$, and $\nu$, let $\Gamma_{\lambda,\mu}^\nu$ be as in (\ref{Gamma}).  Then $\lambda\mu=\sum_\nu\Gamma_{\lambda,\mu}^\nu\nu$ in $R_t$.
\end{theorem}

\begin{proof} Fix bipartitions $\lambda$ and $\mu$ and let $\nu^{(1)},\ldots,\nu^{(k)}$ be bipartitions such that \begin{equation}\label{nus}
\lambda\mu=\nu^{(1)}+\cdots+\nu^{(k)}
\end{equation} in $R_t$.  By Theorem \ref{liftprops}(3) there exists a positive integer $d$ which simultaneously satisfies  (i) $d\geq l(\lambda)+l(\mu)$; (ii) $d\geq l(\nu^{(i)})$ for each $i=1,\ldots, k$; and (iii) $\lift_d$ fixes $\lambda,\mu,\nu^{(1)},\ldots,\nu^{(k)}$.  Now, $\lift_d$ is a ring isomorphism (Theorem \ref{liftprops}(1)), hence (\ref{nus}) holds in $R_d$ by assumption (iii).  Since $F_d$ is a tensor functor, by Theorem \ref{Fdim} along with assumptions (i) and (ii) we have $$V_\lambda\otimes V_\mu=V_{\nu^{(1)}}\oplus\cdots\oplus V_{\nu^{(k)}}$$ in $\Rep(GL_d)$.  The result now follows from Theorem \ref{Koike} and assumption (i).
\end{proof}

The following corollary lists special cases of Theorem \ref{genten}, which are easy to prove using basic properties of Littlewood Richardson coefficients:

\begin{corollary}  The following equations hold in $R_t$.
\begin{equation}
(\lambda^\black,\emp)(\mu^\black,\emp)=\sum_{\alpha\in\cat{P}}LR^\alpha_{\lambda^\black,\mu^\black}(\alpha,\emp),
\end{equation}
\begin{equation}
(\emp, \lambda^\white)(\emp, \mu^\white)=\sum_{\alpha\in\cat{P}}LR^\alpha_{\lambda^\white,\mu^\white}(\emp, \alpha),
\end{equation}
\begin{equation}\label{leftright}
(\lambda^\black,\emp)(\emp, \mu^\white)=\sum_{\nu}\sum_{\kappa\in\cat{P}}LR_{\kappa,\nu^\black}^{\lambda^\black}LR^{\mu^\white}_{\kappa,\nu^\white}\nu,
\end{equation}
\begin{equation}\label{leftbox}
\lambda(\Box,\emp)=\sum_{\lambda^{\black+}}(\lambda^{\black+}, \lambda^\white)+\sum_{\lambda^{\white-}} (\lambda^\black, \lambda^{\white-}),
\end{equation}
\begin{equation}\label{rightbox}
\lambda(\emp,\Box)=\sum_{\lambda^{\white+}} (\lambda^\black, \lambda^{\white+})+\sum_{\lambda^{\black-}}(\lambda^{\black-}, \lambda^\white),
\end{equation}
where the sums in (\ref{leftbox}) are taken over all partitions $\lambda^{\black+}$ (resp. $\lambda^{\white-}$) obtained from the Young diagram $\lambda^\black$ (resp, $\lambda^\white$) by adding one box (resp. removing one box).  Similarly for (\ref{rightbox}).
\end{corollary}

\begin{example}\label{extenst}  In this example we compute $((2),\emp)\bibox\in R_t$.  By Example \ref{boxempempbox} (or using (\ref{leftright}), (\ref{leftbox}), or (\ref{rightbox})) we have $(\Box,\emp)(\emp,\Box)=\bibox+\biemp\in R_t$.  Hence, using (\ref{leftbox}) and (\ref{rightbox}) we have the following in $R_t$:
\begin{align*}
((2),\emp)\bibox & = ((2),\emp)((\Box,\emp)(\emp,\Box)-\biemp)\\
& = ((2),\emp)(\Box,\emp)(\emp,\Box)-((2),\emp)\\
& = (((2,1),\emp)+((3),\emp))(\emp,\Box)-((2),\emp)\\
& = ((2,1),\Box)+((3),\Box)+((1^2),\emp)+((2),\emp).
\end{align*}
\end{example}

\subsection{Decomposing arbitrary tensor products}\label{arbtens}  To compute the product of two bipartitions $\lambda,\mu\in R_\delta$ for arbitrary $\delta\in\K$ (i.e. to decompose tensor products in $\uRep(GL_\delta)$) we \begin{enumerate}
\item determine the coefficients in $\lift_\delta(\lambda\mu)=\sum_{\nu,\nu'}D_{\lambda,\nu}D_{\mu,\nu'}\nu\nu'$ using Corollary \ref{D},

\item use the results in \S\ref{gencase} to expand $\sum_{\nu,\nu'}D_{\lambda,\nu}D_{\mu,\nu'}\nu\nu'=\nu^{(1)}+\cdots+\nu^{(k)}\in R_t$, 

\item determine $\lift_\delta^{-1}(\nu^{(1)}+\cdots+\nu^{(k)})=\lambda\mu$, which by Theorem \ref{liftprops}(2) consists of a sum of a subset of the bipartitions $\nu^{(1)},\ldots,\nu^{(k)}$.
\end{enumerate}
The following examples illustrate the process described above:

\begin{example} Consider $((2^2),(3,1))(\Box,\emp)\in R_{-1}$.  Since 
$$\includegraphics{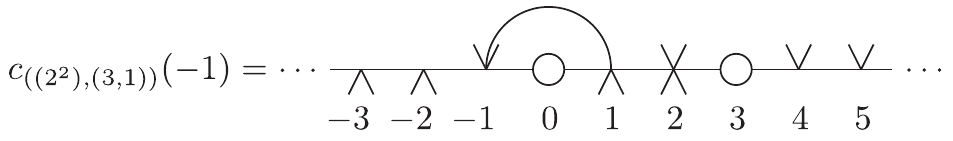}$$ and 
$$\includegraphics{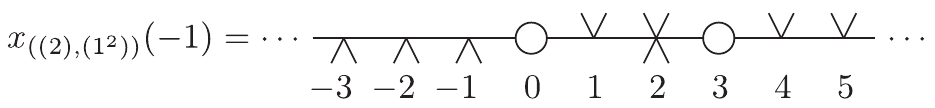}$$
by Corollary \ref{D} we have $\lift_{-1}(((2^2),(3,1)))=((2^2),(3,1))+((2),(1^2))$.  Moreover, by Example \ref{liftemp}, $\lift_{-1}((\Box,\emp))=(\Box,\emp)$.  Since $\lift_{-1}$ is a ring map (Theorem \ref{liftprops}(1)) it follows that 
\begin{equation*}
\lift_{-1}(((2^2),(3,1))(\Box,\emp)) = ((2^2),(3,1))(\Box,\emp)+((2),(1^2))(\Box,\emp),
\end{equation*} which, by (\ref{leftbox}), is equal to 
{\small\begin{equation*}\label{bigsum}
((3,2),(3,1))+((2^2,1),(3,1))+((2^2),(3))+((2^2),(2,1))+((2,1),(1^2))+((3),(1^2))+((2),\Box).
\end{equation*}}
Now, by Example \ref{exDprime} along with Corollary \ref{D}  \begin{equation*}
\lift_{-1}(((3,2),(3,1)))=((3,2),(3,1))+((3),(1^2))+((2^2),(3))+((2),\Box).
\end{equation*}Similarly, one can show $\lift_{-1}(((2^2,1),(3,1)))=((2^2,1),(3,1))$ and 
\begin{equation*}
\lift_{-1}(((2^2),(2,1)))=((2^2),(2,1))+((2,1),(1^2)).
\end{equation*}  Hence,  $$\lift_{-1}(((2^2),(3,1))(\Box,\emp))=\lift_{-1}(((3,2),(3,1))+((2^2,1),(3,1))+((2^2),(2,1))).$$ It follows from Theorem \ref{liftprops}(1) that $$((2^2),(3,1))(\Box,\emp)=((3,2),(3,1))+((2^2,1),(3,1))+((2^2),(2,1))\in R_{-1}.$$
\end{example}

\begin{example}\label{dectensex} In this example we compute $((2),\emp)\bibox\in R_\delta$ for arbitrary $\delta\in\K$.  First, since $\lift_\delta$ is a ring map (Theorem \ref{liftprops}(1)), by Examples \ref{liftemp} and \ref{liftex} we have 
\begin{equation*}
\lift_\delta(((2),\emp)\bibox)=\left\{\begin{array}{ll}
((2),\emp)\bibox+((2),\emp)  & \text{if }\delta=0,\\
((2),\emp)\bibox & \text{if }\delta\not=0.
\end{array}\right.
\end{equation*}
Hence, by Example \ref{extenst} we have 
\begin{equation*}
\lift_\delta(((2),\emp)\bibox)=\left\{\begin{array}{ll}
((2,1),\Box)+((3),\Box)+((1^2),\emp)+2((2),\emp) & \text{if }\delta=0,\\
((2,1),\Box)+((3),\Box)+((1^2),\emp)+((2),\emp) & \text{if }\delta\not=0.
\end{array}\right.
\end{equation*}  Now, $\lift_\delta(((1^2),\emp))=((1^2),\emp)$ and $\lift_\delta(((2),\emp))=((2),\emp)$ for all $\delta$ (Example \ref{liftemp}).  Moreover, using Corollary \ref{D}, we compute 
\begin{equation*}
\lift_\delta(((2,1),\Box))=\left\{\begin{array}{ll}
((2,1),\Box)+((1^2),\emp) & \text{if }\delta=-1,\\
((2,1),\Box)+((2),\emp)  & \text{if }\delta=1,\\
((2,1),\Box)  & \text{if }\delta\not=\pm1.
\end{array}\right.
\end{equation*}
\begin{equation*}
\lift_\delta(((3),\Box))=\left\{\begin{array}{ll}
((3),\Box)+((2),\emp)  & \text{if }\delta=-2,\\
((3),\Box)  & \text{if }\delta\not=-2.
\end{array}\right.
\end{equation*}
Since $\lift_\delta$ is an isomorphism for all $\delta$, we have the following in $R_\delta$: \begin{equation*}
((2),\emp)\bibox=\left\{\begin{array}{ll}
((2,1),\Box)+((3),\Box)+((1^2),\emp)+2((2),\emp) & \text{if }\delta=0,\\
((2,1),\Box)+((3),\Box)+((2),\emp) & \text{if }\delta=-1,\\
((2,1),\Box)+((3),\Box)+((1^2),\emp) & \text{if }\delta\in\{1,-2\},\\
((2,1),\Box)+((3),\Box)+((1^2),\emp)+((2),\emp) & \text{otherwise}.
\end{array}\right.
\end{equation*}

\end{example}

\section{Representations of the general linear supergroup}\label{super}

Fix $m,n\geq0$ and consider the algebraic supergroup $GL(m|n)$ over $\K$.  In this paper we will only deal with finite dimensional representations of $GL(m|n)$, which can be identified with integrable representations of the corresponding Lie superalgebra $\gl(m|n)$.  We prefer to exploit this identification and work with $\gl(m|n)$ rather than  $GL(m|n)$, and we will do so for the rest of the paper.   We begin by fixing notation and conventions for representations of $\gl(m|n)$.

\subsection{The category $\Rep(\glmn)$}\label{superrepcategory}

Let $V = \psspace V 0 \oplus \psspace V 1$ denote a superspace over $\K$ with $ \dim_\K \psspace V 0 = m$, $ \dim_\K \psspace V 1 = n$.
Write $\glmn$ for the associated \newterm{general linear Lie superalgebra}, that is, for the Lie superalgebra of endomorphisms of the superspace $V$, considered as $(m+n)\times(m+n)$ matrices.
Then $V$ is called the \newterm{natural module} for $\glmn$.

Let $\Rep(\glmn)$ denote the category of finite-dimensional $\glmn$-modules.
Given $\glmn$-modules $U, U'$, the tensor product of superspaces $U \otimes U'$ is again a $\glmn$-module with action
\begin{equation}\label{glmntensor}
x \cdot (u \otimes u') = (x \cdot u) \otimes u' + (-1)^{\parity x \parity u} \ u \otimes (x \cdot u'),
\end{equation}
for $x\in \glmn, u \in U, u' \in U'$.
Thus, as per example \ref{svectismonoidal}, $\Rep(\glmn)$ is a monoidal category, where the one-dimensional purely-even module $\unit$ carries the trivial action.
Note that, in particular, we consider the tensor product of $\Rep(\glmn)$ to be strict.

If $U$ is a $\glmn$-module, then the dual superspace $U^*$ (cf. example \ref{svectistensor}) is again a $\glmn$-module with action
\begin{equation}\label{glmndual}
(x \cdot \phi) (u) = - (-1)^{\parity x \parity \phi} \phi (x \cdot u),
\end{equation}
for $x \in \glmn, \phi \in U^*, u \in U$. 
The maps $\ev_U$, $\coev_U$ are maps of $\glmn$-modules, and so $\Rep(\glmn)$ is a tensor category.

It is well known that the category $\Rep(\glmn)$ is not semisimple when $n > 0$ or $m>0$ (see e.g. \cite{Serganova1996}).

\subsection{Characters}
Let $\cartan \subset \glmn$ denote the subalgebra of diagonal matrices.
For any object $U$ of $\Rep(\glmn)$, let
$$ \ch U = \sum_{\mu \in \cartan*}  \dim_\K U_\mu \ e^\mu,$$
denote the \newterm{character} of $U$, where $ \dim_\K U_\mu$ denotes the dimension of the $\mu$-weight space of $U$ as a vector space over $\K$, and $e^\mu$ denotes the formal exponential.
Addition and multiplication of characters of objects in $\Rep(\glmn)$ are defined component-wise and by convolution, respectively.
Since $\cartan$ is purely even, it follows that from the definition of the biproduct and equation \eqref{glmntensor} that
$$ \ch (U \oplus U') = \ch U + \ch U', \qquad \ch (U \otimes U') = \ch U \cdot \ch {U'} $$
for any objects $U, U'$ of $\Rep(\glmn)$.

For any $r,s \geq 0$, write
$ T(r,s)=V^{\otimes r}\otimes(V^*)^{\otimes s} $, as per section \S\ref{mtp}.
Let $\{ \ep_i \}_{1 \leq i \leq m+n}$ denote the diagonal coordinate functions of $\glmn$, so that $\ep_i$ takes any matrix to its $(i,i)$-entry.
Then for any $r,s \geq 0$, all weights of $T(r,s)$ are integral linear combinations of the $\ep_i$ (hence also of any submodule of $T(r,s)$).  Let
$$ x_i = e^{\ep_i}, \quad 1 \leq i \leq m, \qquad y_j = e^{\ep_{m+j}}, \quad 1 \leq j \leq n.$$
Then if $U$ is a submodule of $T(r,s)$, then $\ch U = \ch U (x | y)$ is a Laurent polynomial in the variables $x = \{x_i\}$ and $y = \{y_j\}$.
Write $\bar{x}=\{x_i^{-1}\}_{1\leq i\leq m}$ and $\bar{y}=\{y^{-1}_i\}_{1\leq i\leq n}$.
Then it follows from equation \eqref{glmndual} that
$$\ch{(U^*)} (x|y) = \ch U (\bar{x} | \bar{y}).$$

\begin{example}\label{examplechars}
One has that $\ch \unit = 1$, while
$$ \ch V = x_1 + \cdots + x_m + y_1 + \cdots + y_n, \qquad \ch V^* = x_1^{-1} + \cdots + x_m^{-1} + y_1^{-1} + \cdots + y_n^{-1}.$$
\end{example}

\subsection{The functor $F_{m|n}:\uRep(GL_{m-n})\to\Rep(\mathfrak{gl}(m|n))$} 
Let $$F_{m|n}:\uRep(GL_{m-n})\to\Rep(\mathfrak{gl}(m|n))$$ denote the tensor functor which sends $\black\mapsto V$ defined by Proposition \ref{uni} .
For any $\zeta \in \K$, the central element $\zeta \cdot \id_{m+n} \in \glmn$ acts on the mixed tensor power $T(r,s)$ by the scalar $\zeta^{r-s}$, and so
$\Hom_{\Rep(\glmn)} (T(r,s), T(r',s') = 0$ unless $r+s'=r'+s$.
Moreover, it is well-known that the $\K$-algebra map $\K \Sigma_p \to \End_{\Rep(\glmn)}(V^{\otimes p})$ defined by the symmetric braiding is surjective, for any $p \geq 0$ (see \cite[Remark 4.15]{BR} or \cite{Serg}).
Thus, by Theorem \ref{fullness}, the functor $F_{m|n}$ is full and any indecomposable summand of a mixed tensor powers $T(r,s)$ is isomorphic to $F_{m|n} (L(\lambda))$ for some bipartition $\lambda$.
Generalizing our notation from \S\ref{Fd}, we set $W(\lambda)=F_{m|n}(L(\lambda))$.
The rest of this section is devoted to describing $W(\lambda)$ for arbitrary $\lambda$.  More precisely, we will give a formula for computing the character of $W(\lambda)$ in \S\ref{charW} and give a criterion for the vanishing of $W(\lambda)$ in \S\ref{vanishW}.  By Theorem \ref{fullness}, these results give a classification of indecomposable summands of mixed tensor space in $\Rep(\mathfrak{gl}(m|n))$.

\subsection{Composite supersymmetric Schur polynomials}
In \S\ref{charW} we give a formula for computing the character of $W(\lambda)$ in $\Rep(\mathfrak{gl}(m|n))$ (see Theorem \ref{chW}).  In the case  $\lambda^\white=\emp$ (resp. $\lambda^\black=\emp$), the character of $W(\lambda)$ was computed in \cite{BR} and \cite{Serg} and is called a \emph{covariant (resp. contravariant) supersymmetric Schur polynomial}.
Our formula for $\ch W(\lambda)$ for arbitrary $\lambda$  is in terms of the numbers $D_{\lambda,\mu}$ (\S\ref{lift}) and the so-called \emph{composite supersymmetric Schur polynomials\footnote{Also known as \emph{composite supersymmetric S-polynomials.}}} (see for instance \cite{MV06}).  There are many equivalent definitions of  composite supersymmetric Schur polynomials; we will use the determinantal formula found, for instance, in \cite[(38)]{MV06}.   In order to state this formula we need a few preliminary definitions.  As in \S\ref{superrepcategory}, we work with the variables $x=\{x_i\}_{1\leq i\leq m}$ and $y=\{y_i\}_{1\leq i\leq n}$, and write $\bar{x}=\{x_i^{-1}\}_{1\leq i\leq m}$ and $\bar{y}=\{y^{-1}_i\}_{1\leq i\leq n}$.   Now, we define the \emph{complete supersymmetric polynomials} by 
\begin{equation*}
\h_k=\h_k(x|y)=\sum_{i=0}^kh_{k-i}(x)e_i(y)
\end{equation*}
where $h_k(x)$ and $e_k(y)$ are the complete and elementary symmetric polynomials respectively  (see for instance \cite[\S I.2]{Mac}).  In particular, $\h_0=1$ and $\h_k=0$ whenever $k<0$.
Next, we write ${\bf\bar{h}}_k=\h_k(\bar{x}|\bar{y})$.  Now, given a bipartition $\lambda$, we define the \emph{composite supersymmetric Schur polynomial} $\s_\lambda=\s_\lambda(x|y)$ as the following determinate (compare with \cite[(38)]{MV06}\footnote{In the literature $\s_\lambda(x|y)$ is sometimes denoted $s_{\overline{\lambda^\white}; \lambda^\black}(x/y)$ or even $\{\overline{\lambda^\white}; \lambda^\black\}$.}):
\begin{equation*}\
\s_\lambda=\text{det}\left(
\begin{array}{cccccccc}
\bh_{\lambda^\white_{q}} & &  \vdots\\
\bh_{\lambda^\white_{q}-1} & \ddots & \bh_{\lambda^\white_{2}+1}  &  \vdots\\
\vdots & & \bh_{\lambda^\white_{2}}& \bh_{\lambda^\white_{1}+1} & \vdots \\
& &  \bh_{\lambda^\white_{2}-1} & \bh_{\lambda^\white_{1}} & \h_{\lambda^\black_{1}-1}& \vdots\\
&&\vdots& \bh_{\lambda^\white_{1}-1} & \h_{\lambda^\black_{1}} & \h_{\lambda^\black_2-1}\\
&&&\vdots & \h_{\lambda^\black_{1}+1} & \h_{\lambda^\black_2}& &\vdots\\
&&&& \vdots & \h_{\lambda^\black_2+1} & \ddots & \h_{\lambda_p^\black-1}\\
&&&&&\vdots & &\h_{\lambda_p^\black}
\end{array}
\right)
\end{equation*} where $p$ (resp. $q$) is any integer greater than or equal to $ l(\lambda^\black)$ (resp. $ l(\lambda^\white)$).

\begin{example}\label{schurex}  Fix $m=1$ and $n=2$.  Then 
\begin{align*}
\s_{((1),(2))} & =\text{det}\left(\begin{array}{cc}
\bh_2 & \h_0  \\
\bh_1  & \h_1
\end{array}\right)\\
& =\text{det}\left(\begin{array}{cc}
\frac{1}{x_1^2}+\frac{1}{x_1y_1}+\frac{1}{x_1y_2}+\frac{1}{y_1y_2} &  1\\
\frac{1}{x_1}+\frac{1}{y_1}+\frac{1}{y_2} &  x_1+y_1+y_2
\end{array}\right)\\
&=\tfrac{x_1}{y_1y_2}+\tfrac{y_1}{x_1y_2}+\tfrac{y_2}{x_1y_1}+\tfrac{y_1}{x_1^2}+\tfrac{y_2}{x_1^2}+2\tfrac{1}{x_1}+\tfrac{1}{y_1}+\tfrac{1}{y_2},\\
\text{whereas}\\
\s_{((1^2),(3))} & =\text{det}\left(\begin{array}{ccc}
\bh_3 & \h_0 & \h_{-1}\\
\bh_2 & \h_1 & \h_0\\
\bh_1 & \h_2 & \h_1
\end{array}\right)\\
& =\text{det}\left(\begin{array}{ccc}
\frac{1}{x_1^{3}}+\frac{1}{x_1^{2}y_1}+\frac{1}{x_1^{2}y_2}+\frac{1}{x_1y_1y_2} & 1 & 0\\
\frac{1}{x_1^{2}}+\frac{1}{x_1y_1}+\frac{1}{x_1y_2}+\frac{1}{y_1y_2} & \text{{\small $x_1+y_1+y_2$} }& 1\\
\frac{1}{x_1}+\frac{1}{y_1}+\frac{1}{y_2} & \text{{\small $x_1^2+x_1y_1+x_1y_2+y_1y_2$}} & \text{{\small $x_1+y_1+y_2$}}
\end{array}\right)\\
&=\tfrac{y_1^2}{x_1^{3}}+\tfrac{y_1y_2}{x_1^3}+\tfrac{y_2^2}{x_1^3}+\tfrac{y_1^2}{x_1^2y_2}+\tfrac{y_2^2}{x_1^2y_1}+2\tfrac{y_1}{x_1^2}+2\tfrac{y_2}{x_1^2}+\tfrac{1}{x_1}+\tfrac{y_1}{x_1y_2}+\tfrac{y_2}{x_1y_1}-\tfrac{x_1}{y_1y_2}.
\end{align*}
\end{example}

\subsection{The character of $W(\lambda)$}\label{charW}
The following proposition lists some of the well-known properties of $\s_\lambda$ which will be useful for this paper.

\begin{proposition}\label{sprops}
(1) $\ch W(\lambda)=\s_\lambda$ whenever $\lambda^\black=\emp$ or $\lambda^\white=\emp$.

(2) $\s_{(\lambda^\black,\emp)}\s_{(\emp,\mu^\white)}=\sum_{\nu}\sum_{\kappa\in\cat{P}}LR_{\kappa,\nu^\black}^{\lambda^\black}LR^{\mu^\white}_{\kappa,\nu^\white}\s_\nu$.
\end{proposition}

\begin{proof} (1) follows from the corresponding determinantal formula for the (non-composite) supersymmetric Schur polynomials (see \cite[(6)]{MV03} and references therein).   For (2) see \cite[(3.2)]{CK}.
\end{proof}

We are now ready to prove our formula for computing the character of $W(\lambda)$.

\begin{theorem}\label{chW} $\ch W(\lambda)=\sum_{\mu}D_{\lambda,\mu}(m-n)\s_\mu$ for any bipartition $\lambda$.
\end{theorem}

\begin{proof}  We induct on the size of $\lambda$.  The case $\lambda=\biemp$ is easy to check.  Now, by Proposition \ref{zdecomp} we can write 
\begin{equation}\label{a}
\lambda=(\lambda^\black,\emp)(\emp,\lambda^\white)-\sum_{\nu\atop{|\nu|<|\lambda|}}a_{\lambda,\nu}\nu\in R_{m-n}
\end{equation} for some $a_{\lambda,\nu}\in\Z$.  Write $D_{\lambda,\mu}=D_{\lambda,\mu}(m-n)$.  
Applying $\lift_{m-n}$ to (\ref{a}) and using Example \ref{liftemp} we have 
\begin{equation}\label{aa}
\sum_{\mu}D_{\lambda,\mu}\mu=(\lambda^\black,\emp)(\emp,\lambda^\white)-\sum_{\nu\atop{|\nu|<|\lambda|}}a_{\lambda,\nu}\sum_{\mu}D_{\nu,\mu}\mu\in R_t.
\end{equation}
Using  formula (\ref{leftright}) we compute the coefficient of $\mu$ in (\ref{aa}) to be
\begin{equation}\label{aaa}
D_{\lambda,\mu}=\sum_{\kappa\in\cat{P}}LR_{\kappa,\mu^\black}^{\lambda^\black}LR^{\lambda^\white}_{\kappa,\mu^\white}-\sum_{\nu\atop{|\nu|<|\lambda|}}a_{\lambda,\nu}D_{\nu,\mu}.
\end{equation}
On the other hand, since $F_{m|n}$ is a tensor functor, it follows from (\ref{a}) that 
\begin{equation*}
\ch W(\lambda)=\ch W((\lambda^\black,\emp))\ch W((\emp,\lambda^\white))-\sum_{\nu\atop{|\nu|<|\lambda|}}a_{\lambda,\nu}\ch W(\nu).
\end{equation*}
Hence, by Proposition \ref{sprops}(1) along with induction, we have 
\begin{equation*}\ch W(\lambda)=\s_{(\lambda^\black,\emp)}\s_{(\emp,\lambda^\white)}-\sum_{\nu\atop{|\nu|<|\lambda|}}a_{\lambda,\nu}\sum_{\mu}D_{\nu,\mu} \s_\mu.
\end{equation*}
Thus, by Proposition \ref{sprops}(2), 
\begin{equation*}\ch W(\lambda)=\sum_\mu\sum_{\kappa\in\cat{P}}LR_{\kappa,\mu^\black}^{\lambda^\black}LR^{\lambda^\white}_{\kappa,\mu^\white}\s_\mu-\sum_{\nu\atop{|\nu|<|\lambda|}}a_{\lambda,\nu}\sum_{\mu}D_{\nu,\mu} \s_\mu,
\end{equation*} and we are done by (\ref{aaa}).
\end{proof}

\begin{example} Let $m=1$, $n=2$ and consider the bipartition $((1^2),(3))$.   Since
$$\includegraphics{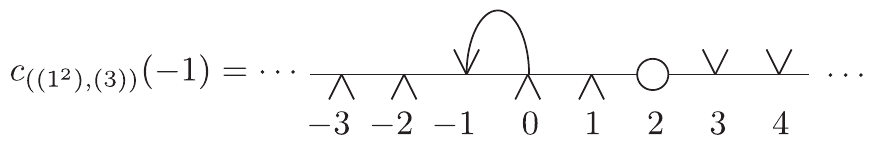}$$ and $$\includegraphics{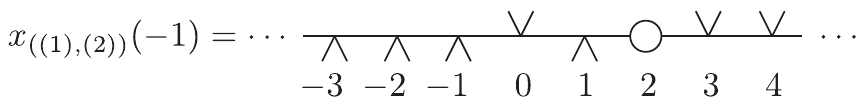}$$ by Corollary \ref{D} we have $\lift_{-1}(((1^2),(3)))=((1^2),(3))+((1),(2))$.  Hence, by Theorem \ref{chW} and Example \ref{schurex} we have 
\begin{align*}
\ch W(((1^2),(3))) & =\s_{((1^2),(3))}+\s_{((1),(2))}\\
& = \tfrac{y_1^2}{x_1^3}+\tfrac{y_1y_2}{x_1^3}+\tfrac{y_2^2}{x_1^3}+\tfrac{y_1^2}{x_1^2y_2}+\tfrac{y_2^2}{x_1^2y_1}+3\tfrac{y_1}{x_1^2}+3\tfrac{y_2}{x_1^2}+2\tfrac{y_1}{x_1y_2}+2\tfrac{y_2}{x_1y_1}\\
&+3\tfrac{1}{x_1}+\tfrac{1}{y_1}+\tfrac{1}{y_2}.
\end{align*}
\end{example}

For any Laurent polynomial $f = f(z_1, \ldots, z_k)$, write
$$ \edeg f = \text{max} \{\ \deg (f |_{z_i := z_i^{\ep_i} , \ 1 \leq i \leq k}) \ | \ \ep_i \in \{\pm 1\}, 1 \leq i \leq k \ \}$$
for the \newterm{extremal degree of $f$}, where $\deg$ defines the familiar total degree.
Then one has $\edeg{(fg)} \leq \edeg f + \edeg g$.
For any bipartition $\lambda$, we consider $\s_\lambda$ and $\ch W(\lambda)$ as Laurent polynomials in the variables $x = \{ x_i \}$ and $y = \{ y_j \}$.
Then we have the following corollary to Theorem \ref{chW}:

\begin{corollary}\label{charedeg}
Let $\lambda$ be a bipartition.  Then $\edeg \ch W(\lambda)  \leq | \lambda |$.
\end{corollary}
\begin{proof}
That $\edeg \s_\mu$ for any bipartition $\mu$ follows from e.g. \cite[(32)]{MV06}.
The claim then follows from Theorems \ref{liftprops} (2) and \ref{chW}.
\end{proof}

\subsection{Dimensions}
In this subsection we derive a formula for the $\K$-dimension of $W(\lambda)$ in $\Rep(\gl(m|n))$.  To do so, let $\de_k=\de_k(m|n)$ denote the result of setting $x_i=y_j=1$ for all $1\leq i\leq m$, $1\leq j\leq n$ in $\h_k$.  From the definition of $\h_k$ we have 
\begin{equation}\label{dk}
\de_k=\sum_{0\leq i\leq k}{m+k-i-1\choose m-1}{n\choose i}
\end{equation}  whenever $m>0$, and $\de_k={n\choose k}$ when $m=0$. 
Now, given a bipartition $\lambda$ we let $\de_\lambda=\de_\lambda(m|n)$ denote the result of setting $x_i=y_j=1$ for all $1\leq i\leq m$, $1\leq j\leq n$ in $\s_\lambda$, so that
\begin{equation}\label{dmat}
\de_\lambda=\text{det}\left(
\begin{array}{cccccccc}
\de_{\lambda^\white_{q}} & &  \vdots\\
\de_{\lambda^\white_{q}-1} & \ddots & \de_{\lambda^\white_{2}+1}  &  \vdots\\
\vdots & & \de_{\lambda^\white_{2}}& \de_{\lambda^\white_{1}+1} & \vdots \\
& &  \de_{\lambda^\white_{2}-1} & \de_{\lambda^\white_{1}} & \de_{\lambda^\black_{1}-1}& \vdots\\
&&\vdots& \de_{\lambda^\white_{1}-1} & \de_{\lambda^\black_{1}} & \de_{\lambda^\black_2-1}\\
&&&\vdots & \de_{\lambda^\black_{1}+1} & \de_{\lambda^\black_2}& &\vdots\\
&&&& \vdots & \de_{\lambda^\black_2+1} & \ddots & \de_{\lambda_p^\black-1}\\
&&&&&\vdots & &\de_{\lambda_p^\black}
\end{array}
\right)
\end{equation} where $p$ (resp. $q$) is any integer greater than or equal to $ l(\lambda^\black)$ (resp. $ l(\lambda^\white)$).  It follows immediately from Theorem \ref{chW} that 
\begin{equation}\label{dimW}
\dim_\K W(\lambda)=\sum_\mu D_{\lambda,\mu}\de_\mu
\end{equation} 
in $\Rep(\gl(m|n))$.  The following technical lemma concerning $\de_k$ will be useful later:

\begin{lemma}\label{delem} Assume $m>0$ and let $g_k(u)=\prod_{1\leq j<m}(k-u+j)$  (an $(m-1)$-degree polynomial in the variable $u$). Then 
$\de_{k}+(-1)^{m-1}\de_{n-m-k}=\frac{1}{(m-1)!}\sum_{0\leq i\leq n} g_k(i){n\choose i}$ for all $k\in\Z$.
\end{lemma}

\begin{proof} Since ${m+k-i-1\choose m-1}=\frac{1}{(m-1)!}\prod_{1\leq j< m}(k-i+j)$, by (\ref{dk}) we have 
\begin{equation*}
\de_k=\frac{1}{(m-1)!}\sum_{0\leq i\leq k}\prod_{1\leq j< m}(k-i+j){n\choose i}=\frac{1}{(m-1)!}\sum_{0\leq i\leq k}g_k(i){n\choose i}.
\end{equation*}
Moreover, \begin{align*}
\de_{n-m-k}&=\frac{1}{(m-1)!}\sum_{0\leq i\leq n-m-k}\prod_{1\leq j< m}(n-m-k-i+j){n\choose i}\\
& = \frac{1}{(m-1)!}\sum_{k+m\leq i\leq n}\prod_{1\leq j< m}(i-m-k+j){n\choose i}\\
& = \frac{1}{(m-1)!}\sum_{k+m\leq i\leq n}\prod_{1\leq j< m}(i-k-j){n\choose i}\\
&= \frac{(-1)^{m-1}}{(m-1)!}\sum_{k+m\leq i\leq n}g_k(i){n\choose i}.
\end{align*}
  Since $g_k(i)=0$ whenever $k< i< k+m$, the result follows.
\end{proof}

\subsection{A criterion for the vanishing of $W(\lambda)$}\label{vanishW}  We now concern ourselves with the question of when $W(\lambda)$ is zero.  In the case $\lambda^\white=\emp$ or $\lambda^\black=\emp$, the answer is well known and due to \cite{BR} and \cite{Serg}.  To state their result, we say a partition $\alpha$ is \emph{$(m|n)$-hook} if $\alpha_{m+1}\leq n$.  Pictorially, $\alpha$ is $(m|n)$-hook if and only if its Young diagram fits into an $m$-high, $n$-wide hook:
$$\includegraphics{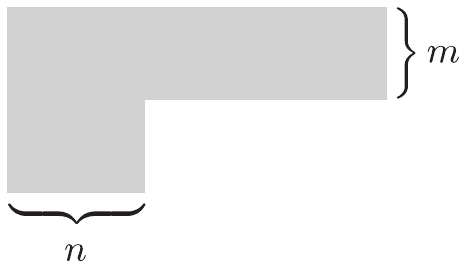}$$
We now record the result on the vanishing of $W(\lambda)$ found in \cite{BR} and \cite{Serg}:

\begin{theorem}\label{hookthm} Suppose $\lambda$ is a bipartition with $\lambda^\white=\emp$ (resp. $\lambda^\black=\emp$).  Then $W(\lambda)\not=0$ in $\Rep(\gl(m|n))$ if and only if $\lambda^\black$ (resp. $\lambda^\white$) is $(m|n)$-hook.
\end{theorem}

The goal of this subsection is to prove a theorem analogous to Theorem \ref{hookthm} which holds for arbitrary bipartitions $\lambda$ (see Theorem \ref{crossthm}).  First, we generalize the notion of $(m|n)$-hook as follows:
We call a bipartition $\lambda$ \emph{$(m|n)$-cross} if there exists $k$ with $0\leq k\leq m$ such that $\lambda^\black_{k+1}+\lambda^\white_{m-k+1}\leq n$.  Pictorially, $\lambda$ is $(m|n)$-cross if and only if its diagram (see \S\ref{bipart}) fits into an $m$-high, $n$-wide cross:
$$\includegraphics{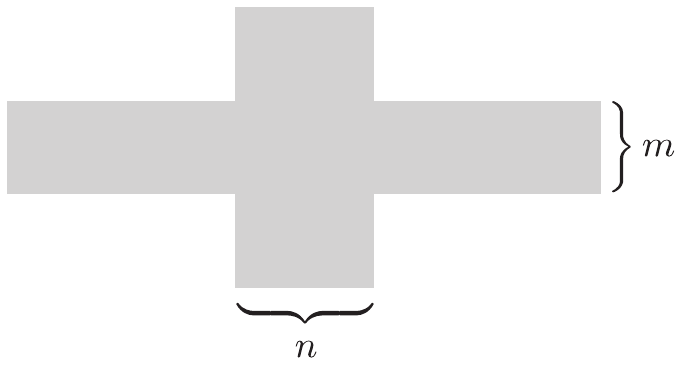}$$  
Indeed, if $\lambda^\black_{k+1}+\lambda^\white_{m-k+1}\leq n$ then $\lambda$ can by covered by an $m$-high, $n$-wide cross if the cross is placed so that the horizontal strip covers exactly the first $k$ rows of $\lambda^\black$ and the vertical strip covers exactly the first $\lambda_{k+1}^\black$ columns of $\lambda^\black$:
$$\includegraphics{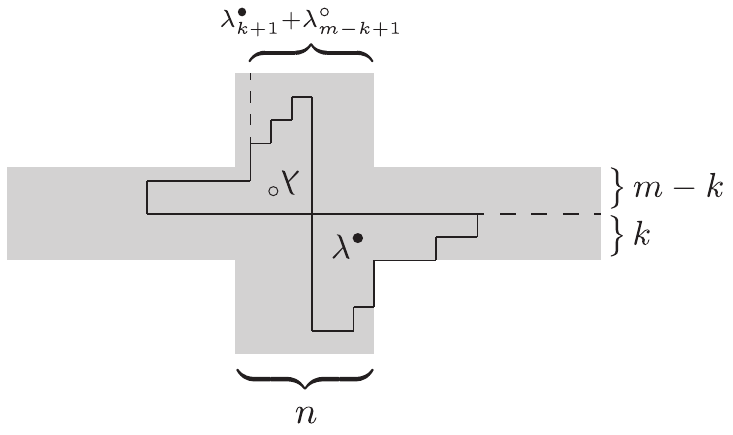}$$
For example, $((3,1),(2))$ is both $(2|1)$-cross and $(1|3)$-cross as shown below:
$$\includegraphics{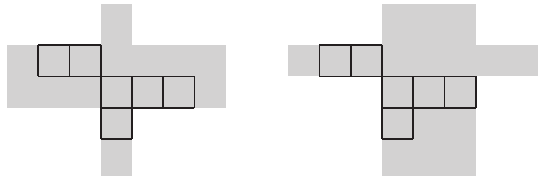}$$ However, $((3,1),(2))$ is not $(1|2)$-cross since it is impossible to cover its diagram with a $1$-high, $2$-wide cross:
$$\includegraphics{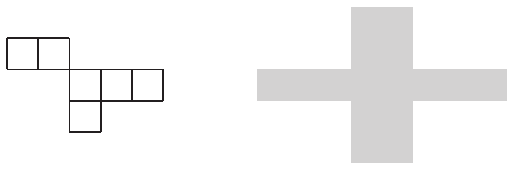}$$
We will eventually show $W(\lambda)\not=0$ in $\Rep(\gl(m|n))$ if and only if $\lambda$ is $(m|n)$-cross (see Theorem \ref{crossthm}).  Notice that for bipartitions $\lambda$ with $\lambda^\white=\emp$ (resp. $\lambda^\black=\emp$) we have that $\lambda$ is $(m|n)$-cross if and only if $\lambda^\black$ (resp. $\lambda^\white$) is $(m|n)$-hook.  Hence our criterion for the vanishing of $W(\lambda)$ will generalize Theorem \ref{hookthm}.  
The following lemma gets us half way towards proving our criterion for the vanishing of $W(\lambda)$.

\begin{lemma}\label{crosslem} If $\lambda$ is $(m|n)$-cross, then $W(\lambda)\not=0$ in $\Rep(\gl(m|n))$.
\end{lemma}

\begin{proof} Assume $k$ is such that $0\leq k\leq m$ and $\lambda^\black_{k+1}+\lambda^\white_{m-k+1}\leq n$.  For convenience, write $l=\lambda^\black_{k+1}$ and $l'=\lambda^\white_{m-k+1}$.  Now set  $\eta^\black_i=(\lambda^\black)^{\bf t}_i-k$ for all $1\leq i\leq l$ and $\eta^\white=(\lambda^\white)^{\bf t}_i-m+k$ for all $1\leq i\leq l'$.  Now, $\ch W((\lambda^\black,\emp))$ (resp. $\ch W((\emp,\lambda^\white))$) is a covariant (resp. contravariant) supersymmetric Schur polynomials, which can be computed using so-called \emph{supertableaux} (see \cite{BR}).  From the supertableaux definition of supersymmetric Schur polynomials it is apparent that the coefficient of $x_1^{\lambda^\black_1}\cdots x^{\lambda^\black_k}_ky_1^{\eta^\black_1}\cdots y^{\eta^\black_l}_l$ (resp. $x_{k+1}^{-\lambda^\white_1}\cdots x^{-\lambda^\white_{m-k}}_my_{n-l'+1}^{-\eta^\white_1}\cdots y^{-\eta^\white_{l'}}_n$) in 
$\ch W((\lambda^\black,\emp))$ (resp. $\ch W((\emp,\lambda^\white))$) has nonzero coefficient.  Hence the coefficient of \begin{equation}\label{goodterm}
\frac{x_1^{\lambda^\black_1}\cdots x^{\lambda^\black_k}_ky_1^{\eta^\black_1}\cdots y^{\eta^\black_l}_l}{x_{k+1}^{\lambda^\white_1}\cdots x^{\lambda^\white_{m-k}}_my_{n-l'+1}^{\eta^\white_1}\cdots y^{\eta^\white_{l'}}_n}
\end{equation}
in  $C := \ch(W((\lambda^\black,\emp))\otimes W((\emp,\lambda^\white)))$ is nonzero.  
Thus
 \begin{equation*}
 \edeg C \geq \sum_{1\leq i\leq k}\lambda^\black_i+\sum_{1\leq i\leq l}\eta^\black_i+\sum_{1\leq i\leq m-k}\lambda^\white_i+\sum_{1\leq i\leq l'}\eta^\white_i=|\lambda|,
 \end{equation*} 
 since $l + l' \leq n$.
On the other hand, by  Proposition \ref{zdecomp} we have 
\begin{equation}\label{chdec}
C=\ch W(\lambda)+\ch W(\mu^{(1)})+\cdots+\ch W(\mu^{(s)})
\end{equation}
 for some bipartitions $\mu^{(1)},\ldots,\mu^{(s)}$ with $|\mu^{(i)}|<|\lambda|$ for all $i=1,\ldots,s$.
Thus the coefficient of \eqref{goodterm} is non-zero in $\ch W(\lambda)$ by Corollary \ref{charedeg}. 
\end{proof}

The rest of this subsection is devoted to proving the converse of  Lemma \ref{crosslem}.  We begin by focusing our attention to certain class of non-$(m|n)$-cross bipartitions.  More precisely, we call a bipartition $\lambda$ \emph{almost $(m|n)$-cross} if $ l(\lambda^\black), l(\lambda^\white)\leq m+1$ and $\lambda^\black_{k+1}+\lambda^\white_{m-k+1}=n+1$ whenever $0\leq k\leq m$.  Equivalently, $\lambda$ is almost $(m|n)$-cross if $\lambda$ is not $(m|n)$-cross, but any bipartition obtained from $\lambda$ by removing a box is $(m|n)$-cross.  Pictorially, $\lambda$ is almost $(m|n)$-cross if and only if the Young diagrams of $\lambda^\black$ and $\lambda^\white$ can be glued together to form a $(m+1)\times(n+1)$ rectangle:
$$\includegraphics{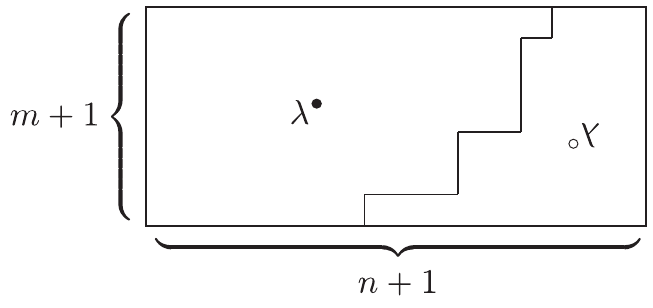}$$ 
The next proposition states that almost $(m|n)$-cross bipartitions lift trivially from $R_{m-n}$, hence characters and dimensions of the corresponding representations of $\gl(m|n)$ can be easily computed (see Theorem \ref{chW} and (\ref{dimW})).

\begin{proposition}\label{allift}
If $\lambda$ is almost $(m|n)$-cross, then $\lift_{m-n}(\lambda)=\lambda$.
\end{proposition}

\begin{proof}
Assume $\lambda$ is almost $(m|n)$-cross.  Then 
\begin{align*}
I_\up(\lambda)&=\{\lambda_1^\black,\ldots,\lambda^\black_{m+1}-m,-m-1,-m-2,\ldots\},\\
I_\down(\lambda,m-n)&=\{1-m+n-\lambda^\white_1,\ldots,1+n-\lambda^\white_{m+1},n+2,n+3,\ldots\}\\
&= \{\lambda^\black_{m+1}-m,\ldots,\lambda_1^\black,n+2,n+3,\ldots\}.
\end{align*}
Hence, $$\includegraphics{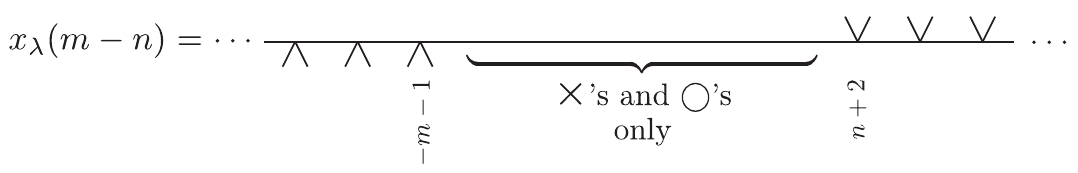}$$  In particular, $c_\lambda(m-n)$ has no caps and we are done by Corollary \ref{D}.
\end{proof}

\begin{corollary}\label{aldim}
If $\lambda$ is almost $(m|n)$-cross, then $\dim_\K W(\lambda)=\de_\lambda$ in $\Rep(\gl(m|n))$.
\end{corollary}

\begin{proof}
This follows immediately from Proposition \ref{allift} and (\ref{dimW}).
\end{proof}

We are now ready to prove the following lemma, which is the crucial step towards proving the converse of Lemma \ref{crosslem}.

\begin{lemma}\label{alzero}
If $\lambda$ is almost $(m|n)$-cross, then $W(\lambda)=0$ in $\Rep(\gl(m|n))$.
\end{lemma}

\begin{proof} By Corollary \ref{aldim} it suffices to show $\de_\lambda=0$ whenever $\lambda$ is almost $(m|n)$-cross.  If $m=0$, then almost $(m|n)$-cross bipartitions are exactly ones of the form $\lambda=((l), (n+1-l))$ for some $0\leq l\leq n+1$.  In this case  
\begin{equation*}
\de_\lambda=\det\left(\begin{array}{cc}
\de_{n+1-l} & \de_{l-1}\\
\de_{n-l} & \de_{l}
\end{array}\right)
=\det\left(\begin{array}{cc}
{n\choose n+1-l} & {n\choose l-1}\\[2pt]
{n\choose n-l} & {n\choose l}
\end{array}\right)=0.
\end{equation*}  Hence we may assume $m>0$.  Assume $\lambda$ is almost $(m|n)$-cross so $ l(\lambda^\black), l(\lambda^\white)\leq m+1$ and $\lambda^\black_i=n+1-\lambda^\white_{m-i+2}$ for all $1\leq i\leq m+1$.  Then, by (\ref{dmat}), $\de_\lambda$ is the determinant of the $(2m+2)\times(2m+2)$ block matrix $(A~B)$ where 
\begin{equation*}\label{Amat}
A=\left(\begin{array}{cccc}
\de_{\lambda^\white_{m+1}} & \de_{\lambda^\white_{m}-1}   & \cdots& \de_{\lambda^\white_1+m} \\
\de_{\lambda^\white_{m+1}-1} & \de_{\lambda^\white_{m}} & \cdots & \de_{\lambda^\white_1+m-1} \\
\vdots & \vdots &  & \vdots\\
\de_{\lambda^\white_{m+1}-2m-1} & \de_{\lambda^\white_{m}-2m} & \cdots & \de_{\lambda^\white_1-m-1} 
\end{array}\right),
\end{equation*}
\begin{equation*}\label{Bmat}
B=\left(\begin{array}{cccc}
\de_{n-m-\lambda^\white_{m+1}} & \de_{n-m+1-\lambda^\white_{m}}   & \cdots& \de_{n-2m-\lambda^\white_1} \\
\de_{n-m+1-\lambda^\white_{m+1}} & \de_{n-m-\lambda^\white_{m}} & \cdots & \de_{n-2m+1-\lambda^\white_1} \\
\vdots & \vdots & & \vdots\\
\de_{n+m+1-\lambda^\white_{m+1}} & \de_{n+m-\lambda^\white_{m}} & \cdots & \de_{n+1-\lambda^\white_1} 
\end{array}\right).
\end{equation*}  Hence, if we let $C$ denote the $(2m+2)\times(m+1)$ matrix whose $j$th column is given by $\text{Col}_j(C)=\text{Col}_j(A)+(-1)^{m-1}\text{Col}_j(B)$, then by Lemma \ref{delem} we have 
\begin{equation*}\label{Cmat}
C=\tfrac{1}{(m-1)!}\sum_{0\leq i\leq n}{n\choose i}\left(\begin{array}{cccc}
g_{\lambda^\white_{m+1}}(i) & g_{\lambda^\white_{m}+1}(i)   & \cdots& g_{\lambda^\white_1+m}(i) \\
g_{\lambda^\white_{m+1}-1}(i) & g_{\lambda^\white_{m}}(i) & \cdots & g_{\lambda^\white_1+m-1}(i) \\
\vdots & \vdots &  & \vdots\\
g_{\lambda^\white_{m+1}-2m-1}(i) & g_{\lambda^\white_{m}-2m}(i) & \cdots & g_{\lambda^\white_1-m-1}(i) 
\end{array}\right)
\end{equation*} 
where $g_k(u)$ is the polynomial defined in Lemma \ref{delem}.  To show $\de_\lambda=0$ we will verify that the columns of $C$ (and hence of $(A~B)$) are linearly dependent.  To do so, notice $g_{k-l}(u)=g_k(u+l)$ for any $k,l\in\Z$.  Hence, the $j$th column of $C$ is 
\begin{equation}\label{Ccol}
\text{Col}_j(C)=\tfrac{1}{(m-1)!}\sum_{0\leq i\leq n}{n\choose i}\left(\begin{array}{c}
g_{\lambda^\white_{m+2-j}+j-1}(i) \\
g_{\lambda^\white_{m+2-j}+j-1}(i+1) \\
\vdots \\
g_{\lambda^\white_{m+2-j}+j-1}(i+2m+1) 
\end{array}\right).
\end{equation} 
The $m+1$ polynomials $g_{\lambda^\white_{m+2-j}+j-1}(u)$ for $1\leq j\leq m+1$ are each of degree $m-1$ in the variable $u$.  Hence, there exist $a_1,\ldots,a_{m+1}\in\K$ (not all zero) with $\sum_{j=1}^{m+1}a_jg_{\lambda^\white_{m+2-j}+j-1}(u)=0$.  It follows from (\ref{Ccol}) that  $\sum_{j=1}^{m+1}a_j\text{Col}_j(C)=0$.
\end{proof}

Finally, we are in position to prove our criterion for the vanishing of $W(\lambda)$.

\begin{theorem}\label{crossthm}  
$W(\lambda)\not=0$ in $\Rep(\gl(m|n))$ if and only if $\lambda$ is $(m|n)$-cross.
\end{theorem}

\begin{proof}  One direction is Lemma \ref{crosslem}.  To prove the other, assume $\lambda$ is not $(m|n)$-cross.  We will proceed by inducting on $|\lambda|=(r,s)$.  The non-$(m|n)$-cross bipartitions of minimal size are exactly the almost $(m|n)$-cross bipartitions, hence the base case of our induction is Lemma \ref{alzero}.  If $\lambda$ is not almost $(m|n)$-cross, then there exists a non-$(m|n)$-cross bipartition $\mu$ which is obtained from $\lambda$ by removing a box.  If $\mu\vdash(r-1,s)$ then by (\ref{leftbox}) we have $\mu(\Box,\emp)=\lambda+\mu^{(1)}+\cdots+\mu^{(k)}\in R_t$ for some bipartitions $\mu^{(1)},\ldots,\mu^{(k)}$ with $|\mu^{(i)}|<|\lambda|$ for all $1\leq i\leq k$.  Hence, by Theorem \ref{liftprops}(2), $\mu(\Box,\emp)=\lambda+\nu^{(1)}+\cdots+\nu^{(l)}\in R_{m-n}$ for some bipartitions $\nu^{(1)},\ldots,\nu^{(l)}$, which implies 
$W(\mu)\otimes W((\Box,\emp))=W(\lambda)\oplus W(\nu^{(1)})\oplus\cdots\oplus W(\nu^{(l)})$
in $\Rep(\gl(m|n))$.  By induction, we know $W(\mu)=0$, hence $W(\lambda)=0$ too.  For the case $\mu\vdash(r,s-1)$, one argues similarly using (\ref{rightbox}).
\end{proof}

\begin{remark} As mentioned above, if $\lambda^\white=\emp$ (resp. $\lambda^\black=\emp$), then $\lambda$ being $(m|n)$-cross is equivalent to $\lambda^\black$ (resp. $\lambda^\white$) being $(m|n)$-hook.  Hence Theorem \ref{crossthm} is a generalization of Theorem \ref{hookthm}.  

On the other hand, suppose $n=0$ and set $m=d$ so that $F_{m|n}$ is identified with $F_d$ (see \S\ref{functorFd}).   In this case $\lambda$ is $(m|n)$-cross if and only if $ l(\lambda)\leq d$.  Hence, Theorem \ref{crossthm} also generalizes the vanishing criterion in Theorem \ref{Fdim}.
\end{remark}

\subsection{Decomposing $W(\lambda)\otimes W(\mu)$}  Since $F_{m|n}$ is a tensor functor, Theorem \ref{crossthm} along with our method of decomposing arbitrary tensor products in $\uRep(GL_\delta)$ from \S\ref{arbtens} immediately allow us to decompose tensor products of  the form $W(\lambda)\otimes W(\mu)$. 

\begin{example}  From Example \ref{dectensex} with $\delta=0$ we have the following decomposition in $\Rep(\gl(m|m))$:
$$W((2),\emp)\otimes W(\bibox)\cong W(((2,1),\Box))\oplus W(((3),\Box))\oplus W(((1^2),\emp))\oplus W(((2),\emp))^{\oplus 2}.$$  If $m>1$, all the bipartitions in the decomposition above are $(m|m)$-cross.  Hence by Theorem \ref{crossthm} the summands in the decomposition above are all nonzero.  On the other hand, $((2,1),\Box)$ is the only bipartition above which is not $(1|1)$-cross.  Hence by Theorem \ref{crossthm}, we have the following decomposition into nonzero indecomposables in $\Rep(\gl(1|1))$:
$$W((2),\emp)\otimes W(\bibox)\cong W(((3),\Box))\oplus W(((1^2),\emp))\oplus W(((2),\emp))^{\oplus 2}.$$

\end{example}

\bibliographystyle{alphanum}	
	\bibliography{references}

\newcommand{\etalchar}[1]{$^{#1}$}
\def\cprime{$'$} \def\cprime{$'$}
\begin{thebibliography}{CDDM}

\bibitem[AF]{AndersonFuller}
F.W. Anderson and K.R. Fuller.
\newblock {\em Rings and categories of modules}.
\newblock Graduate texts in mathematics. Springer-Verlag, 1992.

\bibitem[BCH{\etalchar{+}}]{MR1280591}
Georgia Benkart, Manish Chakrabarti, Thomas Halverson, Robert Leduc, Chanyoung
  Lee, and Jeffrey Stroomer.
\newblock Tensor product representations of general linear groups and their
  connections with {B}rauer algebras.
\newblock {\em J. Algebra}, 166(3):529--567, 1994.

\bibitem[Ben]{Benson}
D.~J. Benson.
\newblock {\em Representations and cohomology. {I}}, volume~30 of {\em
  Cambridge Studies in Advanced Mathematics}.
\newblock Cambridge University Press, Cambridge, 1991.
\newblock Basic representation theory of finite groups and associative
  algebras.

\bibitem[BR]{BR}
A.~Berele and A.~Regev.
\newblock Hook {Y}oung diagrams with applications to combinatorics and to
  representations of {L}ie superalgebras.
\newblock {\em Adv. in Math.}, 64(2):118--175, 1987.

\bibitem[BS1]{BS}
J.~{Brundan} and C.~{Stroppel}.
\newblock {Gradings on walled Brauer algebras and Khovanov's arc algebra}.
\newblock {\em ArXiv e-prints}, July 2011.

\bibitem[BS2]{BS2}
Jonathan Brundan and Catharina Stroppel.
\newblock Highest weight categories arising from {K}hovanov's diagram algebra
  {II}: {K}oszulity.
\newblock {\em Transform. Groups}, 15:1--45, 2010.

\bibitem[BS3]{BS1}
Jonathan Brundan and Catharina Stroppel.
\newblock Highest weight categories arising from {K}hovanov's diagram algebra
  {I}: cellularity.
\newblock {\em to appear in Mosc. Math. J.}, 2011.

\bibitem[BS4]{BS3}
Jonathan Brundan and Catharina Stroppel.
\newblock Highest weight categories arising from {K}hovanov's diagram algebra
  {III}: category {$\mathcal O$}.
\newblock {\em Represent. Theory}, 15:170--243, 2011.

\bibitem[BS5]{BS4}
Jonathan Brundan and Catharina Stroppel.
\newblock Highest weight categories arising from {K}hovanov's diagram algebra
  {IV}: the general linear supergroup.
\newblock {\em to appear in J. Eur. Math. Soc.}, 2011.

\bibitem[CD]{CD}
A.~{Cox} and M.~{De Visscher}.
\newblock {Diagrammatic Kazhdan-Lusztig theory for the (walled) Brauer
  algebra}.
\newblock {\em ArXiv e-prints}, September 2010.

\bibitem[CDDM]{CDDM}
Anton Cox, Maud {De Visscher}, Stephen Doty, and Paul Martin.
\newblock On the blocks of the walled {B}rauer algebra.
\newblock {\em J. Algebra}, 320(1):169--212, 2008.

\bibitem[CK]{CK}
C~J Cummins and R~C King.
\newblock Composite young diagrams, supercharacters of u(m/n) and modification
  rules.
\newblock {\em Journal of Physics A: Mathematical and General}, 20(11):3121,
  1987.

\bibitem[CO]{CO1}
Jonathan Comes and Victor Ostrik.
\newblock On blocks of {D}eligne's category
  {$\underline{\operatorname{Re}}\!\operatorname{p}(S_t)$}.
\newblock {\em Advances in Mathematics}, 226(2):1331 -- 1377, 2011.

\bibitem[Del1]{Del07}
P.~Deligne.
\newblock La cat\'egorie des repr\'esentations du groupe sym\'etrique {$S\sb
  t$}, lorsque {$t$} n'est pas un entier naturel.
\newblock In {\em Algebraic groups and homogeneous spaces}, Tata Inst. Fund.
  Res. Stud. Math., pages 209--273. Tata Inst. Fund. Res., Mumbai, 2007.

\bibitem[Del2]{Del96}
Pierre Deligne.
\newblock La s\'erie exceptionnelle de groupes de {L}ie.
\newblock {\em C. R. Acad. Sci. Paris S\'er. I Math.}, 322(4):321--326, 1996.

\bibitem[FH]{FH}
William Fulton and Joe Harris.
\newblock {\em Representation theory}, volume 129 of {\em Graduate Texts in
  Mathematics}.
\newblock Springer-Verlag, New York, 1991.
\newblock A first course, Readings in Mathematics.

\bibitem[GL]{GL}
J.~J. Graham and G.~I. Lehrer.
\newblock Cellular algebras.
\newblock {\em Invent. Math.}, 123(1):1--34, 1996.

\bibitem[Koi]{Koike}
Kazuhiko Koike.
\newblock On the decomposition of tensor products of the representations of the
  classical groups: by means of the universal characters.
\newblock {\em Adv. Math.}, 74(1):57--86, 1989.

\bibitem[Mac1]{Mac}
I.~G. Macdonald.
\newblock {\em Symmetric functions and {H}all polynomials}.
\newblock Oxford Mathematical Monographs. The Clarendon Press Oxford University
  Press, New York, second edition, 1995.
\newblock With contributions by A. Zelevinsky, Oxford Science Publications.

\bibitem[Mac2]{MacLane}
Saunders MacLane.
\newblock {\em {Categories for the Working Mathematician (Graduate Texts in
  Mathematics)}}.
\newblock Springer, 2nd edition, September 1998.

\bibitem[MV1]{MV03}
E.~M. Moens and J.~{Van der Jeugt}.
\newblock A determinantal formula for supersymmetric {S}chur polynomials.
\newblock {\em J. Algebraic Combin.}, 17(3):283--307, 2003.

\bibitem[MV2]{MV06}
E.~M. Moens and J.~{Van der Jeugt}.
\newblock Composite supersymmetric s-functions and characters of
  {${\mathfrak{gl}}(m\vert n)$} representations.
\newblock In {\em Lie theory and its applications in physics {VI}}. World Sci.
  Publ., River Edge, NJ, 2006.

\bibitem[{Sel}]{Selinger}
P.~{Selinger}.
\newblock {A survey of graphical languages for monoidal categories}.
\newblock {\em ArXiv e-prints}, August 2009.

\bibitem[Ser1]{Serganova1996}
Vera Serganova.
\newblock Kazhdan-{L}usztig polynomials and character formula for the {L}ie
  superalgebra {$\mathfrak{ gl}(n|m)$}.
\newblock {\em Selecta Math. (N.S.)}, 2(4):607--651, 1996.

\bibitem[Ser2]{Serg}
A.~N. Sergeev.
\newblock Representations of the {L}ie superalgebras {$\mathfrak{ gl}(n,\,m)$}
  and {$Q(n)$} in a space of tensors.
\newblock {\em Funktsional. Anal. i Prilozhen.}, 18(1):80--81, 1984.

\bibitem[Ste]{Stembridge1987}
John~R. Stembridge.
\newblock Rational tableaux and the tensor algebra of gln.
\newblock {\em Journal of Combinatorial Theory, Series A}, 46(1):79 -- 120,
  1987.

\bibitem[Tur]{MR1024455}
V.~G. Turaev.
\newblock Operator invariants of tangles, and {$R$}-matrices.
\newblock {\em Izv. Akad. Nauk SSSR Ser. Mat.}, 53(5):1073--1107, 1135, 1989.

\bibitem[Wey]{WeylClassical}
H.~Weyl.
\newblock {\em The classical groups: their invariants and representations}.
\newblock Princeton landmarks in mathematics and physics. Princeton University
  Press, 1997.

\end{thebibliography}

\end{document}